\documentclass[reqno]{amsart}
\usepackage{etex}
\usepackage{amsfonts}
\usepackage{amscd}
\usepackage{amsbsy}
\usepackage{amsxtra}
\usepackage{amssymb}
\usepackage{epsfig}
\usepackage{epic,eepic}
\usepackage{graphicx}
\usepackage[all]{xy}
\usepackage{xypic}
\usepackage{psfrag}
\usepackage{amsmath}
\usepackage{amsthm}
\usepackage{setspace}
\usepackage{hyperref}
\usepackage{url}
\usepackage{here}
\usepackage{todonotes}
\newlength{\halfbls}
\usepackage{tikz}
\usepackage{tikz}
\usetikzlibrary{calc}
\usetikzlibrary{decorations.pathreplacing,decorations.markings}
\usetikzlibrary{arrows}
\usetikzlibrary{decorations.markings, decorations.shapes}
\usetikzlibrary{arrows}
\usetikzlibrary{positioning}
\usetikzlibrary{intersections}
\usetikzlibrary{decorations.pathmorphing}

\usepackage{thmtools}

\DeclareRobustCommand{\SkipTocEntry}[9]{}

\DeclareMathAlphabet\mathbfcal{OMS}{cmsy}{b}{n}
\usepackage{xparse}

\theoremstyle{plain}
\newtheorem{Defi}{Definition}[section]  
   \newtheorem{Prop}[Defi]{Proposition}
\newtheorem{Lemma}[Defi]{Lemma}    \newtheorem{Cor}[Defi]{Corollary}
\newtheorem{Thm}[Defi]{Theorem}  \newtheorem{asu}[Defi]{Assumption}  
\newtheorem{KeyLe}[Defi]{Key Lemma}  
  
\theoremstyle{remark}
\newtheorem{Ex}[Defi]{Example}

\newcommand{\CC}{\mathbb{C}}

\newcommand{\NN}{\mathbb{N}}
 
\newcommand{\PP}{\mathbb{P}} 
\newcommand{\QQ}{\mathbb{Q}} 
\newcommand{\RR}{\mathbb{R}}

\newcommand{\ZZ}{\mathbb{Z}}

\newcommand{\frakm}{\mathfrak{m}}

\def\fd{\mathfrak d}

\newcommand{\cLL}{{\mathcal L}} \newcommand{\cAA}{{\mathcal A}}   
  
\newcommand{\cDD}{{\mathcal D}} \newcommand{\cXX}{{\mathcal X}}  
\newcommand{\cCC}{{\mathcal C}} \newcommand{\cSS}{{\mathcal S}}
 \newcommand{\cGG}{{\mathcal G}}
\newcommand{\cHH}{{\mathcal H}} 
\newcommand{\cMM}{{\mathcal M}} \newcommand{\cOO}{{\mathcal O}}
\newcommand{\cQQ}{{\mathcal Q}}

  \def\M{\cMM}  \def\OmM{\Omega\M}

\newcommand{\bfcH}{{\mathbfcal H}}

\newcommand{\bff}{{\bf f}}
\newcommand{\bfg}{{\bf g}}
\newcommand{\bfh}{{\bf h}}

\newcommand{\bfn}{{\bf n}}
\newcommand{\bfp}{{\bf p}}

\newcommand{\bfP}{{\bf P}}
\newcommand{\bfR}{{\bf R}}

\newcommand{\bfu}{{\bf u}}
\newcommand{\bfz}{{\bf z}}

\renewcommand\P{{\bf P}}

\newcommand{\bfmu}{{\boldsymbol{\mu}}}
\newcommand{\bfLA}{{\boldsymbol{\Lambda}}}

\newcommand{\ual}{{\boldsymbol{\alpha}}}

\newcommand{\Aut}{{\rm Aut}}

\newcommand{\Cliff}{{\rm Cliff}}

\newcommand{\cyl}{{\rm cyl}} 
\newcommand{\phy}{{\rm phy}}
\newcommand{\Mp}{{\cMM{\rm -p}}}
 
\newcommand{\ho}{{\rm hom}} 
\newcommand{\area}{{\rm area}} 

\newcommand{\odd}{{\rm odd}}
\newcommand{\even}{{\rm even}}

\newcommand{\hyp}{{\rm hyp}}
\newcommand{\str}{{\rm str}}

\DeclareMathOperator{\vol}{vol}

\def\={\;=\;}  \def\+{\,+\,} \def\m{\,-\,}       \def\h{\tfrac12}  
      \def\ssm{\smallsetminus}

\def\pp{{\boldsymbol \partial}}  

\DeclareMathAlphabet{\eucal}{U}{eus}{m}{n}
\DeclareMathAlphabet{\newcal}{U}{dutchcal}{m}{n}
\newcommand{\PPP}{{\mathcal P}}

 \def\Q{\QQ}   \def\fd{\mathfrak d} \def\pp{\partial}
  
\def\ev{{\rm ev}}  \def\Ev{{\rm Ev}}  \def\evX{\Ev}
\def\evh{\ev}
   \def\sbrX#1{\langle#1\rangle_X}
\def\brh#1{\bigl\langle#1\bigr\rangle_\hslash}
\def\sbrh#1{\langle#1\rangle_\hslash}

\def\la{\langle}  \def\ra{\rangle}

   \def\wt#1{\widetilde{#1}}     
\newcommand{\ol}{\overline}

  \def\wM{\wt M} 

\def\SL{{\textrm{SL}}}
\def\GL{{\textrm{GL}}}

\newcommand{\SP}{{\rm SP}}

\def\bq#1{\bigl\langle#1\bigr\rangle_q}   \def\sbq#1{\langle#1\rangle_q}
   \def\sbqs#1{\langle#1\rangle_q^\star}  
\def\bL#1{\bigl\langle#1\bigr\rangle_L}

\def\bstr#1{\bigl\langle#1\bigr\rangle_{\rm str}}

\def\bsL#1{\bigl\langle#1\bigr\rangle_{{\rm str},L}}

\def\a{\alpha}       \newcommand{\ve}{{\varepsilon}}
   \def\l{\lambda} 
  
      \def\p{\partial}  
        \def\Om{\Omega}

\def\be{\begin{equation}}   \def\ee{\end{equation}}     \def\bes{\begin{equation*}}    \def\ees{\end{equation*}}
\def\ba{\be\begin{aligned}} \def\ea{\end{aligned}\ee}   \def\bas{\bes\begin{aligned}}  \def\eas{\end{aligned}\ees}

\newcommand{\proj}{{\mathbb P}}
\newcommand{\moduli}[1][g]{{\mathcal M}_{#1}}
\newcommand{\omoduli}[1][g]{{\Omega \mathcal M}_{#1}}

\newcommand{\barmoduli}[1][g]{{\overline{\mathcal M}}_{#1}}
\newcommand{\obarmoduli}[1][g]{{\overline{\Omega\mathcal M}}_{#1}}

\def\M{\cMM}  \def\OmM{\Omega\M}
\newcommand{\oM}{\overline{\mathcal M}}
\newcommand{\oOmM}{\overline{\Omega\M}}
\newcommand{\oH}{\overline{\mathcal H}}

\newcommand{\BH}{\overline{H}}
\newcommand{\Hmu}{\Pi}

\newcommand{\lra}{\leftrightarrow}

\newcommand{\ES}{E_{0, \{1,2 \}}}
\newcommand{\BB}{{\mathrm{BB}}}
\newcommand{\ABB}{{\mathrm{ABB}}}
\newcommand{\ms}{\scalebox{0.75}[1.0]{$-$}}

\newcounter{savedtocdepth}
\newcommand*{\SaveTocDepth}[1]{%
  \addtocontents{toc}{%
    \protect\setcounter{savedtocdepth}{\protect\value{tocdepth}}%
    \protect\setcounter{tocdepth}{#1}%
  }%
}

\DeclareDocumentCommand{\LMS}{ O{\mu} O{g}} {\Xi\mathcal{M}_{#2}(#1)}

\title[Masur-Veech volumes and intersection theory]
      {Masur-Veech volumes and intersection theory
        on moduli spaces of Abelian differentials}

\begin{document}
\bibliographystyle{halpha}
\author{Dawei Chen}
\thanks{Research of the first author is partially supported by the NSF CAREER grant DMS-1350396.}
\address{Department of Mathematics, Boston College, Chestnut Hill, MA 02467, USA}
\email{dawei.chen@bc.edu}

\author{Martin M\"oller}
\thanks{Research  of the second author is partially supported  
  by the DFG-project MO 1884/1-1 and by the LOEWE-Schwerpunkt
``Uniformisierte Strukturen in Arithmetik und Geometrie''.}
\address{
Institut f\"ur Mathematik, Goethe--Universit\"at Frankfurt,
Robert-Mayer-Str. 6--8,
60325 Frankfurt am Main, Germany
}
\email{moeller@math.uni-frankfurt.de}

\author{Adrien Sauvaget}
\address{
Mathematical Institute, Utrecht University, Budapestlaan 6 / Hans Freudenthal Bldg, 3584 CD Utrecht, The Netherlands
}
\email{a.c.b.sauvaget@uu.nl}

\author{Don Zagier}
\address{MPIM Bonn, Vivatsgasse 7,
  53111 Bonn, Germany
}
\email{donzagier@mpim-bonn.mpg.de}

\maketitle

\begin{abstract}
We show that the Masur-Veech volumes and area Siegel-Veech constants can be obtained by intersection numbers on the strata of Abelian differentials with prescribed orders of zeros. As applications, we evaluate their large genus limits and compute the saddle connection Siegel-Veech constants for all strata. We also show that the same results hold for the spin and hyperelliptic components of the strata.  
\end{abstract}

\tableofcontents
\noindent
\SaveTocDepth{1} 


\section{Introduction} \label{sec:Intro}

Computing volumes of moduli spaces has significance in many fields. For instance, the Weil-Petersson volumes of moduli spaces of Riemann surfaces can be written as intersection numbers of tautological classes due to the work of Wolpert (\cite{wolpert85}) and of Mirzakhani for hyperbolic bordered surfaces with geodesic boundaries (\cite{MirzWP}). In this paper we establish similar results for the Mazur-Veech volumes of moduli spaces of Abelian differentials. 
\par 
Denote by $\omoduli[g,n](\mu)$ the moduli spaces (or strata) of Abelian differentials (or flat surfaces) with labeled zeros of type $\mu = (m_1, \ldots, m_n)$, where $m_i\geq 0$ and where $\sum_{i=1}^n m_i = 2g-2$. Masur (\cite{masur82}) and Veech (\cite{veech82}) showed that the hypersurface of flat surfaces of area one in $\omoduli[g,n](\mu)$ has finite volume, called the Masur-Veech volume, and we denote it by ${\rm vol}\,(\omoduli[{g,n}](\mu))$. The starting point of this paper is the following expression of Masur-Veech volumes in terms of intersection numbers on the incidence variety compactification
$\proj\oOmM_{g,n}(\mu)$ described in \cite{strata}. For $1\leq i\leq n$ we define 
\be \label{eq:beta}
\beta_i \= \frac{1}{m_i+1} \,\xi^{2g-2} \, \prod_{j\neq i} \psi_{j}
\quad \in H^{2(2g-3+n)}(\proj\oOmM_{g,n}(\mu),\Q)\,
\ee
where $\xi$ is the universal line bundle class of the projectivized Hodge bundle and 
$\psi_j$ is the vertical cotangent line bundle class associated to the $j$-th marked point
(see Section~\ref{sec:RelMZInt} for a more precise definition of these tautological classes). 
\par
\begin{Thm} \label{intro:IntFormula} The Masur-Veech volumes
can be computed as intersection numbers
\begin{eqnarray} \label{eq:MV/int}
{\rm vol}\,(\omoduli[{g,n}](m_1,\ldots,m_n)) 
&\=& - \frac{2(2i\pi)^{2g}}{(2g-3+n)!}
\int_{\proj \obarmoduli[g,n](\mu)} \!\!\!\!\!\!\!\!  \xi^{2g-2} \cdot \prod_{i=1}^{n} \psi_i \\
\label{for:intMV1} &\=& \frac{2(2i\pi)^{2g}}{(2g-3+n)!}
\int_{\proj \obarmoduli[g,n](\mu)} \!\!\!\!\!\!\!\! \beta_i\cdot  \xi\,
\end{eqnarray}
for each $1\leq i\leq n$. 
\end{Thm}
\par
The equality of the two expressions on the right-hand side is a non-trivial
claim about intersection numbers on $\proj \obarmoduli[g,n](\mu)$. Note that
we follow the volume normalization in \cite{emz} that differs slightly
from the one in \cite{eo} (see \cite[Section~19]{cmz} for the conversion).
\par
Theorem~\ref{intro:IntFormula} is the interpolation and generalization of~\cite[Proposition 1.3]{SauvagetMinimal} (for the minimal strata) and 
\cite[Theorem~4.3]{cmz} (for the Hurwitz spaces of torus covers). In order to prove it, we show that both sides of equation~\eqref{eq:MV/int} satisfy the same recursion formula. On the volume side, the recursion formula
is expressed via an operator acting on Bloch and Okounkov's algebra of
shifted symmetric functions (see~Section~\ref{sec:D2rec}). The recursion for intersection numbers is first proved at the numerical level using the techniques developed in~\cite{SauvagetClass}
to compute the classes of $\proj \obarmoduli[g,n](\mu)$ (see Sections~\ref{sec:hproj} and~\ref{sec:RelMZInt}). Then we 
formally lift this relation to the algebra of shifted symmetric functions and show that it is equivalent to the previous one (see Section~\ref{sec:D2VR}).
\par
In particular, the recursion arising from intersection calculations provides the following useful formula.  We define the rescaled volume 
\ba\label{eq:vol-rescale}
v(\mu) = (m_1+1)\cdots (m_n+1){\rm vol}\,(\omoduli[{g,n}](m_1,\ldots,m_n))\,.
\ea 
For a partition $\mu$, we denote by $n(\mu)$ the cardinality of $\mu$ and by $|\mu|$ the sum of its entries. 
\par
\begin{Thm} \label{intro:VolRec}
The rescaled volumes of the strata satisfy the recursion
\ba \label{eq:volintro}
v(\mu)  \=  \sum_{k \geq 1} 
\sum_{\bfg, \bfmu}
h_{\PP^1}((m_1,m_2),\bfp) 
\cdot  \frac{ \prod_{i=1}^k (2g_i-1+n(\mu_i))!\, v(\mu_i,  p_i-1)}
      {2^{k-1}\, k!\, (2g-3+n)!}\,
\ea
where $\bfg=(g_1,\ldots,g_k)$ is a partition of $g$, where~$\bfmu = (\mu_1,
\ldots,\mu_k)$ is a $k$-tuple of multisets with 
$(m_3,\ldots, m_n)= \mu_1\sqcup \cdots \sqcup  \mu_k$, 
and where $\bfp = (p_1,\ldots,p_k)$ is defined by $p_i = 2g_i-1-|\mu_i|$
and required to satisfy $p_i>0$.  
Here the Hurwitz number $h_{\PP^1}((m_1,m_2),\bfp)$ is defined for any~$\bfp$ 
by
\be \label{eq:defhP1}
h_{\PP^1}((m_1,m_2),\bfp)= (k-1)! [t_1^{m_1+1}t_2^{m_2+1}] \left( \prod_{i=1}^k
t_1t_2  \frac{(t_2^{p_i}-t_1^{p_i})}{t_2-t_1}\right)\,.
\ee
\end{Thm}
The  relevant Hurwitz spaces of $\PP^1$ covers will be introduced in
Section~\ref{sec:hproj}.  Note that $h_{\PP^1}((m_1,m_2),\bfp) \neq 0 $ only
if $\sum_{i=1}^k (p_i+1) = m_1 + m_2+2$. This implies that $k \leq \min(m_1+1,m_2+1)$
in the summation of the theorem.
\par
For special $\mu$ the strata $\omoduli[g,n](\mu)$ can be disconnected, with up to three spin and hyperelliptic connected components, classified by Kontsevich and Zorich (\cite{kz03}). We show the refinements of Theorem~\ref{intro:IntFormula} and Theorem~\ref{intro:VolRec} for the spin and hyperelliptic components respectively, given as Theorem~\ref{thm:refinedINT} and Theorem~\ref{thm:hypint} (conditional on Assumption~\ref{asu} which will be proved in an appendix).
\par
Equation~\eqref{eq:volintro} has a similar form compared to the recursion formula obtained by Eskin, Masur and Zorich (\cite{emz}) for computing saddle connection Siegel-Veech constants (joining two distinct zeros).  Consider a generic flat surface with~$n$ labeled zeros of orders $\mu=(m_1,\ldots, m_n)$. The growth rate
of the number of saddle connections of length at most~$L$ joining, say,
the first two zeros is quadratic and the leading term of the asymptotics
(up to a factor of $\pi$ to ensure rationality) is called the saddle connection Siegel-Veech constant.
Intuitively, the saddle connection Siegel-Veech constant should be proportional
to the cone angles around the two concerned zeros. For quadratic differentials this is not correct as shown by Athreya,
Eskin and Zorich (\cite{aez}). Nevertheless as an application of our formulas, we show that for Abelian
differentials the intuitive expectation indeed holds, if we use a minor modification~$c^\ho_{1\lra 2}(\mu)$ of the Siegel-Veech
constant counting homologous saddle connections only once. An
overview about the variants of Siegel-Veech constants is given in Section~\ref{sec:SV}.
\par
\begin{Thm} \label{intro:SVproduct}
The saddle connection Siegel-Veech constant $c^\ho_{1\lra 2}(\mu)$  joining
the first and the second zeros on a generic flat surface of type~$\mu$ is given by
\be \label{eq:SVmain}
c^\ho_{1\lra 2}(\mu) \= (m_1 +1)(m_2 + 1)\,.
\ee
\end{Thm}
\par
We also show that the theorem holds for each connected component of a disconnected stratum under the technical assumption (see Section~\ref{sec:spin}). As an asymptotic equality as $g$ tends to infinity, formula~\eqref{eq:SVmain}
was previously shown in the appendix by Zorich to~\cite{aggarwal} for saddle connections of multiplicity one and by Aggarwal~\cite{Agg2} for all multiplicities.  
\par 
Another important kind of Siegel-Veech constants is the area Siegel-Veech constant, which counts cylinders (weighted by the reciprocal of their areas) on flat surfaces and is related to the sum of Lyapunov exponents (\cite{ekz}) (see Section~\ref{sec:SV}
for the definition of area Siegel-Veech constants). We similarly establish an intersection formula for area Siegel-Veech constants.  
\par 
\begin{Thm} \label{intro:IntFormula2} The area Siegel-Veech constants of
the strata can be evaluated as 
\be \label{eq:SV/int}
c_{\rm area}(\mu) 
\= \frac{-1}{4\pi^2} \frac{
\int_{\proj \obarmoduli[g,n](\mu)} \!\beta_i\cdot  \delta_0}
{\int_{\proj \obarmoduli[g,n](\mu)} \! \beta_i\cdot  \xi}
\ee
for each $1\leq i\leq n$, where $\delta_0$ is the divisor class of the 
locus such that the underlying curve has a non-separating node.  
\end{Thm}
\par 
Theorem~\ref{intro:IntFormula2} completes the investigation of area Siegel-Veech constants
begun in \cite[Section 4]{cmz} (for the principal strata) and \cite[Equation~(2)]{SauvagetMinimal} (for the minimal strata). 
\par
Another application of the volume recursion is a geometric proof of 
the large genus limit conjecture by Eskin and Zorich (\cite{EZVol})
for the volumes of the strata and area Siegel-Veech constants. A proof using direct 
combinatorial arguments was given by Aggarwal (\cite{Agg2, aggarwal}).  Our proof indeed gives a uniform expression for the 
second order term as conjectured in~\cite{SauvagetMinimal} (see 
Section~\ref{sec:asy}).
\par
\begin{Thm} [{\cite[Main Conjectures]{EZVol}}]  \label{cor:EZconj}
Consider the strata $\omoduli[{g,n}](\mu)$ such that all the entries
of $\mu$ are positive.  Then 
\bas
{v(\mu)} &\= 4 \,-\, \frac{2\pi^2}{3\cdot \sum_{i=1}^n (m_i+1)} \+ 
O(1/g^2) \,, 
\\
{c_{\rm area}(\mu)}&\= \frac{1}{2} \,-\, \frac{1}{ 2 \sum_{i=1}^n (m_i+1)} \+
O(1/g^2) \,, 
\eas
where the implied constants are independent of $\mu$ and~$g$.
\end{Thm}
\par
Finally we settle another conjecture of Eskin
and Zorich on the asymptotic comparison of spin components. 
\par
\begin{Thm} [{\cite[Conjecture 2]{EZVol}}]  \label{thm:EZconj3} The volumes 
of odd and even spin components are comparable for large values of $g$. More
precisely, 
\bes
\frac{v(\mu)^{\rm odd}}{v(\mu)^{\rm even}} \= 1 \+ O(1/g)\,,
\ees
where the implied constant is independent of $\mu$ and~$g$.
\end{Thm}
\par
From this theorem one can presumably also deduce that  
\bes 
\frac{c_{\rm area}(\mu)^{\rm odd}}{c_{\rm area}(\mu)^{\rm even}} \= 
1 \+  O(1/g)\,,
\ees
by the strategy of Zorich's appendix to \cite{aggarwal} or by repeating
the strategy of the proof for volumes in the context of strict brackets
(see Section~\ref{sec:proofareaSV}). 
\par
\subsection*{Further directions}
Our work opens an avenue to study a series of related questions. First, we point out an interesting comparison with the proofs by Mirzakhani (\cite{MirzWP}), by Kontsevich (\cite{KontWitten}), and by Okounkov-Pandharipande (\cite{OPWitten}) of Witten's conjecture: the generating function of  $\psi$-class intersections on moduli spaces of curves is a solution of the KdV hierarchy of partial differential equations. Mirzakhani considered the Weil-Petersson volumes of moduli spaces of hyperbolic surfaces and analyzed geodesics that bound pairs of pants, while we consider the Masur-Veech volumes of moduli spaces of flat surfaces and analyze geodesics that join two zeros (i.e. saddle connections). Kontsevich interpreted $\psi$-classes as associated to certain polygon bundles, while we have the interpretation of Abelian differentials as polygons. Okounkov and Pandharipande used Hurwitz spaces of $\proj^1$ covers, while we rely on Hurwitz numbers of torus covers. Therefore, we expect that generating functions of Masur-Veech volumes and area Siegel-Veech constants should also satisfy a certain interesting hierarchy as in Witten's conjecture. 
\par
In another direction, one can consider saddle connections joining a zero to itself (see~\cite[Part 2]{emz}) or impose other specific configurations to refine the Siegel-Veech counting (see e.g.~the appendix by Zorich to~\cite{aggarwal}). From the viewpoint of intersection theory, such a refinement should pick up the corresponding part of the principal boundary when flat surfaces degenerate along the configuration, hence we expect that the resulting Siegel-Veech constant can be described similarly by a recursion formula involving intersection numbers.  
\par
One can also investigate volumes and Siegel-Veech constants for affine invariant manifolds (i.e. $\SL_2(\mathbb R)$-orbit closures in the strata). It is thus natural to seek intersection theoretic interpretations of these invariants for affine invariant manifolds, e.g.~the strata of quadratic differentials (see~\cite{DGZZ} for interesting related results in the case of the principal strata). 
\par
We plan to treat these questions in future work.  
\par
\subsection*{Organization of the paper}
In Section~\ref{sec:hproj} we introduce relevant intersection numbers on Hurwitz spaces of $\PP^1$ covers that will appear as coefficients in the volume recursion. In Section~\ref{sec:RelMZInt} we prove that the expression of volumes by intersection numbers satisfies the recursion in~\eqref{eq:volintro}, thus showing the equivalence of Theorems~\ref{intro:IntFormula} and~\ref{intro:VolRec}. In Section~\ref{sec:D2rec} we exhibit another recursion of volumes by using the algebra of shifted symmetric functions and cumulants. In Section~\ref{sec:D2VR} we show that the two recursions are equivalent by interpreting  them as the same summation over certain oriented graphs, thus completing the proof of Theorems~\ref{intro:IntFormula} and~\ref{intro:VolRec}. In Section~\ref{sec:spin} we refine the results for the spin and hyperelliptic components of the strata. In Sections~\ref{sec:SV}, ~\ref{sec:VRandSC} and~\ref{sec:NewHurwitz} we 
respectively review the definitions of various Siegel-Veech constants, prove Theorem~\ref{intro:SVproduct} regarding saddle connection Siegel-Veech constants and interpret the result from the perspective of Hurwitz spaces of torus covers.  In Section~\ref{sec:areaSV} we establish similar intersection and recursion formulas for area Siegel-Veech constants, thus proving Theorem~\ref{intro:IntFormula2}. Finally in Section~\ref{sec:asy} we apply our results to evaluate large genus limits of volumes and area Siegel-Veech constants, proving Theorems~\ref{cor:EZconj} and~\ref{thm:EZconj3}.  
\par
\subsection*{Acknowledgments}
We are grateful to Anton Zorich for valuable suggestions on Siegel-Veech constants and numerical cross-checks.  
Moreover, we would like to thank Alex Eskin and Andrei Okounkov for their help about spin components.  We also thank Amol Aggarwal, Qile Chen, 	
Yitwah Cheung, Eleny Ionel, Felix Janda, Xin Jin, Edward Looijenga, Duc-Manh Nguyen, Aaron Pixton, Alex Wright and Dimitri Zvonkine for inspiring discussions on related topics.  Finally we thank the Mathematisches Forschungsinstitut Oberwolfach (MFO), the Max Planck Institute for Mathematics (MPIM Bonn), the l'Institut Fourier (Grenoble), the American Institute of Mathematics (AIM), the Banff International Research Station (BIRS) and the Yau Mathematical Sciences Center (YMSC) for their hospitality during the preparation of this work. 
 

\section{Hurwitz spaces of $\PP^1$ covers}\label{sec:hproj}

In this section we recall the definition of the moduli
space of admissible covers of~\cite{harrismumford} as a compactification of the classical Hurwitz space (see also \cite{harrismorrison}), and prove formulas to compute recursively
intersection numbers of $\psi$-classes on these moduli spaces.
These intersection numbers will appear as coefficients and multiplicities
in the volume recursion. Along the way we introduce basic notions on stable graphs and
level functions.

\subsection{Hurwitz spaces and admissible covers}
\label{sec:genHurwitz}

Let $d,g,$ and~$g'$ be non-negative integers.
Let $\Hmu = (\mu^{(1)}, \cdots, \mu^{(n)})$ 
be a {\em ramification profile} consisting of~$n$ partitions.
We define the {\em Hurwitz space} $H_{d,g,g'}(\Hmu)$ to be the moduli space 
parametrizing branched covers of smooth connected curves $p\colon X \to Y$ of degree~$d$ with 
profile $\Hmu$ and such that the genera of $X$ and $Y$ are given by $g$ and $g'$ 
respectively. That is, $p$ is ramified over~$n$
points and over the $i$-th branch point the sheets coming together form 
the partition $\mu^{(i)}$ (completed by singletons if $|\mu^{(i)}| < \deg(p)$).
\par
The Hurwitz space $H_{d,g,g'}(\Hmu)$ has a natural compactification
$\BH_{d,g,g'}(\Hmu)$ parame\-trizing {\em admissible covers}. An admissible 
cover $p\colon X\to Y$ is a finite morphism of connected nodal curves such that
\begin{itemize}
\item[i)] the smooth locus of~$X$ maps to the smooth locus of~$Y$ and the nodes
of~$X$ map to the nodes of $Y$,
\item[ii)] at each node of $X$ the two branches have the same ramification order, and
\item[iii)] the target curve~$Y$ marked with the branch points is stable.
\end{itemize}

The space $\BH_{d,g,g'}(\Hmu)$ is equipped with two forgetful maps
$$
\xymatrix{
& \BH_{d,g,g'}(\Hmu) \ar[ld]_{f_S} \ar[rd]^{f_T} &\\
\oM_{g,n}&& \oM_{g',m}
}
$$
obtained by mapping an admissible cover to the stabilization of the source or 
the target. Here~$n$ denotes the number of branch points or equivalently
the length of~$\Hmu$ and~$m$ denotes the number of ramification points
or equivalently the number of parts (of length~$>1$) of all the $\mu^{(i)}$.  
The {\em Hurwitz number} $N^{\circ}_{d,g,g'}(\Hmu)$ is the degree of 
the map~$f_T$, or equivalently the number of connected covers $p\colon X \to Y$
of degree~$d$ with profile~$\Hmu$ and the location of the
branch points fixed in $Y$. We also denote by $N_{d,g,g'}(\Hmu)$ the Hurwitz 
number of covers without requiring~$X$ to be connected. We remark that each 
cover is counted with weight given by  the reciprocal of the order of its 
automorphism group, as is standard for the Hurwitz counting problem.  

\subsection{Intersection of $\psi$-classes on Hurwitz spaces} \label{subsec:n=2}

From now on in this section we will consider the special case $g'=0$.
Let $\mu [0]=(m_1,\ldots,m_n)$  be a list of non-negative integers
and $\mu [\infty ]=(p_1,\ldots,p_k)$ a list of positive integers.
We consider the Hurwitz space with profile $\Hmu$ given
by $\mu^{(i)} = (m_i+1)$ for $i \leq n$ and $\mu^{(n+1)} = (p_1,\ldots,p_k)$
such that $d=\sum_{i=1}^k p_i$, i.e.\ we consider
$$\BH_{\proj^1}(\mu [0],\mu [\infty ]) \=
\BH_{d,g,0}((m_1+1),\ldots,(m_n+1),(p_1,\ldots,p_k))\,.
$$ 
By the Riemann-Hurwitz formula, the genus $g$ of the covering surfaces satisfies that 
$$ 2-2g = k+d - \sum_{i=1}^n m_i = k + \sum_{i=1}^k p_i -  \sum_{i=1}^n m_i\,.$$ 
The forgetful map $f_S$ goes from $\BH_{\proj^1}(\mu [0],\mu [\infty ])$ to $\oM_{g,n+k}$, where we assume that the first $n$ marked points are the first $n$ ramification points and the preimages of the last branch point are the $k$ last marked points. Since there are $n+1$ branch points in the target surface of genus zero, we conclude that 
$\dim \BH_{\proj^1}(\mu [0],\mu [\infty ]) = \dim \barmoduli[0,n+1] = n-2$.  
\par
For $g = 0$, define the following intersection numbers on the Hurwitz spaces 
\be \label{eq:HNdef}
  h_{\proj^1}(\mu [0],\mu [\infty ]) \=
\int_{\BH_{\proj^1}(\mu [0 ],\mu [\infty ])} f_S^*\left(\prod_{i=3}^{n} \psi_i\right)\,.
\ee 
The definition of $\psi$-classes will be recalled in Section~\ref{sec:RelMZInt}. 
If $n=2$, then $\BH_{\proj^1}(\mu [0 ],\mu [\infty ])$ is of dimension zero, and hence the intersection number on the right is just the number
of points of the Hurwitz space, i.e.\  
$$h_{\proj^1}((m_1,m_2),\mu [\infty ])
\= N^{\circ}_{d,0,0}((m_1+1), (m_2+1), (p_1,\ldots,p_k))\,.$$
Again we emphasize that the Hurwitz number on the right-hand side is counted with
weight $1/|{\rm Aut}|$ for each cover. Correspondingly
the intersection numbers are computed on the Hurwitz space treated as a stack. Our goal for the rest of the section is to show the following result.  
\par
\begin{Prop} \label{prop:intpsiPP1}
For $n = 2$, $h_{\proj^1}((m_1,m_2),\mu [\infty ])$ can be computed by the coefficient extraction
\bes
h_{\proj^1}((m_1,m_2),\mu [\infty ])
\= (k-1)! [t^{m_1+1}]  \prod_{i=1}^k \frac{\,\, t-t^{p_i+1}}{\!\!\! 1-t }\,,
\ees
and for $n \geq 3$, $h_{\proj^1}(\mu [0],\mu [\infty ])$ can be computed recursively by the sum
\bes
h_{\proj^1}(\mu[0],\mu[\infty])
\=\sum_{\Gamma \in {\rm RT}(\mu[0],\mu[\infty])_{1,2}} h(\Gamma)\, 
\ees
over rooted trees.
\end{Prop}
\par
The definitions of rooted trees and the local contributions $h(\Gamma)$
are given in Sections~\ref{subsec:RT} and~\ref{sec:HurRT} respectively.
The above formula for the Hurwitz number obviously agrees
with~\eqref{eq:defhP1}.
\par
\begin{proof}[Proof of Proposition~\ref{prop:intpsiPP1}, case $n=2$]
Let $S_d$ be the symmetric group acting on $[\![1,d]\!] = \{1,\ldots, d \}$, 
where $d = p_1 + \cdots + p_k = m_1 + m_2 + 2 - k$. Define the set of Hurwitz 
tuples 
$$A(m_1,m_2,\mu [\infty ]) \= 
\{ (\sigma_1,\sigma_2,\sigma_\infty) \} \subset S_d \times S_d \times S_d$$ 
such that 
\begin{itemize}
\item the permutation $\sigma_\infty$ is in the conjugacy class of the partition $(p_1,\ldots,p_k)$, and $\sigma_1$ and $\sigma_2$ are cycles of order $m_1+1$ and $m_2+1$ respectively,  
\item the relation $\sigma_1\circ \sigma_2=\sigma_\infty$ holds, and  
\item the group generated by $\sigma_1$, $\sigma_2$ and $\sigma_\infty$ acts transitively on $[\![1,d]\!]$.
\end{itemize}
Then the (weighted) Hurwitz number $h_{\proj^1}((m_1,m_2),\mu [\infty ]) = |A(m_1,m_2,\mu [\infty ])| / d!$.  
\par
The second and third conditions above imply that the union of the supports of the cycles $\sigma_1$ and $\sigma_2$ is $[\![1,d]\!]$. Therefore, $\sigma_1$ and $\sigma_2$ contain exactly 
$(m_1+1)+(m_2+1)-d=k$ common elements. We can write  
$$ \sigma_1 = (a_1,\ldots, a_{i_1-1}, c_1; a_{i_1+1}, \ldots, a_{i_2-1}, c_2; \ldots ; a_{i_{k-1}+1}, \ldots, a_{i_k-1}, c_k)\,, $$
where $1\leq i_1 < i_2 < \cdots < i_k = m_1+1$ and $c_1, \ldots, c_k$ are the common elements of $\sigma_1$ and $\sigma_2$. Since $\sigma_1\circ \sigma_2=\sigma_\infty$ is of conjugacy type $(p_1,\ldots,p_k)$, it is easy to see that $\sigma_2$ must be of the form
$$ \sigma_2 \= (b_1,\ldots, b_{j_1-1}, c_k; b_{j_1+1}, \ldots, b_{j_2-1}, c_{k-1}; \ldots ; b_{j_{k-1}+1}, \ldots, b_{j_k-1}, c_1)\,, $$
for certain $b_i$, such that 
$$ \{ j_1+i_1, (j_2 - j_1) + (i_{k} - i_{k-1}), \ldots, (j_{k} - j_{k-1}) + (i_2 - i_1) \} = \{p_1 + 1, p_2 + 1, \ldots, p_k + 1\}\,. $$ 
If $\tau$ is a permutation on $[\![1,k]\!]$ such that $(j_{k+1-\ell} - j_{k-\ell}) + (i_{\ell+1} - i_\ell) = p_{\tau(\ell)} + 1$, then we have $1\leq i_{\ell+1} - i_\ell \leq p_{\tau(\ell)}$. Conversely, such $\tau$ and $i$-indices determine the $j$-indices. 
\par 
There are $m_1! { d \choose m_1+1}$ choices for $\sigma_1$. Fixing $\sigma_1$, 
to construct $\sigma_2$ we first choose 
\begin{itemize}
\item a permutation $\tau \in S_k$, and then 
\item a partition $(i_1, i_2 - i_1, \ldots, i_k - i_{k-1})$ of $m_1+1$ 
such that $1\leq i_{\ell+1} - i_\ell \leq p_{\tau(\ell)}$ for all $\ell$.
\end{itemize}
This gives $k! [t^{m_1+1}]\prod_{i=1}^k (t+\cdots+t^{p_i})$ choices. Choose $c_1$ out of the elements in $\sigma_1$, which gives $m_1+1$ choices. Along with the $i$-indices this determines the elements $c_2, \ldots, c_k$ as well as the set of $b$'s as the complement of the union of $a$'s and $c$'s. Finally $\sigma_2$ is determined by arranging the $b$'s, which gives $(m_2+1-k)!$ choices. Note that in this process only 
the cyclic order of $(c_1, \ldots, c_k)$ matters and we cannot actually determine which one is the first $c$, hence we need to divide the final count by $k$.   
\par 
In summary, we conclude that 
$$ |A(m_1,m_2,\mu [\infty ])|  =  (m_1+1)!(m_2+1-k)! { d \choose m_1+1} \cdot (k-1)! [t^{m_1+1}]\prod_{i=1}^k (t+\cdots+t^{p_i})\,. $$
Since $d = m_1+m_2 + 2 - k$, we obtain
$$h_{\proj^1}((m_1,m_2),\mu [\infty ]) 
\= |A(m_1,m_2,\mu [\infty ])| / d! 
\= (k-1)! [t^{m_1+1}]\prod_{i=1}^k \frac{t(1-t^{p_i})}{1-t} $$
using that $(m_1+1)!\,(m_2+1-k)!\, { d \choose m_1+1} = d!$. 
\end{proof}
\par
We remark that the above Hurwitz counting problem can also be interpreted by the angular data of the configurations of saddle connections joining two zeros $z_1$ and $z_2$ of order $m_1$ and $m_2$ respectively in the setting of \cite{emz}.  
Suppose $f\colon \proj^1 \to \proj^1$ is a branched cover parameterized in the Hurwitz space $H_{d,0,0}((m_1+1), (m_2+1), (p_1,\ldots,p_k))$, where we treat $f$ as a meromorphic function with $k$ poles of order $p_1, \ldots, p_k$. Then the meromorphic differential $\eta = df$ has two zeros of order $m_1$ and $m_2$ as well as $k$ poles of order $p_1+1, \ldots, p_k+1$ with no residue. Conversely given such $\eta$, integrating $\eta$ gives rise to a desired branched cover $f$. Such $\eta$ can be constructed using flat geometry as in \cite[Section 2.4]{chenPB}. In particular, it is determined by the angles $2\pi (a'_i+1)$ between the saddle connections (clockwise) at $z_1$ and the angles $2\pi (a''_i+1)$ (counterclockwise) at $z_2$, such that 
$ \sum_{i=1}^k (a'_i+1) = m_1 + 1$, $\sum_{i=1}^k (a''_i+1) = m_2 + 1$ and $a'_i + a''_{i} + 2= p_i + 1$. We see again that the choices involve a partition $(a'_1+1, \ldots, a'_k+1)$ 
 of $m_1 + 1$ such that $a'_i + 1 \leq p_i$ for all $i$.   

\subsection{Level graphs and rooted trees}\label{subsec:RT}

The boundary of the Deligne-Mumford compactification $\barmoduli[g,n]$ 
is naturally stratified by the topological types of the stable marked 
surfaces. These boundary strata  are in one-to-one correspondence with 
stable graphs, whose definition we recall below. The boundary strata of 
Hurwitz spaces and of moduli spaces of Abelian differentials are encoded 
by adding level structures and twists to stable graphs.
\par
\begin{Defi}\label{def:stgraph} A {\em stable graph} is the datum of
\begin{equation*}
\Gamma \= (V, H, g\colon V \to \mathbb{N}, a\colon H \to V, i\colon H \to H, E, L\simeq [\![1,n]\!])
\end{equation*}
satisfying the following properties:
\begin{itemize}
\item $V$ is a vertex set with a genus function $g$;
\item $H$ is a half-edge set equipped with a vertex assignment $a$
and an involution~$i$ (and we let $n(v) = |a^{-1}(v)|$);
\item $E$, the edge set, is defined as the set of length-$2$ orbits of~$i$
in $H$ (self-edges at vertices are permitted);
\item $(V, E)$ define a connected graph;
\item $L$ is the set of fixed points of $i$, called {\em legs} or {\em markings}, and is identified with $[\![1,n]\!]$;
\item for each vertex $v$, the stability condition $2g(v) - 2 + n(v) > 0$
holds.
\end{itemize}
Let $v(\Gamma)$ and $e(\Gamma)$ denote the cardinalities of $V$ and $E$ respectively.  The genus of $\Gamma$ is defined 
by $\sum_{v\in V(\Gamma)} g(v) + e(\Gamma) - v(\Gamma) + 1$. 
\end{Defi}
\par
We denote by ${\rm Stab}(g,n)$ the set of stable graphs of genus~$g$
and with $n$~legs. A stable graph is said of {\em compact type} if
$h^1(\Gamma)=0$, i.e.\ if the graph has no loops, which is thus a {\em tree}.  
\par
\medskip
We will use two extra structures on stable graphs, called level functions
and twists. As in \cite{strata} we define a {\em level graph} to be a
stable graph $\Gamma$
together with a level function $\ell\colon V(\Gamma) \to \RR_{\leq 0}$. An edge 
with the same starting and ending level is called a {\em horizontal edge}. 
A {\em bi-colored graph} is a level graph with two levels (in which case we 
normalize the level function to take values in $\{0,-1\}$) that has 
no horizontal edges. We denote the set of bi-colored graphs by ${\rm Bic}(g,n)$.
\par
Recall the notation $\mu[0] = (m_1, \ldots, m_n)$ and $\mu[\infty] = (p_1, \ldots, p_k)$
where $m_i \geq 0$ and $p_j > 0$ for all $i$ and $j$.    
\par
\begin{Defi} \label{def:twist}
Let $\Gamma$ be a stable graph in ${\rm Stab}(g,n+k)$.
A {\em twist assignment} on $\Gamma$ of type $(\mu[0],\mu[\infty])$ is a function
$\bfp\colon H(\Gamma)\to \ZZ$ satisfying the following conditions: 
\begin{itemize}
\item If $(h,h')$ is an edge, then $\bfp(h)+\bfp(h')=0$.
\item For all $1\leq i\leq n$, the twist of the $i$-th leg is $m_i+1$ and
for all $1\leq i\leq k$ the twist of the $(n+i)$-th leg is $-p_i$. 
\item For all vertices $v$ of $\Gamma$ \bes
2g(v) -2 + n(v) \= \!\!\!\!\! \sum_{h\in L, \, a(h) = v} \bfp(h)\,.
\ees
\end{itemize}
\end{Defi}
\par
Suppose the graph~$\Gamma$ comes with a level structure~$\ell$. We say that
a twist~$p$ is {\em compatible} with the level structure if for all
edges $(h,h')$ the condition $\bfp(h)>0$ implies that $\ell(a(h))>\ell(a(h'))$, 
and respectively for the cases~$<$ and~$=$. In this case
we call the triple $(\Gamma,\ell,\bfp)$ a {\em twisted level graph}. For the reader
familiar with related results of compactifications of strata of Abelian differentials, 
the above definition characterizes  twisted differentials (or canonical divisors)
in \cite{strata} and \cite{fapa} (regarding the $m_i$ as the zero orders and $p_j+1$
as the pole orders of twisted differentials on irreducible components of the corresponding stable curves). In particular, 
every level graph has only finitely many compatible twists. For a graph
of compact type, there exists a unique twist~$\bfp$ if the entries 
of $(\mu[0],\mu[\infty])$ satisfy the condition that  
$\sum_{i=1}^n m_i - \sum_{j=1}^k (p_j+1) = 2g-2$. 
\par
\begin{Defi}\label{def:RT}
Let $1\leq i< j\leq n$. A {\em stable rooted tree} (or simply a {\em rooted tree}) is a twisted
level graph $(\Gamma,\ell,\bfp)$ of compact type satisfying the following conditions: 
\begin{itemize}
\item[i)] One vertex $v_j$, called the {\em root}, carries the $i$-th and $j$-th legs and no other
of the first~$n$ legs, a vertex on the path from~$v$ to the root is called an {\em ancestor} of~$v$, and a vertex whose ancestors contain~$v$  is called a~{\em descendant} of $v$;
\item[ii)] There are no horizontal edges; 
\item[iii)] A vertex $v$ is on level $0$ if and only if $v$ is a leaf.  If $v$ is not a leaf, then 
$\ell(v) = {\rm min}\{\ell(v')\,|\, v'\ \mbox{is a descendant of}\ v\}-1$;
\item[iv)] All vertices of positive genus are leaves and hence on level $0$; 
\item[v)] Each vertex of genus zero other than $v_j$ carries exactly
one of the first $n$ legs.
\end{itemize}
\end{Defi}
Since the root is an ancestor of any other vertex, by definition it is the unique vertex lying on the bottom level, hence it has genus zero. Moreover, it is easy to see from the definition that any path towards the root is strictly going down. In particular, any vertex except the root has a unique ancestor.  
\par
\begin{figure}[ht]
\begin{tikzpicture}
\tikzstyle{every node}=[font=\scriptsize]
\draw[] (-2,1.72) node[] {$\ell =~0$};
\draw[] (-2,1.1) node[] {$\ell =-1$};
\draw[] (-2,.65) node[] {$\ell =-2$};
\draw[] (-2,.2) node[] {$\ell =-3$};

\path[name path=line A,draw] 
	(-1,1.6) node[shape=circle,draw,inner sep=.5pt,xshift=-1pt,yshift=3.5pt] 
	{$5$} -- +(-55:2) node[xshift=2pt,yshift=-5pt] {$2$}; 
\path[name path=line D,draw]
	(1,1.6) node[shape=circle,draw,inner sep=.5pt,xshift=1pt,yshift=3.5pt] 
	{$1$} -- +(-125:2) node[xshift=-2pt,yshift=-5pt] {$1$}; 
\path[name path=line C,draw] 
	(.4,1.6) node[shape=circle,draw,inner sep=.5pt,xshift=1pt,yshift=3.5pt] 
	{$2$} -- +(-125:1.6) node[xshift=-2pt,yshift=-5pt] {$4$}; 
\path[name path=line B,draw] 
	(-.2,1.6) node[shape=circle,draw,inner sep=.5pt,xshift=1pt,yshift=3.5pt] 
	{$3$} -- +(-125:1.2) node[xshift=-2pt,yshift=-5pt] {$3$};

\fill[name intersections={of=line A and line B}] (intersection-1) circle (1.5pt);
\fill[name intersections={of=line A and line C}] (intersection-1) circle (1.5pt);
\fill[name intersections={of=line A and line D}] (intersection-1) circle (1.5pt);

\draw[] (1.09,1.6) -- +(-50:.15) node[xshift=2pt,yshift=-2pt] {$6$}; 
\draw[] (1.045,1.6) -- +(-90:.15) node[yshift=-4pt] {$7$};
\draw[] (.49,1.6) -- +(-50:.15) node[xshift=2pt,yshift=-2pt] {$5$};
\end{tikzpicture} 
\caption{A rooted tree of genus eleven with three vertices of genus zero (black) 
and seven legs} \label{cap:RT}
\end{figure}
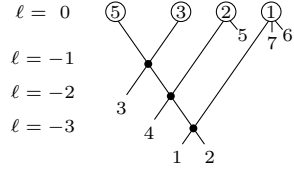
\par
We denote by ${\rm RT}(g,\mu[0],\mu[\infty])_{i,j}$ the set of such 
rooted trees, and sometimes simply by ${\rm RT}(\mu[0],\mu[\infty])_{i,j}$ if $ g = 0$.   

\subsection{The sum over rooted trees} \label{sec:HurRT}

Now we assume that $g=0$. Below we define the local contributions from rooted trees 
in Proposition~\ref{prop:intpsiPP1} and complete its proof. Consider a graph $\Gamma \in {\rm RT}(\mu[0],\mu[\infty])_{1,2}$. 
Since by assumption every vertex of $\Gamma$ has genus zero, condition~v) implies that $\Gamma$ has
exactly $n-1$ vertices and $n-2$ edges. Denote by $v_2, \ldots, v_n$ the vertices of $\Gamma$ such that 
 $v_i$ carries the $i$-th leg $h_i$ for $3\leq i\leq n$ and $v_2$ carries the first two legs. This convention is consistent
with our previous notation for the root. We denote by $\mu[\infty]_i$ the list of negative twists at half-edges adjacent to $v_i$. 
These half-edges are either part of the whole edges joining $v_i$ to its descendants (as adjacent vertices to $v_i$ on higher level) or part of the $k$ last legs (corresponding to the $k$ marked poles).
\par
If $i\neq 2$, then there is a unique (non-leg) half-edge $\widetilde{h}_i \neq h_i$
adjacent to~$v_i$ such that $\widetilde{m}_i:=\bfp(\widetilde{h}_i)-1\geq 0$.
Namely, this half-edge is part
of the whole edge joining $v_i$ to its ancestor (as the adjacent vertex to $v_i$ on lower level). 
With this notation we define
the contribution of the rooted tree $\Gamma$ as
\ba\label{eq:h(gamma)}
h(\Gamma) \= h_{\proj^1}((m_1,m_2), \mu[\infty]_2) \,\cdot \,
\prod_{i=3}^{n} h_{\proj^1}((m_i, \widetilde{m}_i),\mu[\infty]_i)\,.
\ea
\par
Let $J\subset [\![1,n+k]\!]$ be a subset such that the cardinalities of $J$ and $J^c$ are at least two. 
Denote by $\delta_J$ the class of the boundary divisor of
$\oM_{0,n+k}$ parameterizing curves that consist of a component with
the markings in $J$ union a component with the markings in $J^{c}$. 
We need the following classical result (see e.g.~\cite[Lemma~7.4]{acgh2}).  
\par
\begin{Lemma} For all $3\leq i\leq n + k$, the following relation of divisor classes holds on $\oM_{0,n+k}$:  
$$
\psi_i \=\sum_{i\in J\subset [\![ 3,n+k]\!]} \delta_J\,.
$$
\end{Lemma}
\par
If $J$ is a subset of $[\![3,n]\!]$, then we denote by $\widetilde{\delta}_{J}=\sum_{J'\subset [\![n+1,n+k]\!]} \delta_{J\cup J'}$.
For $3\leq i \leq n$, the above lemma implies that 
\begin{equation}\label{eqn:psi0}
\psi_i \=\sum_{i\in J\subset [\![ 3,n]\!]} \widetilde{\delta}_J\,.
\end{equation}
\par
We also need the following result about the boundary divisors of $\BH_{\proj^1}(\mu [0],\mu [\infty ])$. 
\par
\begin{Lemma}\label{lem:hurwitz-boundary}
There is a bijection between the boundary divisors of $\BH_{\proj^1}(\mu [0],\mu [\infty ])$ and the corresponding bi-colored graphs (i.e. with two levels only). 
\par
Moreover, a boundary divisor does not drop dimension under the source map $f_S\colon \BH_{\proj^1}(\mu [0],\mu [\infty ]) \to \oM_{0,n+k}$ only if its bi-colored graph has two vertices (i.e. a unique vertex on each level).  
\end{Lemma}
 \par
 \begin{proof}
 The first part of the claim follows from the same argument as in~\cite[Proposition 7.1]{SauvagetClass}. Here the vertices of level $0$ in the bi-colored graphs correspond to the components of the admissible covers that contain the marked poles. For the other part, suppose that a generic point of a boundary divisor has at least two vertices on level $0$ (or on level $-1$). Then one can scale one of the two functions that induce the covers on the two vertices such that the domain marked curve is fixed while the admissible covers vary. It implies that $f_S$ restricted to this boundary divisor has positive dimensional fibers.    
 \end{proof}
\par
\begin{proof}[Proof of Proposition~\ref{prop:intpsiPP1}, case $n \geq 3$]
We will prove the result by induction on $n$. The beginning case $n=2$ follows from the definition of $h(\Gamma)$ in~\eqref{eq:h(gamma)} and we have also described it explicitly in Section~\ref{subsec:n=2}.  
The strategy of the induction for higher $n$ is by successively
replacing the $\psi_i$ in~\eqref{eq:HNdef} with the sum
over boundary divisors as in the preceding lemma, starting with $i=3$. To
simplify notation, we write $\BH(\mu [0],\mu [\infty ])$ instead of
$\BH_{\proj^1}(\mu [0],\mu [\infty ])$.  We also simply write~$\psi$ and~$\delta$
as classes in the Hurwitz space for their pullbacks via~$f_S$.  
\par 
Consider a boundary divisor $\delta_J$ of $\oM_{0,n+k}$ pulled back to $\BH(\mu [0],\mu [\infty ])$, which is a union of certain boundary divisors of $\BH(\mu [0],\mu [\infty ])$. We would like to compute the intersection number $\delta_J\cdot \prod_{i= 4}^n \psi _i$ on $\BH(\mu [0],\mu [\infty ])$.  By Lemma~\ref{lem:hurwitz-boundary} and the projection formula, the only possible non-zero contribution is from the loci in $\delta_J$ whose 
bicolored graphs have a unique edge $e=(h,h')$ connecting two vertices $v_{0}$ and $v_{-1}$ on level $0$ and level $-1$ respectively, such that the last $k$ markings (i.e. the $k$ marked poles) are contained in $v_0$. In this case we can assume that $\bfp(h)> 0$ (and hence $\bfp(h') < 0$ as $\bfp(h) + \bfp(h') = 0$ by definition). The admissible covers restricted to $v_{0}$ and to $v_{-1}$ belong to Hurwitz spaces of similar type, where at the node
(i.e.\ the edge~$e$) the ramification order of the restricted maps is given by $\bfp(h)-1 = -\bfp(h')-1$. It implies that the locus of such admissible covers can be identified with $\BH(\mu[0]^0,\mu[\infty]^0)\times \BH(\mu[0]^{-1},\mu [\infty]^{-1})$, where $\mu[0]^0$ is the part of $\mu[0]$ contained in $v_0$ union with $\bfp(h)-1$, $\mu[\infty]^0 = \mu[\infty]$, $\mu[0]^{-1}$ is the part of 
$\mu[0]$ contained in $v_{-1}$, and $\mu [\infty]^{-1}$ has a single entry $\bfp(h)$.  
\par
By equation~\eqref{eqn:psi0} we have $\psi_3= \sum_{3\in J\subset [\![3,n+k]\!]} \delta_J$.  
Fix a subset $J\subset [\![ 3,n+k]\!]$ such that $3 \in J$.
If the third marking belongs to $v_{-1}$, we claim that  
$$\BH(\mu [0],\mu [\infty ])\cdot \delta_J\cdot \prod_{i= 4}^n \psi _i \= 0\,.$$
To see this, note that the dimension of $\BH(\mu [0]^{-1},\mu [\infty ]^{-1})$ is equal
to $n(\mu [0]^{-1}) - 2$, as $\mu [\infty ]^{-1}$ has only one entry.  
However there are
$n(\mu [0]^{-1})-1$ markings with label $\geq 4$ on $v_{-1}$. Consequently the intersection with the product of those 
$\psi_i$ vanishes on $\BH(\mu [0]^{-1},\mu [\infty ]^{-1})$. 
\par
Therefore, we only need to consider the case when $v_0$ contains the third marking (hence all markings labeled by $J$), and consequently  
$v_{-1}$ contains the first and second markings (hence all markings labeled by $J^c$). In this case we obtain that  
\begin{eqnarray*}
\BH(\mu [0],\mu [\infty ])\cdot \delta_J\cdot \prod_{i= 4}^n \psi _i
\= h_{\proj^1} (\mu [0]^0,\mu [\infty]^0) \cdot h_{\proj^1} (\mu[0]^{-1},\mu [\infty ]^{-1})\,,
\end{eqnarray*}
where $h_{\proj^1}$ is defined in~\eqref{eq:HNdef}, where in the first factor on the right-hand side the $\psi$-product skips the third marking and the marking from the half-edge of $v_0$, and where in the second factor the $\psi$-product skips the first and second markings. 
\par
Now we use the induction hypothesis to decompose the factors
$h_{\proj_1} (\mu [0]^a,\mu [\infty ]^a)$ for $a=0$ and $a=-1$. It leads to a sum over all possible pairs of rooted trees, where the two rooted trees in each pair generate a new rooted tree. More precisely, one rooted tree in the pair contains the markings of $J^c\cup \{h'\}$ whose root $v_2$ carries the first and second markings, the other rooted tree contains the markings in $J\cup \{h\}$ whose root $v_3$ carries the third marking and $h$, and they generate a new rooted tree by gluing the legs $h$ and $h'$ as a whole edge and by using $v_2$ as the new root.  
\par
Therefore, if $J$ is a subset of $[\![3,n]\!]$ such that $3\in J$, then we obtain that 
\begin{eqnarray*}
  \BH(\mu [0],\mu [\infty ])\cdot \widetilde{\delta}_J\cdot \prod_{i= 4}^n \psi _i \=
  \sum_{\begin{smallmatrix}\Gamma \in {\rm RT}(\mu[0],\mu[\infty])_{1,2}, \\ j\in J \Leftrightarrow \ell(v_3)\leq \ell(v_j) \end{smallmatrix}} h(\Gamma)\,,
\end{eqnarray*}
where the sum is over all rooted trees $\Gamma$ such that the descendants of $v_3$ are exactly the vertices $v_j$ for $j\in J\setminus \{3\}$.   
\par
In summary if we write $\psi_3= \sum_{3\in J\subset [\![3,n]\!]} \widetilde{\delta}_J$, then by the above analysis we thus conclude that 
$\BH(\mu [0],\mu [\infty ])\cdot\psi_3\cdot \prod_{i= 4}^n \psi _i$ is equal to the sum of the contributions $h(\Gamma)$ over all rooted trees $\Gamma$.  
\end{proof}



\section{Volume recursion via intersection theory}
\label{sec:RelMZInt}

In this section we show that the two main theorems of the
introduction, Theorem~\ref{intro:IntFormula} and Theorem~\ref{intro:VolRec} 
are equivalent. This section does not yet provide a direct proof of
either of them. 
\par
We first show that the intersection numbers in Theorem~\ref{intro:IntFormula}
are given by a recursion formula of the same shape as in
Theorem~\ref{intro:VolRec}. Together with an agreement on the minimal
strata this proves the equivalence of the two theorems. Along the way
we introduce special classes of stable graphs that are used for recursions
throughout the paper. 

\subsection{Intersection numbers on the projectivized Hodge bundle}
\label{ssec:classes}

Fix $g$ and $n$ such that $2g-2+n>0$. We denote by $f\colon \cXX \to \barmoduli[g,n]$
the universal curve and  by $\omega_{\cXX/\barmoduli[g,n]}$ the relative dualizing line bundle.
We will use the following cohomology classes: 
\begin{itemize}
\item Let $1\leq i \leq n$. We denote by $\sigma_i\colon \barmoduli[g,n] \to \cXX$ the section of $f$ corresponding to the $i$-th marked point and by $\cLL_i=\sigma_i^* \omega_{\cXX/\barmoduli[g,n]}$ the cotangent line at the $i$-th marked point. With this notation, we define $\psi_i =c_1(\cLL_i)\in H^2(\barmoduli[{g,n}],\QQ)$.
\item For $1\leq i\leq g$, we denote by
$\lambda_i=c_i(\obarmoduli[{g,n}]) \in H^{2}(\barmoduli[{g,n}],\QQ)$
the $i$-th Chern class of the Hodge bundle. (We use the same notation for a vector bundle and its total space.) 
\item We denote by $\delta_0\in H^{2}(\barmoduli[{g,n}],\QQ)$ the Poincar\'e-dual class of the divisor parameterizing marked curves with at least one non-separating node.
\item The projectivized Hodge bundle $\proj\obarmoduli[{g,n}]$ comes with
the universal line bundle class $\xi=c_1(\cOO(1))\in H^2(\proj\obarmoduli[{g,n}],\QQ)$.
\end{itemize}
Unless otherwise specified, we denote by the same symbol a class
in $H^*(\oM_{g,n},\QQ)$ and its pull-back via the projection
$p\colon \proj\obarmoduli[{g,n}] \to \barmoduli[g,n]$. Recall that the splitting
principle implies that the structure of the cohomology ring of the
projectivized Hodge bundle is given by
$$
H^*(\proj\oOmM_{g,n},\Q) \,\simeq \,
H^*(\oM_{g,n},\Q)[\xi]/(\xi^g+\lambda_1 \xi^{g-1}+\cdots+\lambda_g)\,.
$$
\par
\medskip
Let $\mu=(m_1,\ldots,m_n)$ be a partition of $2g-2$. We denote by $\proj\obarmoduli[g,n](\mu)$ the closure of the projectivized
stratum $\proj\omoduli[g,n](\mu)$ inside the total space of
the projectivized Hodge bundle $\proj\obarmoduli[g,n]$.  This space is called the {\em (ordered) incidence variety compactification}\footnote{In \cite{strata}
the notation $\proj\obarmoduli[g,n]^{\textrm{inc}}(\mu)$ is used. 
Here we drop the superscript~``${\textrm{inc}}$'' for simplicity.
In \cite{SauvagetClass} this space is denoted by $\proj\oH_{g,n}(\mu)$.}.
\par
In this section we study the intersection numbers 
\begin{eqnarray}
\label{eq:a(mu)}
a_i(\mu) &=& \int_{\proj\obarmoduli[g,n](\mu)} \!\!\!\!\!\!\! \beta_i\cdot \xi =\frac{1}{m_i+1} \int_{\proj\obarmoduli[g,n](\mu)} \!\!\!\!\!\!\!  \xi^{2g-1} \cdot \prod_{j\neq i} \psi_j\, 
\end{eqnarray}
for all $1\leq i\leq n$. The reader should think of the~$a_i(\mu)$
as certain normalization of volumes. In fact, Theorem~\ref{intro:IntFormula} can be reformulated as
\be \label{eq:volai}
{\rm vol}(\omoduli[g,n](\mu)) \= \frac{2(2\pi)^{2g}(-1)^g}{(2g-3+n)!} \, 
a_i(\mu)\,,
\ee
implying in particular that $a_i(\mu)$ is independent of~$i$.
\par
We prove a collection of properties defining recursively the $a_i(\mu)$
 as the coefficients of some formal series. 
As the base case for $n=1$, i.e.~$\mu = (2g-2)$, define 
 the formal series 
\bes
\cAA(t) \= \frac{1}{t} +\sum_{g\geq 1} (2g-1)^2\, a_1(2g-2)  t^{2g-1} \quad
\in \frac{1}{t}\QQ[[t]]
\ees
and set 
\be \label{eq:defBandb}
B(z) \= \frac{z/2}{\sinh(z/2)} 
\, =: \, \sum_{j \geq 0} b_j z^j\,.
\ee
For a partition $\mu$, recall that $n(\mu)$ denotes the number of its entries and 
$|\mu|$ denotes the sum of the entries. 
\par
\begin{Thm}\label{thm:indintall}
The  generating function $\cAA$ of the intersection numbers 
$a_i(2g-2)$ is determined by the coefficient extraction identity
\be \label{eq:ind1}
b_j \= [t^{0}] \frac{1}{j!} \cAA(t)^{j} \,,
\ee
while the intersection numbers $a(\mu) = a_i(\mu)$ with $n(\mu) \geq 2$ are 
given recursively by
\begin{flalign} 
& \phantom{\=}
(m_1+1)(m_2+1) a(m_1,\ldots,m_n)  \label{eq:ind2} \\
&\=  \sum_{k=1}^{\min(m_1+1,m_2+1)} \!\! \frac{1}{k!}\, 
\sum_{\bfg, \bfmu}
h_{\proj^1}((m_1,m_2),\bfp) 
\prod_{j=1}^k (2g_j-1+n(\mu_j))\, p_j \, a(p_j-1,\mu_j)\,,
\nonumber
\end{flalign}
with the same summation conventions as in Theorem~\ref{intro:VolRec}.
\end{Thm}
\par
The first identity~\eqref{eq:ind1} was proved in~\cite{SauvagetMinimal}
and gives
\bes
\cAA(t) \= 
\frac{1}{t} - \frac1{24}t + \frac3{640}t^3 - \frac{1525}{580608}t^5
+ \frac{615881}{199065600}t^8 -\cdots\,.
\ees 
By Lagrange inversion, this formula can be written equivalently as
\bes
\cAA(t) \= \frac{1}{Q^{-1}(t)} \,, \quad \text{where} \quad
Q(u) \= u \, \exp\biggl(\sum_{k \geq 1} (k-1)! b_k u^k\biggr)
\ees 
and will in fact be proved in this form in Section~\ref{sec:appvol}. We observe
in passing that $Q(u)$ is the asymptotic expansion of $\psi(u^{-1} + \tfrac12)$
as $u \to 0$, where $\psi(x) = \Gamma'(x)/\Gamma(x)$ is the digamma function.
The proof of the second identity~\eqref{eq:ind2} will be completed by the end 
of Section~\ref{ssec:treestrata}. 
\par
In the course of proving Theorem~\ref{thm:indintall} we will prove the
following complementary result, justifying the
implicitly used fact that $a_i(\mu)$ is independent of~$i$.
\par
\begin{Prop}\label{independent}
For all $1\leq i\leq n$, we have 
\begin{eqnarray*}
a_i(\mu) \= -  \int_{\proj\obarmoduli[g,n](\mu)} \!\!\!\!\!\!\!
  \xi^{2g-2} \cdot \prod_{j=1}^n \psi_j \,.
\end{eqnarray*}
\end{Prop}
\par
\subsection{Boundary components of moduli spaces of Abelian differentials}\label{ssec:boundary}

In Section~\ref{sec:hproj} we introduced several families of stable graphs
to describe the boundary of Hurwitz spaces. Here we show how these graphs
encode relevant parts of the boundary of moduli spaces of Abelian differentials. 
\par
The recursions in Theorem~\ref{intro:VolRec} and
Theorem~\ref{thm:indintall} can be phrased as sums over a small
subset of twisted level graphs, with only two levels and more constraints,
that we call (rational) backbone graphs, inspired by Figure~\ref{cap:BB}.
\par
\begin{figure}[ht]
\begin{tikzpicture}
\tikzstyle{every node}=[font=\scriptsize]

\fill (0,0) circle (1.5pt) node[right] {$X_5$};
\draw[](0,0) -- (45:2.45) node[above] {$X_4$}
	    (0,0) -- (75:1.8) node[above] {$X_3$}
	    (0,0) -- (105:1.8) node[above] {$X_2$}
	    (0,0) -- (135:2.45) node[above] {$X_1$};
\fill[]{(0,0) -- ++ (45:2.45) circle  (1.5pt)
	(0,0) -- (75:1.8) circle  (1.5pt)
	(0,0) -- (105:1.8) circle  (1.5pt)
	(0,0) -- (135:2.45) circle  (1.5pt)};
\draw[] (-2.5,.2) node[] {$\ell =-1$};
\draw[] (-2.5,1.7) node[] {$\ell =0$};

\draw[](3,-.2) -- (3,1.8) node[above] {$X_5$};
\draw[] (2.9,.5) -- (4.5,.5) node[right] {$X_3$}
        (2.9,1) -- (4.5,1)  node[right] {$X_2$};     
\path[draw]{(2.9,1.5) arc (117:90:1.75cm)  -- +(0:.8) node[right] {$X_1$}};
\path[draw]{(2.9,0) arc (-117:-90:1.75cm)  -- +(0:.8) node[right] {$X_4$}};

\end{tikzpicture} 
\caption{A backbone graph and the corresponding stable curve} \label{cap:BB}
\end{figure}
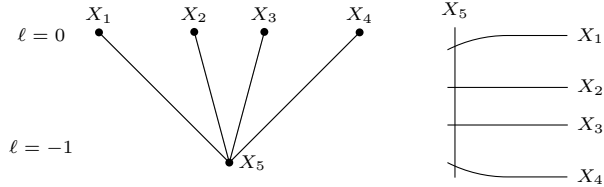
\par
Recall that a bi-colored graph is a level graph with two levels $\{0, -1\}$ 
that has no horizontal edges. 
\par
\begin{Defi}
An  {\em almost backbone graph} is a bi-colored graph with only one vertex at
level~$-1$. For such a graph to be a {\em (rational) backbone graph} we
require moreover that it is of compact type and that the vertex at
level~$-1$ has genus zero.
\end{Defi}
\par
We denote by ${\rm BB}(g,n)\subset {\rm ABB}(g,n) \subset {\rm Bic}(g,n)$
the sets of backbone, almost backbone and bi-colored graphs.  We denote
by ${\rm BB}(g,n)_{1,2}\subset {\rm BB}(g,n)$ the set of backbone graphs such
that the first and second legs are adjacent to the vertex of level $-1$.
Moreover, let $\BB(g,n)^\star_{1,2}\subset {\rm BB}(g,n)_{1,2}$ be the subset
where precisely the first two legs are adjacent to the lower level vertex.
Similarly, we define $\ABB(g,n)_{1,2}$ and $\ABB(g,n)^\star_{1,2}$ and drop $(g,n)$
if there is no source of confusion. 
Below we fix some notations for these graphs, used throughout in the sequel.
\par
For $\Gamma \in {\rm BB}(g,n)$ we denote by $v_{-1}$ the vertex of
level~$-1$. backbone graphs will usually have $k$ vertices of level $0$. 
Given a partition~$\mu = (m_1, \ldots, m_n)$ of $2g-2$, let $\bfp$ be the unique
twist of type $(\mu, \emptyset)$ for $\Gamma$ (see Definition~\ref{def:twist}).
We denote by $\mu[0]_{-1}$ the list of~$m_i$ for all
legs~$i$ at level~$-1$ and with a slight abuse of notation we denote
by $\bfp = (p_1,\ldots,p_k)$ the list of $\bfp(h)$ for half-edges $h$ that are adjacent
to the $k$ vertices of level~$0$.
Said differently, the restriction of the twist to level~$-1$ provides
$v_{-1}$ with a twist of type~$(\mu[0]_{-1}, \mu[\infty]_{-1} = \bfp)$. 
Finally if~$v$ is a vertex of level~$0$, we denote by $\mu_v$ the list
of $\bfp(h)-1$ for all half-edges adjacent to~$v$.
\par
\medskip
The goal in the remainder of the section is to introduce the classes
$\alpha_{\Gamma,\ell,p}$ in~\eqref{eq:defalphaGlI} below that will be used in
Proposition~\ref{prop:indclasses} to compute intersection numbers
on $\proj\obarmoduli[g,n](\mu)$. A stable graph  $\Gamma \in {\rm Stab}(g,n)$
determines the moduli space 
\begin{equation*}
\oM_\Gamma \= \prod_{v \in V(\Gamma)} \oM_{g(v),n(v)}\,
\end{equation*}
and comes with a natural morphism $\zeta_\Gamma\colon \oM_\Gamma \to  \oM_{g,n}$.
Let $\ell$ be a level function on $\Gamma$ such that $(\Gamma,\ell)$ is
a bi-colored graph with two levels $\{0, -1 \}$. We define the following 
vector bundle
$$
\oOmM_{\Gamma,\ell} \= \left(\prod_{v\in V(\Gamma), \ell(v)=0} \oOmM_{g(v),n(v)} \right)
\times  \left(\prod_{v\in V(\Gamma), \ell(v)=-1} \oM_{g(v),n(v)}\right)
$$
over $\oM_{\Gamma}$. This space comes with a natural morphism 
$\zeta_{\Gamma,\ell}^{\#}\,\colon\,\oOmM_{\Gamma,\ell} \to \oOmM_{g,n}$, defined by the 
composition $\oOmM_{\Gamma,\ell}\to \zeta_{\Gamma}^* \left(\oOmM_{g,n}\right) 
\to \oOmM_{g,n}$ where the first arrow is the inclusion of a vector sub-bundle 
and the second is the map on the  Hodge bundles induced from $\zeta_{\Gamma}$
by pull-back. The morphism  $\zeta_{\Gamma,\ell}^{\#}$ determines a 
morphism (denoted by the same symbol) on the projectivized Hodge bundles 
$\zeta_{\Gamma,\ell}^{\#}\,\colon\, \proj\oOmM_{\Gamma,\ell} \to \proj\oOmM_{g,n}\,$. The 
image of $\zeta_{\Gamma,\ell}^{\#}$ is the closure of the locus of differentials supported on curves with dual graph $\Gamma$ such that the differentials vanish identically on components of level~$-1$. In the sequel we will need the following lemma (see~\cite[Proposition~5.9]{SauvagetClass}).
\begin{Lemma}\label{lem:divisible}
 The Poincar\'e-dual class of $\zeta^{\#}_{\Gamma,\ell}(\proj\oOmM_{\Gamma,\ell})$ is divisible by $\xi^{h^1(\Gamma)}$ in $H^*(\proj\obarmoduli[g,n],\QQ)$.
\end{Lemma}
\par
Take a bi-colored graph $(\Gamma, \ell)$ and a partition $\mu$ 
of $2g-2$.
Now we consider a twist~$\bfp$ of type $(\mu[0] = \mu, \mu[\infty] = \emptyset)$
compatible with $\ell$ and
construct a subspace $\proj\oOmM_{\Gamma,\ell}^{\bfp}\subset  \proj\oOmM_{\Gamma,\ell}$
such that $\zeta_{\Gamma,\ell}^{\#}(\proj\oOmM_{\Gamma,\ell}^{\bfp})$ lies in the
boundary of $\proj\obarmoduli[g,n](\mu)$. Let 
\begin{eqnarray} \label{eq:m0m-1}
\OmM_{0} &\subset & \prod_{v\in V(\Gamma), \ell(v)=0} \OmM_{g(v),n(v)}\,, \nonumber\\
\M_{-1}&\subset& \prod_{v\in V(\Gamma), \ell(v)=-1} \mathcal M_{g(v),n(v)} 
\end{eqnarray}
be the loci defined by the following three conditions:
\begin{itemize}
\item[i)] A differential in $\OmM_{0}$ has zeros of orders $m_i$ at the relevant
  marked points and of orders $\bfp(h)-1$ at the relevant branches of the nodes.
\item[ii)] For each $v$ of level $-1$ there exists a non-zero (meromorphic) 
differential~$\omega_v$ on the component $X_v$ corresponding to $v$ that has
zeros at the relevant marked points of orders prescribed by $\mu$ and poles 
at the relevant branches of the nodes of orders prescribed by~$\bfp$, i.e.\
such that the canonical divisor class of $X_v$ is given by
$\sum_{h\in H, a(h)=v} (\bfp(h)-1)x_h$, where $x_h \in X_v$ is the marked point 
or the node corresponding to the half-edge~$h$. 
\item[iii)] There exist complex numbers $k_v\neq 0$ for all vertices $v$ of level $-1$ such that $\omega=\sum_{\ell(v) = -1} k_v \omega_v$ satisfies the global residue condition of~\cite{strata}.  
\end{itemize}
In particular for a backbone graph $\Gamma$, since it is of compact type with a unique vertex $v_{-1}$ of level $-1$, we have the identification 
$\OmM_0=\prod_{v\neq v_{-1}} \OmM_{g(v),n(v)}(\mu_v)$. 
\par
We define $\proj\oOmM_{\Gamma,\ell}^{\bfp}$ as the Zariski closure of
$\proj \OmM_{0} \times \M_{-1}$ in $\proj\oOmM_{\Gamma,\ell}$ and define  
\be \label{eq:defalphaGlI}
 \alpha_{\Gamma,\ell,\bfp} = \left\{ \begin{array}{cl}
   \zeta_{\Gamma,\ell*}^{\#} [\proj\oOmM_{\Gamma,\ell}^{\bfp}] & \text{if }
   \dim(\proj\oOmM_{\Gamma,\ell}^{\bfp})=\dim(\proj\obarmoduli[g,n](\mu))-1 \\
0& \text{otherwise}
\end{array} 
\right.
\ee
as the corresponding class in $H^*(\proj\obarmoduli[g,n],\QQ)$. By~\cite[Proposition~5.9]{SauvagetClass}, we can describe $\alpha_{\Gamma,\ell,\bfp}$ with the following two lemmas.
\par
\begin{Lemma}\label{lem:support}
 The class $\alpha_{\Gamma,\ell,\bfp}$ lies in the subring generated by $\xi$ and $\zeta_{\Gamma*}H^*(\oM_{\Gamma})$.
\end{Lemma}
\par
\begin{Lemma}\label{lem:divisible2}
If  $(\Gamma,\ell,\bfp)$ is a bi-colored graph of compact type, then
$\alpha_{\Gamma,\ell,\bfp}\neq 0$ if and only if there is a unique vertex  $v_{-1}$ of
level $-1$, and in this case $\alpha_{\Gamma,\ell,\bfp}$ is divisible by
$$
\xi^{g(v_{-1})} + \zeta_{\Gamma*}(\lambda_{v_{-1},1}) \xi^{g(v_{-1})-1} +\cdots +  \zeta_{\Gamma*}(\lambda_{v_{-1},g(v_{-1})})\,,
$$
where 
\begin{eqnarray*}
\lambda_{v_{-1},i}=(\lambda_i, 1,\ldots, 1) & \in & H^*(\oM_{g(v_{-1}),n(v_{-1})},\QQ) \!\!\!\bigotimes_{v\in V(\Gamma), v\neq v_{-1}} \!\!\! H^*(\oM_{g(v),n(v)},\QQ) \\
&\simeq & H^*(\oM_{\Gamma},\QQ)\,.
\end{eqnarray*}
\end{Lemma}

\subsection{A first reduction of the computation }\label{ssec:reduction}

Recall the (marked and projectivized) Hodge bundle projection  
$p\colon \proj\obarmoduli[{g,n}] \to \barmoduli[g,n]$.  As before we usually denote by the same symbol a class in  $\barmoduli[g,n]$ and its pullback via $p$.  
In this section we show that many $p$-push forwards of intersections
of $\alpha_{\Gamma,\ell,p}$ with tautological classes vanish or can
be computed recursively. The starting point is the following important
lemma proved by Mumford in~\cite[Equation (5.4)]{mumford83}.
\par
\begin{Lemma} \label{lemMum}
The Segre class of the Hodge bundle is the Chern class
of the dual of the Hodge bundle, i.e.\ 
$$
c_*(\oOmM_{g,n})\cdot c_*(\oOmM_{g,n}^{\vee}) \= 1\,.
$$
In particular, we have $\lambda_g^2=0\in H^{4g}(\oM_{g,n},\QQ)$.
\end{Lemma}
\par
Together with the definition of Segre class, this lemma implies that 
$$
p_*(\xi^{k} \gamma) \= s_{k-g+1}(\oOmM_{g,n}) \gamma \= 
(-1)^{k-g+1} \lambda_{k-g+1} \gamma  
$$
for all $\gamma\in H^*(\oM_{g,n},\QQ)$ and all $k\geq g-1$. Another 
important lemma is the following (see e.g.~\cite[Section 13, 
Equation~(4.31)]{acgh2}).
\par
\begin{Lemma}\label{lem:intlambda} 
Let $1\leq k\leq g$ and let $\Gamma$ be a stable graph. Then 
$$\zeta_{\Gamma}^* \lambda_{k} \=
\sum_{(k_v)_{v\in V} \in \NN^V , \atop |(k_v)|=k} \, \prod_{v\in V} \lambda_{k_v}\,,$$
where the sum is over all partitions of $k$ into non-negative integers $k_v$ 
assigned to each vertex $v \in V = V(\Gamma)$. 
\end{Lemma}
\par
In particular if $h^1(\Gamma)> g-k$, then $\sum_{v\in V}g(v) = g - h^1(\Gamma) < k$, hence the above lemma implies that 
$\zeta_{\Gamma}^* \lambda_{k} = 0$ as there exists some $k_v > g(v)$ for any partition $(k_v)_{v\in V}$ of $k$.  
\par
As a consequence of the above discussion, we obtain the following result.
\begin{Lemma}\label{lem:toplambda}
Let $\alpha=\sum_{i\geq 0}\xi^i \alpha_i$ be a class in $H^*(\proj \oOmM_{g,n},\QQ)$ where the classes $\alpha_i$ are pull-backs from $H^*(\oM_{g,n},\Q)$. Then we have  
\begin{eqnarray*}
p_*(\xi^{2g-1} \alpha)&=&(-1)^{g} \alpha_0 \lambda_g\,, \\
p_*(\xi^{2g-2} \alpha)&=&  (-1)^{g} \alpha_1 \lambda_g +(-1)^{g-1} \alpha_0  \lambda_{g-1}\, ,\\
p_*(\xi^{2g-2} \delta_0\alpha)&=&  (-1)^{g-1} \alpha_0 \delta_0 \lambda_{g-1}\,.
\end{eqnarray*}
\end{Lemma}
\par
Recall the expressions of the intersection numbers $a_i(\mu)$ in~\eqref{eq:a(mu)} and in Proposition~\ref{independent}. 
In order to compute $a_i(\mu)$, by Lemma~\ref{lem:toplambda}
we only need to consider the $\xi$-degree zero and one parts of the class $[\proj\oOmM_{g,n}(\mu)]$ in $H^*(\proj \oOmM_{g,n},\QQ)$.  
\par
Combining Lemmas~\ref{lem:intlambda} and~\ref{lem:toplambda} together with Lemmas~\ref{lem:divisible} and~\ref{lem:divisible2} of the previous section, we can already prove the following vanishing result for classes associated with some bi-colored graphs.
\par
\begin{Prop}\label{pr:intbb}
If $(\Gamma,\ell,\bfp)$ is not a backbone graph, then 
$$p_{*}\left(\xi^{2g-2}\alpha_{\Gamma,\ell,\bfp}\right) \= 0 \in H^*(\oM_{g,n},\QQ)\,
$$
where $\alpha_{\Gamma,\ell,\bfp}$ is defined in~\eqref{eq:defalphaGlI}. 
\end{Prop}
\par
\begin{proof}
For simplicity we write $\alpha = \alpha_{\Gamma,\ell,\bfp}$ in the proof.  
We assume first that $\Gamma$ is not of compact type, i.e. $h^1(\Gamma) > 0$. Then by
Lemma~\ref{lem:divisible}, the class $\alpha $ is divisible
by~$\xi$. Therefore, by Lemma~\ref{lem:support} we can write 
$$\xi^{2g-2}\alpha \=\sum_{i\geq 0} \xi^{2g-1+i} \alpha_i'\,,$$
where $\alpha_i'$ is a pullback from $H^*(\oM_{g,n},\QQ)$ that is supported on $\zeta_\Gamma(\oM_{\Gamma})$ for all $i\geq 0$. Thus by Lemma~\ref{lem:toplambda}, $p_*(\xi^{2g-2}\alpha)= (-1)^{g} \alpha_0' \lambda_g=0$, because $\Gamma$ is not of compact type.
\par
Now we assume that $\Gamma$ is of compact type. By Lemma~\ref{lem:divisible2}, we only need to consider the case when there is a unique vertex $v_1$ of level $-1$. 
Since $\Gamma$ is not a backbone graph, $v_1$ has positive genus $g_1$. Still by Lemma~\ref{lem:divisible2} and simplifying the notation ${\zeta_{\Gamma}}_{*}(\lambda_{v_1,i})$ 
by $\lambda_{v_1,i}$, the class $\alpha$ is divisible by $\xi^{g_1}+ \xi^{g_1-1}\lambda_{v_1,1} + \cdots+ \lambda_{v_1,g_1}$. Consequently we can write 
$$
\alpha \= ( \xi \lambda_{v_1,g_1-1}+\lambda_{v_1,g_1}) \gamma_0 \,+\,
\xi \lambda_{v_1,g_1} \gamma_1 \,+ \,O(\xi^2)\,,
$$
where $\gamma_0$ and $\gamma_1$ are pullbacks from $H^*(\oM_{g,n},\QQ)$
and the $O(\xi^2)$ term stands for a class divisible by $\xi^2$. By
Lemma~\ref{lem:toplambda}, we obtain that 
$$
p_*(\xi^{2g-2}\alpha) \= (-1)^{g} (\lambda_{v_1,g_1-1} \lambda_{g}- \lambda_{v_1,g_1} \lambda_{g-1}) \gamma_0 + (-1)^{g} \lambda_{v_1,g_1} \lambda_{g} \gamma_1\,. 
$$
Using Lemma~\ref{lem:intlambda}, we also obtain that  
\begin{eqnarray*}
\zeta_{\Gamma}^*(\lambda_g) \= \bigotimes_{v \in V} \lambda_{g_v} \quad \text{ and } \quad 
\zeta_{\Gamma}^*(\lambda_{g-1}) \=\sum_{v \in V} \left( \lambda_{g_v-1}  \bigotimes_{v' \neq v} \lambda_{g_{v'}} \right).
\end{eqnarray*}
From the projection formula we deduce that 
$$\lambda_{v_1,g_1}\cdot \lambda_g \= {\zeta_{\Gamma}}_{*} (\lambda_{v_1,g_1}\cdot \zeta^{*}_{\Gamma}(\lambda_g)) \= 
\zeta_{\Gamma_*}\left(\lambda_{g_1}^2 \bigotimes_{v' \neq v_1} \lambda_{g_{v'}} \right) \=0\,,$$
because $\lambda_{g_1}^{2}=0 \in H^*(\oM_{g(v_1),n(v_1)},\QQ)$ by Lemma~\ref{lemMum}.
Once again the same lemma implies that 
\begin{eqnarray*}
\lambda_{v_1,g_1-1}\cdot \lambda_g&=&  \zeta_{\Gamma_*}\left(\lambda_{g_1}\lambda_{g_1-1} \bigotimes_{v' \neq v_1} \lambda_{g_{v'}} \right), \\
\lambda_{v_1,g_1}\cdot \lambda_{g-1}&=&\zeta_{\Gamma_*}\left(\lambda_{g_1}\lambda_{g_1-1} \bigotimes_{v' \neq v_1} \lambda_{g_{v'}} \right) + \sum_{v \neq v_1}\sum_{v'\neq v, v_1}  \left( \lambda_{g_1}^2 \otimes \lambda_{g_v-1} \bigotimes_{v' \neq v} \lambda_{g_{v'}} \right)\\
&=&\lambda_{v_1,g_1-1}\cdot \lambda_g\,.
\end{eqnarray*}
Putting everything together, we thus conclude that $p_*(\xi^{2g-2}\alpha)=0$.
\end{proof}
\par
We define the {\em multiplicity} of a twist $\bfp$ to be
\be \label{eq:multtwist}
m(\bfp) \= \prod_{(h,h')\in E(\Gamma)} \sqrt{-\bfp(h)\bfp(h')}\,.
\ee
\par
\begin{Prop}\label{pr:intbb1}
If $(\Gamma,\ell,\bfp)$ is a backbone graph in  $\BB(g,n)_{1,2}$, then 
\bes
\int_{\proj\oOmM_{g,n}} \!\!\!\!\!\alpha_{\Gamma,\ell,\bfp} \cdot \xi^{2g-1} \cdot
\prod_{i=3}^n \psi_i  \= {m(\bfp)} \cdot  h_{\proj^1}(\mu_{-1},\bfp)
\cdot \prod_{v\in V(\Gamma) \atop \ell(v)=0} a_1({p_v-1,\mu_v})\,,
\ees
where $p_v$ is the entry of $\bfp$ corresponding to the twist on the unique edge
of each vertex~$v$ of level $0$ and $\mu_{-1}$ is the list of entries in $\mu$ whose corresponding legs are adjacent to the vertex of level $-1$. 
\end{Prop}
\par
As a preparation for the proof we relate the space $\M_{-1}$ defined 
in~\eqref{eq:m0m-1} to the Hurwitz space for backbone graphs. The idea 
behind this relation was already mentioned in the last paragraph of 
Section~\ref{subsec:n=2}. If $(\Gamma,\ell)$ is a backbone graph, then we 
claim that $H_{\proj^1}(\mu_{-1},\bfp)\cong \M_{-1}$,
where the isomorphism is provided by the source map $f_S$ that marks the critical points of the branched covers. To verify the claim, let $\omega$ be the meromorphic differential on the unique vertex of $\Gamma$ of level $-1$ as in part iii) of the definition for $\M_{-1}$. Since $\Gamma$ is of compact type, 
the global residue condition in~\cite{strata} imposed to $\omega$ implies that all residues of $\omega$ vanish. Therefore, a point in $\M_{-1}$ can be identified with such a meromorphic differential $\omega$ (up to scale)  
on $\proj^1$ without residues, such that $\omega$ has zeros of order $m_i$ for $m_i \in \mu_{-1}$ at the corresponding markings and has poles of order $p_j+1$ 
for $p_j \in \bfp$ at the corresponding nodes. In particular, $\omega$ is an exact differential and integrating it on $\proj^1$ 
provides a meromorphic function that can be regarded as a branched cover $f$ parameterized in $H_{\proj^1}(\mu_{-1},\bfp)$.  Conversely given 
$f$ in $H_{\proj^1}(\mu_{-1},\bfp)$, we can treat $f$ as a meromorphic function and taking $df$ gives rise to such $\omega$.  We thus conclude that 
 $H_{\proj^1}(\mu_{-1},\bfp)\cong \M_{-1}$. 
Consequently for $\Gamma \in \BB(g,n)_{1,2}$, we have 
\begin{equation}\label{eqn:m1}
h_{\proj^1}(\mu_{-1},\bfp) \= \int_{\BH_{\proj^1}(\mu_{-1},\bfp)} f_S^* \left(\prod_{\begin{smallmatrix} 3\leq i\leq n\\  i\mapsto v_{-1} \end{smallmatrix}}  \psi_i\right) \= \int_{\oM_{-1}} \prod_{\begin{smallmatrix} 3\leq i\leq n\\  i\mapsto v_{-1} \end{smallmatrix}}  \psi_i\,,
\end{equation} 
where $i\mapsto v_{-1}$ means that the $i$-th marking belongs to the vertex of
level $-1$. 
\par
Now we can proceed with the proof of Proposition~\ref{pr:intbb1}. 
\begin{proof}
We write $\alpha_{\Gamma,\ell, \bfp}=\sum_{i\geq 0} \alpha_{\Gamma,\ell, \bfp}^i \xi^i$
where $\alpha_{\Gamma,\ell, \bfp}^i$ is a pull-back from $H^*(\oM_{g,n},\QQ)$.
By Lemma~\ref{lem:toplambda} we deduce that 
\ba \label{eq:pushfwd}
p_*(\xi^{2g-1}\alpha_{\Gamma,\ell, p}) &\= (-1)^{g} \lambda_g \alpha_{\Gamma,\ell, \bfp}^0\,.
\ea 
Therefore we only need to consider the degree zero part of
$\alpha_{\Gamma,\ell, \bfp}$, which is given by
\begin{eqnarray*}
\zeta_{\Gamma*}\left( [\oM_{-1}] \bigotimes_{v \in V(\Gamma), \ell(v)=0} [\proj \oOmM_{g_v,n_v}(p_v-1,\mu_v)]^0  \right),
\end{eqnarray*} 
where 
$[\proj \oOmM_{g_v,n_v}(p_v-1,\mu_v)]^0$ is the degree zero part of the
Poincar\'e-dual class of $\proj \oOmM_{g_v,n_v}(p_v-1,\mu_v)$ in
$\proj \oOmM_{g_v,n_v}$. Therefore, we have
\begin{eqnarray*}
\lambda_g \cdot \alpha_{\Gamma,\ell,\bfp}^0 
\= \zeta_{\Gamma*}\left( [\oM_{-1}] \bigotimes_{v \in V(\Gamma), \ell(v)=0} 
\left( \lambda_{g_v} \cdot [\proj \oOmM_{g_v,n_v}(p_v-1,\mu_v)]^0 \right)\right)\,.
\end{eqnarray*} 
Multiplying this expression by $\prod_{i=3}^n \psi_i$, we obtain that 
\begin{eqnarray*}
\lambda_g  \cdot \alpha_{\Gamma,\ell,\bfp}^0\cdot \prod_{i=3}^n \psi_i  \=&& \!\!\!\!\! \!\!\!\!\!  \zeta_{\Gamma*}\Bigg( \Big([\oM_{-1}] \cdot\prod_{i \mapsto v_{-1}, i\geq 3} \psi_i \Big) \\ && \bigotimes_{v \in V(\Gamma), \ell(v)=0} \left( \lambda_{g_v} \cdot [\proj \oOmM_{g_v,n_v}(p_v-1,\mu_v)]^0  \cdot \prod_{i \mapsto v} \psi_i \right) \Bigg)\,.
\end{eqnarray*}
For the first term on the right-hand side, equality~\eqref{eqn:m1} implies that 
$$[\oM_{-1}] \cdot \prod_{i \mapsto v_{-1}, i\geq 3}\psi_i
\= h_{\proj^1}(\mu_{-1}, \bfp)\,.$$ 
Moreover for all $v$ of level $0$, we have
\begin{eqnarray*}
{p_v \,a_1(p_v-1,\mu_v)}&=& \int_{\proj \oOmM_{g_v,n_v}(p_v-1,\mu_v)}  \!\!\!\!\! \!\!\!\!\!  \xi^{2g_v-1} \prod_{i \mapsto v} \psi_i \\
&=& (-1)^{g_v} \int_{\oM_{g_v,n_v}} \!\!\!\!\!   \lambda_{g_v} \cdot \prod_{i \mapsto v} \psi_i\cdot [\proj \oOmM_{g_v,n_v}(p_v-1,\mu_v)]^0\,
\end{eqnarray*}
where the second identity follows from Lemma~\ref{lem:toplambda}. Since $p_v$ is the (positive) twist value assigned to the edge of $v$, the product of $p_v$ over all vertices of level $0$ equals    
$m(\bfp)$ defined in~\eqref{eq:multtwist}. In addition, the sum of $g_v$ over all vertices of level $0$ equals the total genus $g$, because $\Gamma$ 
is of compact type and $v_{-1}$ has genus zero.  Putting everything together we thus obtain that
$$
\int_{\oM_{g,n}}  \!\!\!\!\!  \lambda_g \cdot \alpha_{\Gamma,\ell,\bfp}^0 \cdot \prod_{i=3}^n \psi_i= m(\bfp)\cdot {h_{\proj^1}(\mu_{-1}, \bfp)}\cdot\!\!\!\!\!\!\!  \!\!\!\!\!  \prod_{v\in V(\Gamma),\ell(v)=0}\!\!\!\!\! a_1({p_v-1,\mu_v})
\,,
$$
which is the desired statement.
\end{proof}
\par

\subsection{The induction formula for cohomology classes}\label{ssec:induction}
\label{sec:indCC}

The main tool of the section is the induction formula 
in~\cite[Theorem~6 (1)]{SauvagetClass} which we recall now.
\par
\begin{Prop}\label{prop:indclasses}
For all $1\leq i\leq n$, the relation that 
\be \label{for:indclasses}
\left(\xi+(m_i+1)\psi_i\right)[\proj\obarmoduli[g,n](\mu)]
\= \sum_{\begin{smallmatrix}(\Gamma,\ell,p)\\ i\mapsto v, \ell(v)=-1\end{smallmatrix}}
\frac{m(\bfp)}{|{\rm Aut}(\Gamma,\ell,\bfp)|}  \alpha_{\Gamma,\ell,\bfp}
\ee
holds in $H^*(\proj\obarmoduli[g,n],\QQ)$, where the sum is over all
twisted bi-colored graphs such that the $i$-th leg is carried by a
vertex of level $-1$.
\end{Prop}
There are two ways of using equation~\eqref{for:indclasses}.
First one can compute the Poincar\'e-dual class of $\proj\obarmoduli[g,n](\mu)$ in $H^*(\proj\obarmoduli[g,n],\QQ)$ in terms of the $\psi$, $\lambda$, $\xi$ classes and boundary classes associated to stable graphs. This strategy is used in~\cite{SauvagetMinimal} to deduce the first formula in Theorem~\ref{thm:indintall}.
\par
Alternatively, one can compute relations in the Picard group
of $\proj\obarmoduli[g,n](\mu)$ to deduce relations between intersection numbers on $\proj\obarmoduli[g,n](\mu)$ and intersection numbers on boundary strata associated to twisted graphs. This is the strategy that we will use here. We will use this proposition with $i \in \{1,2\}$ and multiply the formula by $\xi^{2g-1} \prod_{i=3}^n \psi_i$ 
to obtain $a_1(\mu)$  
on the left-hand side. Then we will use
Propositions~\ref{pr:intbb} and~\ref{pr:intbb1} to compute the right-hand
side. A first application of this strategy gives a proof of the
complementary proposition. 
\par
\begin{proof}[Proof of Proposition~\ref{independent}] 
We use Proposition~\ref{prop:indclasses} with $i=1$. Multiplying formula~\eqref{for:indclasses} by $\xi^{2g-2}\cdot \prod_{i=2}^n \psi_i$, we obtain that 
\begin{eqnarray*}
(m_1+1) \left(a_1(\mu)+ \int_{\proj\obarmoduli[g,n](\mu)} \!\!\!\!\!\!\!  \xi^{2g-2} \cdot \prod_{i=1}^n \psi_i\right) &=&\\
 \sum_{\begin{smallmatrix}(\Gamma,\ell,\bfp)\\ 1\mapsto v, \ell(v)=-1\end{smallmatrix}} \!\!\!\!\!\! \frac{m(\bfp)}{|{\rm Aut}(\Gamma,\ell,\bfp)|} \!\!\!\!\!\! && \!\!\!\!\!\! \int_{\proj\obarmoduli[g,n]} \!\!\!\!\!\! \alpha_{\Gamma,\ell,\bfp} \cdot \xi^{2g-2} \cdot \prod_{i=2}^n \psi_i\,.
\end{eqnarray*}
It suffices to check that each summand in the right-hand side vanishes.
Proposition~\ref{pr:intbb} implies that if $(\Gamma,\ell,\bfp)$ is not a
backbone graph, then the corresponding summand vanishes. If $(\Gamma,\ell,\bfp)$
is a backbone graph, then we have seen (in the paragraph below Proposition~\ref{pr:intbb1}) that $\oM_{-1}$ is birational to a Hurwitz space of admissible covers of dimension $n_{-1}-2$,
where $n_{-1}$ is the number of legs adjacent to the vertex of level~$-1$. Since the $\psi$-product restricted to level $-1$ contains $n_{-1}-1$ terms (i.e. it misses $\psi_1$ only), which is bigger than $\dim \oM_{-1}$, it implies that the intersection of $\alpha_{\Gamma,\ell,\bfp}$ with $\prod_{i=2}^n\psi_i$
vanishes.  
\end{proof}
\par
Now we know that $a_i(\mu)$ is independent of the choice
of $1\leq i\leq n$ and hence we can drop the subscript~$i$.
The second use of the strategy presented above leads to the following
induction formula.
\par
\begin{Lemma}\label{lem:indbis}
The intersection numbers $a(\mu)$ satisfy the recursion 
\begin{eqnarray*}
  (m_1+1)(m_2+1) a(\mu)  \=
\sum_{ k \geq 1}
\sum_{\bfg, \bfmu} 
h_{\PP^1}((m_1,m_2,\mu_0),\bfp) \cdot \frac{1}{k!}
\cdot \prod_{i=1}^k p_i^2 a(p_i-1,\mu_i)\,
\end{eqnarray*}
where $\bfg=(g_1,\ldots,g_k)$ is a partition of $g$, where
now ~$\bfmu = (\mu_0, \mu_1,\ldots,\mu_k)$ is a $(k+1)$-tuple of multisets with 
$(m_3,\ldots, m_n)= \mu_0\sqcup \cdots \sqcup  \mu_k$ and 
where $\bfp = (p_1, \ldots, p_k)$ has entries $p_i = 2g_i - 1 - |\mu_i| > 0$.
\end{Lemma}
\par
We remark that this induction formula is not quite the same as the induction formula
of Theorem~\ref{thm:indintall}, e.g.\ the sums in the two formulas do not run over the
same set. Theorem~\ref{thm:indintall} will follow further from a combination
of Lemma~\ref{lem:indbis} and Proposition~\ref{prop:intpsiPP1} of the previous section.
\begin{proof}
We apply the induction formula of Proposition~\ref{prop:indclasses} with $i=2$:
$$
\left(\xi+(m_2+1)\psi_2\right)[\proj\obarmoduli[g,n](\mu)]= \sum_{\begin{smallmatrix}(\Gamma,\ell,\bfp)\\ 2\mapsto v, \ell(v)=-1\end{smallmatrix}}\frac{m(\bfp)}{|{\rm Aut}(\Gamma,\ell,\bfp)|}  \alpha_{\Gamma,\ell,\bfp}\,.
$$
We multiply this expression by $\xi^{2g-1}\prod_{i=3}^n \psi_i$ and apply $p_*$.
Since Lemma~\ref{lemMum} gives $p_*(\xi^{2g}  [\proj\obarmoduli[g,n](\mu)])=0$,
the above equality implies that  
$$
(m_1+1)(m_2+1)a(\mu) \= \sum_{\begin{smallmatrix}(\Gamma,\ell,\bfp)\\ 2\mapsto v, \ell(v)=-1\end{smallmatrix}}\frac{m(\bfp)}{|{\rm Aut}(\Gamma,\ell,\bfp)|} p_*\left(\xi^{2g-1} \cdot \prod_{i=3}^n \psi_i \cdot \alpha_{\Gamma,\ell,\bfp}\right)\,.
$$
By Proposition~\ref{pr:intbb} a term in the sum of the
right-hand side vanishes if $(\Gamma,\ell,p)$ is not a backbone graph. Suppose
$(\Gamma,\ell,\bfp)$ is a backbone graph such that 
the first leg does not belong to the vertex of level $-1$ (which contains $n_{-1}$ legs). Then on level $-1$ the product of $\psi$-classes contains $n_{-1}-1$ terms (i.e.\ this product misses $\psi_2$ only), which exceeds the dimension of $\oM_{-1}$ (being $n_{-1}-2$), hence the corresponding term in the sum also vanishes. 
\par
Now we only need to consider the case when $(\Gamma, \ell,\bfp)$ is a backbone graph
in ${\rm BB}(g,n)_{1,2}$, i.e.\ the vertex of level $-1$ carries both the first and
second legs.   Then the intersection number
$\xi^{2g-1}\cdot\prod_{i=3}^n \psi_i \cdot \alpha_{\Gamma,\ell,\bfp}$ is given
by Proposition~\ref{pr:intbb1}. We thus conclude that  
\begin{eqnarray*}
 (m_1+1)(m_2+1) a(\mu)  &\=&  \!\!\!\!\!\!
\sum_{(\Gamma, \ell, \bfp)\in {\rm BB}_{1,2}}  \frac{h_{\proj^1}(\mu_{-1}, \bfp) \,m(\bfp)^2 }{|{\rm Aut}(\Gamma,\ell,\bfp)|}\\
& & \quad  \cdot  \prod_{v\in V(\Gamma),\ell(v)=0}\!\!\!\!\! a({p_v-1,\mu_v})\\
&\=&  
\!\!\!\!\!  \sum_{k \geq 1} \sum_{\bfg, \bfmu}
m(\bfp)^2 h_{\PP^1}((m_1,m_2,\mu_0),\bfp) \cdot \frac{1}{k!}
\cdot \prod_{i=1}^k  a(p_i-1,\mu_i)\,.
\end{eqnarray*}
The last equality comes from the fact that the datum of $g_1+\cdots +g_k=g, \,g_i \geq 1$ and  $(m_3,\ldots, m_n)= \mu_0 \sqcup \mu_1\sqcup \cdots \sqcup  \mu_k$ determines uniquely a graph $(\Gamma, \ell,\bfp)$ in ${\rm BB}(g,n)_{1,2}$ and an automorphism of the backbone graph is determined by a permutation in $S_k$ that preserves both the partition of $g$ and the sets $\mu_1,\ldots,\mu_k$. 
\end{proof}

\subsection{Sums over rooted trees}\label{ssec:treestrata}

The purpose of this section is to combine the preceding Lemma~\ref{lem:indbis}
with Proposition~\ref{prop:intpsiPP1} that describes the computation of
intersection numbers on Hurwitz spaces. We will show that the numbers
$a(\mu)$ can be expressed as sums over rooted trees in a similar way
as we did for intersection numbers on Hurwitz spaces in Section~\ref{sec:HurRT}.
\par
Let $2\leq i\leq n$ and $(\Gamma,\ell,\bfp)$ be a rooted tree in ${\rm RT}(g,\mu)_{1,i}$
(here $\mu[\infty]$ is empty). Since there is no marked pole, it implies that 
any vertex of genus zero has at least one edge with a negative twist, hence it is an internal vertex of $\Gamma$ and lies on a negative level.  
Denote by $\mu[\infty]_0$ the list obtained by taking the (positive) entries $\bfp(h)$
for all half-edges~$h$ adjacent to a vertex of level~$0$.   
Denote by $\mu[0]_0$ the list of entries of $\mu$ from those legs carried by the internal vertices (of genus zero). 
With this notation we define the rooted tree $(\Gamma_0,\ell_0, \bfp_{0})$ in ${\rm RT}(0,\mu[0]_0,\mu[\infty]_0)_{1,i}$ obtained by removing the leaves of $\Gamma$ (i.e. vertices of positive genus and hence on level $0$). We also define the multiplicity $m_0(\bfp)$ of $(\Gamma_0,\ell_0, \bfp_{0})$ to be the product of entries of $\mu[\infty]_0$. Now we define the $a$-contribution of the rooted tree $(\Gamma,\ell,\bfp)$ as
\begin{eqnarray}\label{eq:a-contribution}
a(\Gamma,\ell,\bfp)&=& m_0(\bfp)^2 h(\Gamma_0,\ell_0,\bfp_{0})
\prod_{v\in V(\Gamma),\ell(v)=0}\!\!\!\!\! a({p_v-1,\mu_v})\,,
\end{eqnarray}
where $h(\Gamma_0,\ell_0,\bfp_{0})$ is the contribution of the rooted tree defined in~\eqref{eq:h(gamma)}.   
\par
\begin{Lemma}\label{lem:inttree} The following equality holds: 
$$
(m_1+1)(m_2+1) a(\mu) \= \sum_{(\Gamma,\ell, \bfp)\in {\rm RT}(g,\mu)_{1,2}} \frac{a(\Gamma,\ell,\bfp)}{|{\rm Aut}(\Gamma,\ell,\bfp)|}\,.
$$
\end{Lemma}
\par
\begin{proof}
Removing the leaves of a rooted tree induces a bijection
between ${\rm RT}(g,\mu)_{1,2}$ and the set 
$$\bigcup_{(\Gamma,\ell, \bfp) \in {\rm BB}_{1,2}} {\rm RT}(0,\mu[0]_0,\mu[\infty]_0)_{1,2}
$$
which is a partition of ${\rm RT}(g,\mu)_{1,2}$ over all possible decorations of the leaves of the rooted trees. Moreover, an automorphism of a rooted tree in ${\rm RT}(g,\mu)_{1,2}$ is determined by an automorphism of the backbone graph in ${\rm BB}(g,n)_{1,2}$, because all internal vertices of the rooted tree (i.e. those of genus zero and hence on negative levels) have marked legs by Definition~\ref{def:RT}. 
Then we can first use Lemma~\ref{lem:indbis} to write $a(\mu)$ as a sum over backbone graphs in ${\rm BB}(g,n)_{1,2}$ and then use Proposition~\ref{prop:intpsiPP1} to express it as the desired sum over the set
$$\bigcup_{(\Gamma,\ell, \bfp) \in {\rm BB}_{1,2}}
{\rm RT}(0,\mu[0]_0,\mu[\infty]_0)_{1,2}\,\simeq\, {\rm RT}(g,\mu)_{1,2}$$
as claimed in the lemma.
\end{proof}
\par
We define 
$$
{\rm RT}(g,\mu)_1= \{ \text{trivial graph}\} \,\cup \,
\bigcup_{i=2}^n {\rm RT}(g,\mu)_{1,i}$$
and the $a$-contribution of the trivial graph~$\bullet$ as
$a(\bullet,\ell, \bfp) = (m_1+1)^2 \, a(\mu)$.
\par
\begin{proof}[End of the proof of Theorem~\ref{thm:indintall}]
We will prove for $n\geq 2$ the equality that 
\begin{align}
\sum_{k \geq 1} \sum_{\bfg, \bfmu}
h_{\proj^1}((m_1,m_2),\bfp)\, \frac{1}{k!}\,&
\prod_{j=1}^k (2g_j-1+n(\mu_j)) p_j a(p_j-1,\mu_j) \nonumber \\
\label{eq:ind3}
 =& \!\!\!\!\!\!\sum_{(\Gamma,\ell, \bfp)\in {\rm RT}(g,\mu)_{1,2}} \!\!\!\!\! \frac{a(\Gamma,\ell,\bfp)}{|{\rm Aut}(\Gamma,\ell,\bfp)|} \,.
\end{align}
This formula together with Lemma~\ref{lem:inttree} thus 
implies Theorem~\ref{thm:indintall}. Since by definition
$\sum_{i=1}^n (m_i+1) = 2g-2 + n$, 
Lemma~\ref{lem:inttree} implies that
\begin{eqnarray*}
(2g-2+n) (m_1+1) a(\mu) &=& (m_1+1)^2   a(\mu)+  \sum_{i=2}^n (m_i+1)(m_1+1) a(\mu)\\
&=& (m_1+1)^2   a(\mu) + \sum_{i=2}^n \sum_{\begin{smallmatrix}  (\Gamma,\ell,\bfp) \in \\ {\rm RT}(g,\mu)_{1,i} \end{smallmatrix}} \frac{a(\Gamma,\ell,\bfp)}{|{\rm Aut}(\Gamma,\ell,\bfp)|}\\
&=&\sum_{\begin{smallmatrix}  (\Gamma,\ell,\bfp) \in \\ {\rm RT}(g,\mu)_{1}\end{smallmatrix}} \frac{a(\Gamma,\ell,\bfp)}{|{\rm Aut}(\Gamma,\ell,\bfp)|}\,.
\end{eqnarray*}
Therefore, the left-hand side of~\eqref{eq:ind3} can be rewritten as
\begin{eqnarray*}
&& \sum_{k \geq 1}\sum_{\bfg, \bfmu}
h_{\proj^1}((m_1,m_2),\bfp) \frac{1}{k!} \, \cdot\, 
\prod_{j=1}^k \! \sum_{\begin{smallmatrix}  (\Gamma,\ell,\bfp) \in \\
{\rm RT}(g_j,(p_j-1,\mu_j))_{1} \end{smallmatrix}}
\frac{a(\Gamma,\ell,\bfp)}{|{\rm Aut}(\Gamma,\ell,\bfp)|}\\
&=& \sum_{k \geq 1} \sum_{\bfg, \bfmu}
h_{\proj^1}((m_1,m_2),\bfp) \,\frac{1}{k!}\, \cdot\!\!
\sum_{\begin{smallmatrix}  (\Gamma_j,\ell_j,\bfp_j) \in \\
{\rm RT}(g_j,(p_j-1,\mu_j))_{1} \end{smallmatrix}}   \prod_{j=1}^k \frac{a(\Gamma_j,\ell_j,\bfp_j)}{|{\rm Aut}(\Gamma_j,\ell_j,\bfp_j)|}\,.
 \end{eqnarray*}
\par
We claim that there is a bijection
$$
{\rm RT}(g,\mu)_{1,2}\,\simeq\,  \bigcup_{(\Gamma',\ell',\bfp')\in \BB^\star_{1,2} }
\prod_{v\in V(\Gamma'),\ell(v)=0} {\rm RT}(g_v,(p_v-1,\mu_v))_{1}\,.
$$
Indeed given a rooted tree $(\Gamma,\ell,\bfp)$ in ${\rm RT}(g,\mu)_{1,2}$ we can
construct $(\Gamma',\ell',\bfp') \in \BB(g,n)^\star_{1,2}$ by contracting all edges
except those adjacent to the root, and the rooted trees $ (\Gamma_v,\ell_v,
\bfp_v)\in {\rm RT}(g_v,(p_v-1,\mu_v))_{1}$ are the connected components of the graph
obtained from $(\Gamma,\ell,\bfp)$ by deleting the root. Moreover for a rooted
tree $(\Gamma,\ell,\bfp)$, by equation~\eqref{eq:h(gamma)} we have 
$$
h(\Gamma_0, \ell_0, \bfp_0) = h_{\proj^1}((m_1,m_2), \bfp') \cdot \prod_{j=1}^k h(\Gamma_{j0}, \ell_{j0}, \bfp_{j0})\,,
$$
where as before $\Gamma_0$ is obtained from $\Gamma$ by removing the leaves and the $\Gamma_{j0}$ are the connected components after removing the root of $\Gamma_0$. Together with the definition of the $a$-contribution in~\eqref{eq:a-contribution}, it implies that 
$$
a(\Gamma,\ell, \bfp) = h_{\proj^1}((m_1, m_2), \bfp') \cdot \prod_{v\in V(\Gamma'),\ell(v)=0} a(\Gamma_v,\ell_v,\bfp_v)\,.
$$
Note also that 
$$
{\rm Aut}(\Gamma,\ell, \bfp) =  {\rm Aut}(\Gamma',\ell',\bfp')\times \prod_{v\in V(\Gamma'),\ell(v)=0}  {\rm Aut}(\Gamma_v,\ell_v, \bfp_v)\,.
$$
Combining the above we thus conclude that equality~\eqref{eq:ind3} holds.
\end{proof}
\par
\begin{proof}[Proof of the equivalence of Theorems~\ref{intro:IntFormula} and~\ref{intro:VolRec}]
We first assume that Theorem~\ref{intro:VolRec} holds. By Theorem~\ref{thm:indintall}, the
quantities 
\bes
{\rm vol}(\omoduli[g,n](m_1,\ldots,m_n)) \quad \text{and} \quad 
\frac{2(2\pi i)^{2g}}{(2g-3+n)!}\, a(m_1,\ldots,m_n)
\ees
satisfy the same induction relation that determines both collections of these 
numbers starting from the case $n=1$. The base case (i.e. the minimal strata) that 
\be \label{eq:volminviaa}
{\rm vol}(\omoduli[g,1](2g-2)) \=\frac{2(2\pi i)^{2g}}{(2g-2)!}\, a(2g-2)\,
\ee
was proved 
in~\cite{SauvagetMinimal} under a mild assumption of regularity of a 
natural Hermitian metric on $\mathcal{O}(-1)$, and we will give an alternative 
(unconditional) proof in Section~\ref{sec:appvol}. Consequently we conclude that 
Theorem~\ref{intro:VolRec} implies Theorem~\ref{intro:IntFormula}. 
The converse implication follows similarly. 
\end{proof}


\section{Volume recursion via $q$-brackets}
\label{sec:D2rec}

In this section we define recursively polynomials in the ring 
$R = \QQ[h_1,h_2,\ldots]$ and show that they compute volumes of the strata 
after a suitable specialization. The method of proof relies on lifting 
the $E_2$-derivative via the Bloch-Okounkov $q$-bracket and expressing 
cumulants in terms of this lift. This recursion looks quite different from 
the recursion given in Theorem~\ref{intro:VolRec}, since it is only defined 
on the level of polynomials in the variables $h_i$ and requires
$h_i$-derivatives. 
\par
To define the substitution, we let
\be \label{eq:defalpha}
P_B(u) \= \exp\Bigl(- \sum_{ j \geq 1} j!\, b_{j+1} u^{j+1}\Bigr)
\quad \text{and}  \quad
\alpha_\ell \= \bigl[u^{\ell}\bigr] \frac{1}{(u/P_B(u))^{-1}}\,,
\ee
where the denominator denotes the inverse function of $u/P_B(u)$. For the 
recursion we define for a finite set  $I=\{i_1,\ldots,i_n\}$ of positive 
integers the formal series $\mathcal{H}_I\in R[[z_{i_1}, \ldots, z_{i_n}]]$
if $|I| \geq 2$ and $\mathcal{H}_{\{i\}} \in \frac{1}{z_i} R[[z_{i}]]$ by
\begin{flalign}
&\mathcal{H}_{\{i\}} \= \frac{1}{z_i} + \sum_{\ell\geq 1} h_\ell z_i^\ell\,, \nonumber \\ 
&\cHH_{\{i,j\}} \= \frac{z_i \cHH'(z_i) - z_j\cHH'(z_j)}
{\cHH(z_j) - \cHH(z_i)} -1 \nonumber  \\
&\,\, \= 2h_1 z_iz_j + h_2 (3z_i^2 z_j  + 3z_i z_j^2) + 
4h_3 z_i^3z_j  + (2h_1^2 + 4h_3 )z_i^2z_j^2 + 4h_3 z_iz_j^3  + \cdots\,, \nonumber \\
&\mathcal{H}_{I} = \frac{1}{2(n-1)}\sum_{I=I'\sqcup I'' \atop I', I'' \neq \emptyset}
 D_2(\mathcal{H}_{I'},\mathcal{H}_{I''})\,, \label{eq:Hindform}
\end{flalign}
where we abbreviate $\mathcal{H}_{n}=\mathcal{H}_{[\![1,n]\!]}$, $\cHH = \cHH_1$
and  $h_{\ell_1,\ldots,\ell_n} = [z_1^{\ell_1}\dots z_n^{\ell_n}]\mathcal{H}_{n}$
and where the symmetric bi-differential operator $D_2$ is defined by
\be \label{eq:D2H}
D_2(f,g) \= \sum_{\ell_1,\ell_2 \geq 1}
h_{\ell_1,\ell_2} \,
\frac{\partial f}{\partial h_{\ell_1}}\,  \frac{\partial g}{\partial h_{\ell_2}} \,.
\ee
\par
\begin{Thm} \label{thm:D2recursion}
The rescaled volume of the stratum with signature $\mu=(m_1,\ldots,m_n)$ can be computed as 
\bes
v(\mu) \=  \frac{(2\pi i)^{2g}}{(2g-2+n)! } \, \bigl.
h_{m_1+1,\ldots,m_n+1}\bigr|_{h_\ell \mapsto \alpha_\ell} 
\ees
using the recursion~\eqref{eq:Hindform} and the values of the $\alpha_\ell$ in~\eqref{eq:defalpha}.
\end{Thm}
\par

\subsection{Three sets of generators for the algebra of 
shifted symmetric functions}

We let $\Lambda^*$ be the algebra of shifted symmetric functions
(see e.g.\ \cite{eo}, \cite{zagBO} or \cite{cmz}) and recall the standard 
generators  
\be \label{eq:defpk}
p_\ell(\l) \= \sum_{i=1}^\infty \left( (\l_i -i +\h)^\ell - (-i + \h)^\ell \right)
\+ (1-2^{-\ell})\,\zeta(-\ell) 
\,. 
\ee
Note that $(1-2^{-\ell})\,\zeta(-\ell)\= \ell! b_{\ell+1}$. The algebra $\Lambda^*$ 
is provided with a grading where each~$p_\ell$ has weight $\ell+1$. For Hurwitz 
numbers the geometrically interesting generators are
\be \label{eq:defssf}
f_{\ell}(\lambda) \= z_{\ell} \chi^\lambda(\ell)/\dim \chi^\lambda\,,
\ee
where $z_{\ell}$ is the size of the conjugacy class of the cycle
of length~$\ell$, completed by singletons. The first few of these functions 
are
\bes f_1 \= p_{1} + \frac{1}{24}, \quad f_2   \=  \frac{1}{2} p_{2},
\quad f_3  \= \frac{1}{3} p_{3} - \frac{1}{2} p_{1}^2
 + \frac{3}{8} p_{1} + \frac{9}{640}\,.
\ees 
The third set of generators, defined implicitly by Eskin and Okounkov, will 
serve as top term approximations of $f_\ell$. We define $h_\ell \in \Lambda^*$
by
\be \label{eq:hldef} 
h_\ell \= \frac{-1}{\ell}  [u^{\ell+1}] P(u)^\ell 
\quad \text{where} \quad
P(u) \= \exp \Bigl( -\sum_{s \geq 1} u^{s+1} p_s \Bigr)\,.
\ee
Observe that by definition~$h_\ell$ has pure weight~$\ell+1$. The first few of 
these functions  are
\bes h_1  \= p_{1}, \quad h_2   \= p_{2},
\quad h_3  \=  p_{3} - \frac{3}{2} p_{1}^2 \,.
\ees
\par 
\begin{Prop}[{\cite[Theorem~5.5]{eo}}] \label{prop:EOconv}
The difference $f_\ell -  h_\ell/\ell$ has weight strictly 
less than~$\ell+1$.
\end{Prop}
\par
We abuse the notation $h_\ell$ for generators of $R$ and for elements 
in $\Lambda^*$. This is intentional and should not lead to confusion, since
the map $h_\ell \mapsto h_\ell$ induces an isomorphism of algebras $R \cong \Lambda^*$, 
by the preceding proposition.
\par
\subsection{The lift of the evaluation map to the Bloch-Okounkov ring}
\label{sec:lift}

Let $f\colon\P\to\QQ$ be an arbitrary function on the set $\P$ of all partitions.
Bloch and Okounkov (\cite{blochokounkov}) associated to~$f$ the formal power series
\be \label{defqbrac} 
  \sbq f \= \frac{\sum_{\l\in\P} f(\l)\,q^{|\l|}}{\sum_{\l\in\P} q^{|\l|}}\;\,\in\;\QQ[[q]]\,,
\ee
which we call the {\em $q$-bracket}, and proved that this $q$-bracket is 
a quasimodular form of weight~$k$ whenever~$f$ belongs to the subspace 
of ~$\Lambda^*$ of weight~$k$ (see \cite{blochokounkov}, and \cite{zagBO} 
or \cite{GoujM} for alternative proofs).  
\par
In \cite[Section~8]{cmz} we studied in detail an evaluation map $\evX$ (implicitly defined
in \cite{eo}) on the ring of quasimodular forms that measures the growth rate
of the coefficients of quasimodular forms. The purpose of this section is
to lift this evaluation map to the Bloch-Okounkov ring and to express it in
terms of the generators~$h_i$ introduced in the previous section.
\par
The map $\evX$ is the algebra homomorphism from the ring of quasimodular forms
$\wM_* = \QQ[E_2,E_4,E_6]$ to~$\QQ[X]$, sending the 
Eisenstein series $E_2$ (normalized to have constant coefficient one) 
to~$X+12$, $E_4$ to~$X^2$, and $E_6$ to~$X^3$. 
In this way, the larger the degree of $\evX(f)$, the larger the (polynomial)
growth of the coefficients of~$f$, see \cite[Proposition~9.4]{cmz} for the
precise statement. It is also convenient to work with the evaluation map\footnote{This is the $h$-bracket from \cite{cmz}, but without
the factor $2\pi i$.}
\be\label{defevh} \evh[F](\hslash) \= \frac1{\hslash^{k/2}}\,
\evX[F] \Bigr|_{X \mapsto \frac{1}\hslash}
\;\in\;\QQ[1/\hslash]\,\quad \text{for} \quad F \in \wM_k\,. \ee
We also use  the brackets $\sbrX f:=\evX\bigl[\sbq f\bigr](X)$ and
$\sbrh f:=\evh\bigl[\sbq f\bigr](\hslash)$ for $f\in\Lambda^*$ as abbreviation. 
Note that~$\Lambda^*$ admits a natural ring homomorphism to~$\Q$, 
the evaluation at the emptyset, explicitly given by the 
map $p_\ell \mapsto \ell! b_{\ell+1}$.
\par
\begin{Prop} \label{prop:liftviadelta}
  There is a second order differential operator $\Delta\colon \Lambda^* 
\to \Lambda^*$ of degree~$-2$ and a derivation $\partial\colon\Lambda^* 
\to \Lambda^*$ of degree~$-1$ such that
\be
\brh f \= \frac{1}{\hslash^k}(e^{\hslash(\Delta - \p^2 - \p/\p p_1)/2} f)\,(\emptyset)\,
\ee
for $f \in \Lambda^*_k$ homogeneous of weight~$k$. 
\end{Prop}
The differential operators are given in terms of the 
generators~$p_\ell$ by
\be
\partial(f) \= \sum_{i \geq 2} ip_{i-1} \frac{\partial}{\partial p_i}
\quad \text{and}\quad
\Delta(f) \= \sum_{k,\,\ell\,\ge\,1} (k+\ell) \,p_{k+\ell-1}\,\frac{\p^2}{\p p_{k}\,\p p_{\ell}} \,.
\ee
\par
\begin{proof}  From the definition and \cite[Proposition~9.2]{cmz}
we deduce that the evaluation map can be computed for any $F\in\wM_k$ as
\be \label{eq:defevh}
\evh[F](\hslash) \= \frac1{\hslash^{k}}a_0\bigl(e^{\hslash\fd}  F\bigr) 
\ee
where $\fd  = 12 \partial/\partial E_2$ and where $a_0\colon F\mapsto F(\infty)$ 
is the constant term map from $\wM_*$ to~$\Q$. From \cite[Proposition~8.3]{cmz}
we deduce (note that differentiation with respect to~$Q_i$ in loc.\ cit.\
gives the extra $p_1$-derivative here) that the differential operators
defined above have the property that 
\be\label{E2deriv}  \fd\,\sbq f \= \bq{\h\bigl(\Delta\,-\,\pp^2
-  \p/\p p_1\bigr)\, f} \qquad (f \in \Lambda^*)\,.\ee
Since the constant term of the $q$-bracket of $f$ is in $\Lambda^*$, the
claim follows from these two equations.
\end{proof}
\par
To motivate the next section, we recall the notion of cumulants. 
Let $R$ and $R'$ be two commutative $\Q$-algebras with unit and $\la\;\,\ra \colon R\to R'$
a linear map sending~1 to~1.  (Of course the cases of interest to us will be when
$R$ is the Bloch-Okounkov ring~$\Lambda^*$ and $\la\;\,\ra$ is the $q$-, $X$-, or 
$\hslash$-bracket to $R'=\wM_*$, $\Q[X]$, or $\Q[\pi^2][\hslash]$,
respectively.) Then we extend
 $\la\;\,\ra$ to a  multi-linear map $R^{\otimes n}\to R'$ for every $n\ge1$, 
called {\em connected brackets}, 
the image of $g_1\otimes\cdots\otimes g_n$  being denoted by 
$\la g_1|\cdots|g_n\ra$,  
that we define by 
\be  \label{slash}
\langle g_1|\cdots|g_n\rangle \= \sum_{\alpha \in \PPP(n)} (-1)^{\ell(\a)-1} 
(\ell(\a)-1)!\, \prod_{A\in\a} \Bigl\la \prod_{a\in A} g_a \Bigr\ra \,.
\ee
The most important property of connected brackets is their appearance 
in the logarithm of the original bracket applied to an exponential:
\bas \log\bigl(\bigl\la e^{g_1+g_2+g_3+\cdots}\bigr\ra\bigr)
   &\= \log\Bigl(1 \+ \sum_i\la g_i\ra \+ \frac1{2!}\sum_{i,j}\la g_ig_j\ra
     \+ \frac1{3!}\sum_{i,j,k}\la g_ig_jg_k\ra \+\cdots\Bigr) \\
  &\= \sum_i\la g_i\ra \+ \frac1{2!}\sum_{i,j}\la g_i|g_j\ra 
   \+ \frac1{3!}\sum_{i,j,k}\la g_i|g_j|g_k\ra \+\cdots \;. \eas
\par
We specialize to the Bloch-Okounkov ring~$\Lambda^*$
and we want to compute the leading terms of the connected brackets
associated with the $\la\cdot\ra_X$- or $\la\cdot\ra_{\hslash}$-brackets. 
Recall from \cite[Proposition~11.1]{cmz}:
\par
\begin{Prop} \label{prop:degdrop}
Let $g_i \in \Lambda^*_{\leq k_i}$ ($i=1,\ldots,n$) be elements
of weight less than or equal to~$k_i$ and let $g_i^\top \in \Lambda^*_{k_i}$
be their top weight components. Let $k=k_1+\cdots+k_n$ be the total weight.
Then $\deg(\la g_1|\cdots|g_n\ra_X) \leq 1-n+k/2$ and 
\be \label{eq:ggtop}
 [X^{1-n +k/2}] \,\la g_1|\cdots|g_n\ra_X 
\=  [X^{1-n +k/2}] \,\la g_1^\top|\cdots|g_n^\top\ra_X\,.
\ee
\end{Prop}
The leading terms of the brackets are consequently
\ba \label{eq:deflead}
\la g_1|\cdots|g_n\ra_L &\,=\, [X^{1-n +k/2}] \,\la g_1|\cdots|g_n\ra_X 
\= \lim_{X\to\infty} \frac{\evX[\langle g_1|\cdots|g_n\rangle_q](X)}{X^{1-n+k/2}} \\
&\= [\hslash^{-k-1+n}] \,\la g_1|\cdots|g_n\ra_\hslash 
\= \lim_{\hslash\to 0} \hslash^{k+1-n}\,\evh[\langle g_1|\cdots|g_n\rangle_q](\hslash)\,.
\ea
We call them {\em rational cumulants}.
\par

\subsection{The cumulant recursion}
\label{sec:cumrec}

In this section we prove a formula for computing the connected brackets
associated with the $\sbq\cdot$- or rather the $\sbrh\cdot$-brackets.
The core mechanism for their computation is 
summarized in the following purely algebraic property.
\par
Let $R$ be an $\NN$-graded commutative $\QQ$-algebra with $R_0 = \QQ$, 
complete with respect to the maximal ideal $\frakm = R_{>0}$. The following
statement gives a general recursion for expressions that appear in cumulants.
We will specialize~$R$ to the Bloch-Okounkov ring subsequently.
\par
\begin{KeyLe} \label{KL:D2recursion}
Suppose that $D\colon R \to R$ is a linear map. Then the following statements 
are equivalent: 
\begin{itemize}
\item[(1)] We have $D(x^3) - 3xD(x^2) + 3x^2D(x) = 0$ for all $x \in R$.
\item[(2)] For all $x \in R$ and all $n \geq 2$ 
$$D(x^n) = \binom{n}{2}D(x^2)x^{n-2} - n(n-2)D(x)x^{n-1} \,.$$
\item[(3)] For all $x,y,z \in R$
$$ D(xyz) \= xD(yz)+yD(xz)+zD(xy) - xyD(z)-xzD(y)-yzD(x)\,.$$
\item[(4)]  If we denote by $D_2\colon R^2 \to R$ the symmetric bilinear form
$$ D_2(x,y) = D(xy) - xD(y) -yD(x)$$
then 
$$ D(x_1\cdots x_n) \= \sum_{i=1}^n D(x_i) x_1\cdots \widehat{x_i}\cdots x_n
+ \sum_{1\leq i<j\leq n} D_2(x_i,x_j) x_1\cdots \widehat{x_i}\cdots \widehat{x_j}\cdots x_n $$
for all $x_1,\ldots,x_n \in R$.
\item[(5)] For any fixed~$x \in R$, the bilinear form  $D_2(x,y)$ is a
derivation in~$y$.
\item[(6)] The map $D \in {\rm Sym}^2({\rm Der}(R))$, i.e.\ $D$ is
a second order differential operator without constant term.
\item[(7)] For all $X \in \frakm$ there exists $\cLL(X) \in R$ such that
\be \label{eq:abstractcumu}
\log(e^{\hslash D}(e^{X/\hslash})) \= \frac{1}{\hslash} \cLL(X) \+ O(1)\quad(\hslash \to 0)\,.
\ee 
\end{itemize}
If any of these statements holds, the leading 
term of~\eqref{eq:abstractcumu} is given by $\cLL(X) = L(1)$, where
\be L(0)\= X\,, \quad L'(t)\= \frac12 D_2(L(t),L(t))\,.\ee
\end{KeyLe}
\par
\begin{proof}
The implication $(1) \Rightarrow (3)$ follows by passing to
the polarization. The implication $(3)\Rightarrow (4)$ can be proved by induction (replace
$x_1$ by $x_0 x_1$). The implications $(2) \Rightarrow (1)$ and 
$(4) \Rightarrow (1),(2),(3)$ follow by specialization.
The equivalence $(5) \Leftrightarrow (3)$ follows by direct computation.
To show $(6) \Leftrightarrow (5)$ think deeply.
To prove $(7) \Rightarrow (1)$ it suffices to consider the cubic term: 
the coefficient of $1/\hslash^2$ is the expression in~(1).
\par
To prove $(6) \Rightarrow (7)$ and the final formula for $\cLL(X)$ we
write
\be \label{eq:exphD}
e^{\hslash D}(e^x) \= e^{y(\hslash)} \,.
\ee
Then $y(0)=x$. Note that $(2)$ implies that  
$$ D(e^x) \= \frac12 D(x^2)e^x + D(x) (1-x)e^x\,.$$
Differentiation of~\eqref{eq:exphD} with respect to~$\hslash$ implies that $y' =D(y)+\tfrac12 D_2(y,y)$. Equivalently, writing
$y \= \sum_{n\geq 0} f_n(x) \hslash^n$, then the initial condition is that $f_0(x)=x$, and 
$$ (n+1)f_{n+1}(x) \= D(f_n) + \frac12 \sum_{m=0}^n D_2(f_m(x),f_{n-m}(x))\,.$$ 
Recursively this implies that $f_n(X/\hslash) = \cLL_n(X) \hslash^{-n-1} + O(\hslash^{-n})$ with
$\cLL_0(X)=X$ and 
$$  (n+1)\cLL_{n+1}(X) = \frac12 \sum_{m=0}^n D_2(\cLL_m(X),\cLL_{n-m}(X))\,.$$ 
We now let $L(t) =\sum_{n\geq 0} \cLL_n(X) t^n$ and the claims follow.
\end{proof}

\subsection{Application to volume computations}
\label{sec:appvol}

We now return to the proof of the main theorem of this section.
Recall the main idea from \cite{eo} that the volume of a stratum 
is given by the growth rate of the number of connected torus covers
and thus to the leading terms of cumulants of the~$f_\ell$. More
precisely for $2g-2 = \sum_{i=1}^n m_i$, the same argument as in \cite[Proposition~19.1]{cmz} 
gives that 
\be \label{eq:cumutovol}
{\rm vol}\,(\omoduli[g,n](m_1,\ldots,m_n)) \= (2\pi i)^{2g}  
\frac{\bL{f_{m_1+1}| \cdots |f_{m_n+1}}}{ (2g-2+n)! } \,.
\ee
\par
\begin{proof}[Proof of Theorem~\ref{thm:D2recursion}, one variable case]
First, $P_B(u) = P(u)|_{p_\ell \mapsto \ell! b_\ell} = P(u)(\emptyset)$.
Next, recall that Lagrange inversion for a power series $F \in u\CC[[u]]$
with non-zero linear term and inverse $G(z)$ states that
$k[z^k]G^n = n [u^{-n}]F^{-k}$ for $k,n \neq 0$. We apply this
to  $F = u/P_B(u)$ and to $k=2g-1$ and $n=-1$ to obtain that 
\ba \label{eq:fellinversion}
\frac{(2g-1)!}{(2\pi i)^{2g}} v(2g-2) &\= (2g-1)\bL{f_{2g-1}} \= \bL{h_{2g-1}} \\
&\= \frac{-1}{2g-1} [u] (P_B(u)/u)^{2g-1}
\= \bigl[u^{2g-1}\bigr] \frac{1}{(u/P_B)^{-1}}
\ea
using Proposition~\ref{prop:EOconv}, \eqref{eq:hldef} and Lagrange inversion.
\end{proof} 
\par
We pause for a moment to check the initial condition of the theorem
in the previous section independently of the Hermitian metric extension problem along the boundary of the strata. 
\par
\begin{proof}[Proof of~\eqref{eq:volminviaa} using~\eqref{eq:fellinversion}]
We want to show that $(2g-1)^2 \,a (2g-2) = \bL{h_{2g-1}}$.
Recall that a version of Lagrange inversion 
(see e.g.\ \cite[Formula~(2.2.8)]{Gessel}), in fact the case $k=0$
excluded in the version of the previous proof, states that
if $F \in z + z^2\CC[[z]]$ with composition inverse~$G(u)$, then for any Laurent
series $\phi(z)$
\be \label{eq:LagLogar}
[z^0] \phi(F)  \= [u^0] \phi(u) + [u^{-1}] \phi'(u) \log(G/u) \,.
\ee
If we let $\widetilde{\cAA}(z) =  1/z + \sum_{g\geq 1} \bL{h_{2g-1}} \, z^{2g-1}$, then we
need to show that $\widetilde{\cAA}(z) = \cAA(z)$.
We apply Lagrange inversion to $\phi(z) = z^{-2g}$ and $F = 1/\widetilde{\cAA}(z)$ to obtain that 
\bas
\frac{1}{2g!}[z^{0}]\widetilde{ \cAA}^{2g} &\= \frac{1}{2g!}[z^0] \phi(1/\widetilde{\cAA}(z)) 
\= \frac{-1}{(2g-1)!} [u^{2g}] \log(1/P_B) \\
&\=  \frac{-1}{(2g-1)!}  [u^{2g}]\,\sum_{s \geq 1} s!b_{s+1} u^{s+1}
\= [u^{2g}]B(u) = \frac{1}{(2g)!}[z^{0}] \cAA^{2g}
\eas
using~\eqref{eq:fellinversion} and~\eqref{eq:ind1}. This implies the claim.
\end{proof}
\par
For the general case of the theorem, we apply Section~\ref{sec:cumrec} to the
differential operator 
\be \label{eq:defD}
D= \tfrac12 (\Delta - \p^2 -  \p/\p p_1)\,.
\ee
\par
\begin{Prop} \label{prop:D2fg}
  The polarization $D_2$ of $D$ can be computed in terms of
  the $h_\ell$-generators by the formula in~\eqref{eq:D2H}.
\end{Prop} 
\par
\begin{proof} In terms of the $p_\ell$-generators the polarization is given by
\be
D_2(f,g) \= \sum_{k,\ell \geq 1} 
\Bigl((k+\ell)p_{k+\ell-1} - k\ell p_{k-1} p_{\ell-1} \Bigr)
\frac{\partial f}{\partial p_k}
\frac{\partial g}{\partial p_\ell} \,.
\ee
The definition~\eqref{eq:hldef} of~$h_\ell$ in terms of~$p_\ell$ implies that 
\be \label{eq:auxsum}
\frac{\partial \cHH(z)}{\partial p_k} \= -\frac{z \cHH'(z)}{\cHH(z)^{k+1}}
\quad \text{and} \quad
\sum_{n \geq 2} n p_{n-1} \cHH^{-n}(z) \= -\cHH(z)/z\cHH'(z)-1\,.
\ee
We compute that
\bas &\phantom{\=} \,\,  \sum_{k,\ell \geq 1} 
\Bigl((k+\ell)p_{k+\ell-1} - k\ell p_{k-1} p_{\ell-1} \Bigr) \frac{z_1 \cHH'(z_1)}
     {\cHH(z_1)^{k+1}} \cdot \frac{z_2 \cHH'(z_2)}{\cHH(z_2)^{\ell+1}} \\
&\= \frac{z_1 \cHH'(z_1)z_2\cHH'(z_2)}{\cHH(z_1)\cHH(z_2)} 
\Bigl( \sum_{n \geq 2} n p_{n-1} \frac{\cHH(z_1)^{1-n}-\cHH(z_2)^{1-n}}
{\cHH(z_2) - \cHH(z_1)} \\
& \phantom{\= \frac{z_1\cHH'(z_1)z_2\cHH'(z_2)}{\cHH(z_1)\cHH(z_2)}mm} -
\Bigl(1 + \frac{\cHH(z_1)}{z_1\cHH'(z_1)}\Bigr)\Bigl(1 + \frac{\cHH(z_2)}
{z_2\cHH'(z_2)}\Bigr) 
\Bigr) \\
&\= \cHH_2(z_1,z_2)\,,
\eas
and this implies the claim by the chain rule.
\end{proof}
\par
We now define the partition function of $h$-brackets  
\begin{flalign} \Phi^H(\bfu)_q & \= \Bigl\langle \exp\Bigl(\sum_{\ell \geq 1} h_\ell u_\ell\Bigr)\,\Bigr\rangle_q
  \=\sum_{\bfn\geq 0}\,\langle\underbrace{h_1\cdots h_1}_{n_1} \underbrace{h_2\cdots h_2}_{n_2}\cdots\rangle_q \,\frac{\bfu^\bfn}{\bfn!} 
   \nonumber  \\   \label{eq:defPhi}   
  &\=\sum_{n=0}^\infty \frac1{n!} \sum_{\ell_1,\ldots,\ell_n \geq 1} \langle h_{\ell_1} \cdots h_{\ell_n} \rangle_q \, u_{\ell_1}\cdots u_{\ell_n}\;
\end{flalign}
in the $h_\ell$-variables. Then the partition function of the
rational cumulants for the $h_\ell$-generators
\ba \label{eq:psiHq}
\Psi^H(\bfu)_q &\= \sum_{\bfn \geq 0}\la\underbrace{h_1|\cdots|h_1}_{n_1}|\underbrace{h_2|\cdots|h_2}_{n_2}|\cdots\ra_q\,\frac{\bfu^\bfn}{\bfn!}
  \= \log \Phi^H(\bfu) _q\,
\ea
is simply the logarithm of $\Phi^H$.
\par
\begin{proof}[Proof of Theorem~\ref{thm:D2recursion}, general case]
We first show that the pieces of $\Phi^H$ sorted by total degree in~$\bfu$
can be recursively computed using the $D_2$-operator. For this purpose
we let $\wt h_i = \hslash^{-i} h_i$. From the definition of cumulants, equation~\eqref{eq:defevh}
and $a_0(\la g \ra_q) = g(\emptyset)$, we obtain that 
\ba
\sum_{\bfn \geq 0}\la\underbrace{\wt h_1|\cdots|\wt h_1}_{n_1}|
\underbrace{\wt h_2|\cdots|\wt h_2}_{n_2}| \cdots\ra_{\hslash}\,\frac{\bfu^\bfn}
{\hslash^{|\bfn|}\, \bfn!}  
&\= \log\Bigl(\la \exp({\frac1\hslash} \sum_{i \geq 1} \wt h_i u_i)
\ra_\hslash\Bigr) \\
&\= \log\Bigl(e^{\hslash D} \exp\Bigl({\tfrac1\hslash} \sum_{i \geq 1} 
\wt h_i u_i\Bigr)\Bigr)(\emptyset)\,.
\ea
By applying the Key Lemma with $X =  \sum_{i \geq 1} \wt h_i u_i$ and undoing the
rescaling of the $h_i$ using~\eqref{eq:deflead} we obtain that 
\ba
\sum_{\bfn \geq 0}\la\underbrace{h_1|\cdots|h_1}_{n_1}|
\underbrace{h_2|\cdots|h_2}_{n_2}| \cdots\ra_L\,\frac{\bfu^\bfn}{ \bfn!}  
= \biggl(\sum_{n=0}^\infty \cLL_n\biggr) (\emptyset)
\ea
with
\be \label{eq:lrec}
\cLL_0 \= \sum_{n \geq 1} h_i u_i \quad \text{and} \quad
\cLL_n \= \frac1{2n} \sum_{r+s = n-1} D_2(\cLL_r,\cLL_s)
\ee
for $n>0$. Now define a linear map $\mho_n\colon \QQ[\bfu] \to \QQ[\bfz]$ by
  $$\mho_n(u_{\ell_1} \cdots u_{\ell_n}) \= {\rm Symm}
  (z_1^{\ell_1} \cdots z_n^{\ell_n})
  $$
and zero for monomials of length different from~$n$,
where ${\rm Symm}$ denotes symmetrization with respect to the $S_n$ action
on the variables~$z_i$. In this notation $\cHH_1 =  \mho_n \cLL_{n-1}$ and 
$$  \cHH_n \= \mho_n \Bigl( 
\sum_{\bfn \geq 0}\la\underbrace{h_1|\cdots|h_1}_{n_1}|\underbrace{h_2|\cdots|h_2}_{n_2}|
\cdots\ra_L\,\frac{\bfu^\bfn}{\bfn!}
\Bigr)\,. $$
Consequently, \eqref{eq:lrec} and ~\eqref{eq:Hindform} together with~\eqref{eq:cumutovol}
and Proposition~\ref{prop:EOconv} imply the claim.
\end{proof}



\section{Equivalence of volume recursions}
\label{sec:D2VR}

In this section we introduce another ``averaged volume'' recursion that 
interpolates between the $D_2$-recursion introduced in Section~\ref{sec:D2rec}
and the volume recursion in  Theorem~\ref{intro:VolRec}. We will show that 
the averaged volume recursion and the $D_2$-recursion give the same generating 
functions, and then Theorem~\ref{intro:IntFormula} will follow from it. 
\par
Recall from~\eqref{eq:Hindform} the definition of $\cHH_{\{i,j\}}\in R[[z_i, z_j]]$ for $i\neq j$, where 
$R=\QQ[h_1,h_2,\ldots]$. For any list of positive integers $\bfp=(p_1,\ldots, p_k)$ we define 
$$
\cHH_{\{i,j\}}^\bfp \= \frac{\partial^{k}}{\partial h_{p_1} \cdots 
\partial h_{p_k}} \cHH_{\{i,j\}}\,.
$$  
For a finite set  $I=\{i_1,\ldots,i_n\}$ of positive integers 
we define the formal series $\cAA_I \in R[[z_{i_1},z_{i_2},\ldots]]$ inductively by
\begin{flalign}
  \cAA_I &\= \cHH_I \in \frac{1}{z_i}R[[z_i]] \qquad \text{if $n = 1$,
and otherwise} \label{eq:calArec} \\ 
\cAA_{I} &\= \frac{1}{n-1} \sum_{1\leq r < s \leq n}\,
\sum_{k \geq 0} 
\sum_{\begin{smallmatrix} \bfp=(p_1,\ldots, p_k)\\ I=\{i_r,i_s\}\sqcup I_1\sqcup\cdots \sqcup I_k  
\end{smallmatrix}} \frac{1}{k!} \,\cHH_{i_r,i_s}^\bfp \cdot
\prod_{j=1}^k \cAA_{I_j}^{[p_j]}\,, \nonumber
\end{flalign}
where $\cAA_I^{[p]} := [z_i^p] \cAA_{I \cup \{i\}}$ for any $i \not\in I$.
We set $\mathcal{A}_{n}=\mathcal{A}_{[\![1,n]\!]}$. Note that $\cAA_I=\cHH_I$
by definition if $|I| =2$. We have 
chosen to sum in definition of $\cAA_{I}$ over all~$\bfp$ rather than
partitions $g=\sum_{i=1}^k g_i$ as e.g.\ in Theorem~\ref{intro:VolRec}. The two 
summations are equivalent, since~$g_i$ and~$p_i$ determine each other
once~$I$ has been partitioned. Our goal here is to show the following result.
\par
\begin{Thm}\label{BOintMV}
For all non-empty sets of positive integers $I$, we have $\cAA_I=\cHH_I.$
\end{Thm}
\par
For the proof of this theorem we will show that both $\cAA_I$ and $\cHH_I$ 
can be written as a sum over a function on certain oriented trees. The two
recursions can then be viewed as stemming from cutting the trees at a
local maximum (a ``top'') or a local minimum (a ``bottom'') respectively.

\subsection{Proof of Theorem~\ref{intro:VolRec} and
 Theorem~\ref{intro:IntFormula}} 

We assume in this section that Theorem~\ref{BOintMV} holds and finish the proof of
Theorems~\ref{intro:VolRec} and~\ref{intro:IntFormula} under this assumption. We
abbreviate $\lambda = (\ell_1,\ldots,\ell_n)$ and recall from Section~\ref{sec:D2rec}
that we denoted the coefficients of $\mathcal{A}_n$ for $n \geq 2$ by
$$\mathcal{A}_n \= \mathcal{H}_n \= \sum_{\ell_1,\ldots, \ell_n \geq 1}
h_\lambda \, 
z_{i_1}^{\ell_1} \cdots z_{i_n}^{\ell_n} 
\,.$$
\par
\begin{Prop} The coefficients  $h_{\lambda}$
are uniquely determined 
by the recursion 
\ba \label{VRBO}
h_\lambda 
\= \sum_{1\leq r< s\leq n} \frac{\ell_r+ \ell_s}{n-1} \,
\sum_{k \geq 1} \frac{1}{k!} \, \sum_{\bfg, \bf\mu} 
h_{\proj^1}(\ell_r-1,\ell_s-1,\bfp) \cdot
\prod_{i=1}^k  h_{\lambda_i,p_i} 
\ea
for $n \geq 2$, where the summation is as in Theorem~\ref{intro:VolRec},
except that $\lambda_1\sqcup \cdots \sqcup \lambda_k = [\![1,n]\!]\smallsetminus\{r,s\}$. 
\end{Prop}
\par
\begin{proof} We begin by showing the formula for $\lambda$ with two entries,
which in view of~\eqref{eq:defhP1} is equivalent to show that
\begin{eqnarray*}
\sum_{\ell_1,\ell_2 \geq 1} \frac{h_{\ell_1,\ell_2}}
{\ell_1+\ell_2} z_1^{\ell_1}z_2^{\ell_2} &=& \sum_{k\geq 1}\frac{1}{k} \left(\sum_{p\geq 1}
h_p  z_1z_2\frac{z_1^p-z_2^p }{z_1-z_2} \right)^k\\
&=& - {\rm log} \left( \frac{z_1z_2}{z_1-z_2}\left(\cHH_1(z_2)-\cHH_1(z_1)\right)\right)\,.
\end{eqnarray*}
Since applying $(z_1 \frac{\partial}{\partial z_1}+z_2 \frac{\partial}{\partial z_2})$
to the left-hand side above gives $\cAA_2$, and since neither side has a constant term, 
this in turn follows from 
\begin{eqnarray*}
\cHH_2 \= \cAA_2
&=& -1+ \frac{z_1 \cHH_1'(z_1)-z_2 \cHH_1'(z_2)}{\cHH_1(z_2)- \cHH_1(z_1)}\\
&=& - \left(z_1 \frac{\partial}{\partial z_1}+z_2 \frac{\partial}{\partial z_2} \right)
{\rm log} \left( \frac{z_1z_2}{z_1-z_2}\left(\cHH_1(z_2)-\cHH_1(z_1)\right)\right)\,.
\end{eqnarray*}
\par 
For $\lambda$ with more entries, we deduce (for all $\ell_r,\ell_s \geq 1$ and 
all $\bfp$) from the preceding calculation that 
\begin{eqnarray*}
[z_r^{\ell_r}z_s^{\ell_s}] \cHH_{\{ r, s\}} &\=&  \sum_{k'>0} \frac{\ell_r+\ell_s}{k'!} \sum_{\bfp'} h_{\proj^1}(\ell_r-1,\ell_s-1,\bfp') \prod_{i=1}^{k'} h_{p_i'}\,\\
\text{and}\qquad [z_r^{\ell_r}z_s^{\ell_s}]\cHH_{\{ r, s\}}^\bfp &\=&  \sum_{k' > 0} \frac{\ell_r+\ell_s}{k'!} \sum_{\bfp'} h_{\proj^1}(\ell_r-1,\ell_s-1,\bfp\cup\bfp') \prod_{i=1}^{k'}
h_{p_i'}\,,
\end{eqnarray*}
Besides, the recursion formula~\eqref{eq:calArec} defining $\cAA_I$ can be translated
for~$\lambda$ with~$n$ parts into 
\begin{flalign} 
&\phantom{\=} h_\lambda 
\= \sum_{1\leq r< s\leq n}  \frac{1}{n-1}
\sum_{\begin{smallmatrix} k\geq 1,g_i, \lambda_i\\ |\lambda_i|>0 \end{smallmatrix}}
\frac{1}{k!} [z_r^{\ell_r}z_s^{\ell_s}] \cHH_{\{ r, s\}}^\bfp  \prod_{i=1}^k
h_{\lambda,p_i} 
 \label{eq:frakalambda} \\
&\= \sum_{1\leq r< s\leq n}  \frac{\ell_r+\ell_s}{n-1}
\sum_{k\geq 1, g_i, |\lambda_i|>0 \atop k' > 0, \bfp' }
\frac{1}{k!k'!} h_{\proj^1}(\ell_r-1,\ell_s-1,\bfp\cup\bfp')
\prod_{i=1}^k  h_{\lambda,p_i} 
\prod_{j=1}^{k'} h_{p_j'}
\,, \nonumber
\end{flalign} 
where $\lambda_1\sqcup \cdots \sqcup\lambda_k$ partitions 
$\lambda \setminus \{\ell_r, \ell_s \}$. Next we remark that the interior
sum of~\eqref{VRBO} is over all backbone graphs with the two markings labelled with $r$ and $s$ at the lower level.
In the preceding formula~\eqref{eq:frakalambda} the contribution of vertices
with at least one marking is separated from the vertices with no markings.
This choice results in a binomial coefficient  ${k+k' \choose k}$ and transforms
 $\frac{1}{k! k'!}$ into $\frac{1}{(k+k')!}$, thus showing that the two 
recursive formulas~\eqref{eq:frakalambda} and~\eqref{VRBO} are equivalent.
\end{proof}
\par
\begin{proof}[Proof of Theorem~\ref{intro:VolRec} and
Theorem~\ref{intro:IntFormula}]
For $\mu=(m_1,\ldots,m_n)$ consider the intersection numbers $a(\mu)$ that satisfy the recursion in
Theorem~\ref{thm:indintall}, 
and recall that $a(\mu) = a_i(\mu)$ is independent of the index $i$ by Proposition~\ref{independent}. 
In particular the $a(\mu)$ satisfy the recursion~\eqref{eq:ind2} for any distinguished pair of indices, and hence 
satisfy every weighted average of these recursions. We use the weighted
average where the recursion with $(i,j)$ distinguished is taken with weight
$\prod_{k \not\in\{i,j\}} (m_k+1)$. Conversely, the $a(\mu)$ are
uniquely determined by this weighted average and the initial values for $\mu$
of length one given in~\eqref{eq:ind1}.
\par
On the other hand, the collection of $(2g-2+n) \prod_{i=1}^n (m_i+1)a(\mu)$
and the collection of $h_{m_1+1,\ldots,m_n+1}$ 
both satisfy the recursion~\eqref{VRBO}, by observing that 
\be \label{eq:average}
(2g-2+n) a(\mu) \prod_{i=1}^n (m_i+1)
\= \sum_{1\leq r< s\leq n} \!\!\frac{(m_r+1+ m_s+1)}{n-1}  a(\mu) \prod_{i=1}^n (m_i+1)\,.
\ee
Note that $\cAA_1|_{h_{\ell}\mapsto \alpha_{\ell}} = \cAA$ by Theorem~\ref{thm:D2recursion} 
and since we already checked (see~\eqref{eq:fellinversion} and the subsequent proof) 
that the one-variable rescaled volumes $v(2g-2)$ and $a(2g-2)$ agree 
(see~\eqref{eq:volminviaa}) up to the factor $(2\pi i)^{2g}/(2g-1)!$.  This implies that 
\be \label{for:liftint}
a(\mu) \= \frac{h_{m_1+1,\ldots,m_n+1}|_{h_{\ell}\mapsto \alpha_{\ell}}}{(2g-2+n)
\prod_{i=1}^n (m_i+1)}\,.
\ee
The claim now follows from Theorem~\ref{thm:D2recursion}, Theorem~\ref{BOintMV}
and the conversion~\eqref{eq:volai} of volumes to the~$a(\mu)$.
\end{proof}

\subsection{Oriented trees}

We now start preparing for the proof of Theorem~\ref{BOintMV}. 
An {\em oriented tree} is the datum of a graph $G=(V, E\subset V\times V)$
whose underlying graph of $(V,E)$ is a tree. In particular it is
required to be connected. If $(v,v')\in E$, we will denote $v>v'$. Moreover,
a vertex $v\in V$ is called a {\em bottom} (respectively a {\em top}) if
there exists no $v'\in V$ such that $v > v'$ (respectively $v < v'$). We will
denote by $B(G)$ and $T(G)$ the sets of bottoms and tops of $G$. 
\par
For any oriented tree $G$ with $n$ vertices, we define the rational number
\be\label{eq:f(G)}
f^{\#}(G) :=\frac{ {\rm Card} \left\{ \sigma\colon V\overset{\sim}{\to} [\![1,n]\!],\ \text{s.t.} \forall (v,v')\in E, \sigma(v)>\sigma(v') \right\}}
{n!} \,,
\ee
whose numerator is the number of total orderings on the set of vertices
compatible with the orientation of~$G$.
\par
\begin{Lemma}\label{bottomtop}
The function $f^{\#}$ can be expressed as 
$$
f^{\#}(G)
\= \frac{1}{n} \cdot \sum_{v \in B(G)} \left(\prod_{G'} f^{\#}(G') \right)
\= \frac{1}{n} \cdot \sum_{v \in T(G)} \left(\prod_{G'} f^{\#}(G') \right)\,,
$$
where in both cases the product is over all connected components~$G'$
of the oriented graph obtained by deleting the vertex $v$.  
\end{Lemma}
\par
\begin{proof}
In order to define a total ordering on $V$ compatible with the orientation
of~$G$, we begin by choosing a minimal element $v\in V$. This element is
necessarily a bottom. Let us fix such a choice and denote by
$(G_1,\ldots,G_k)$ the connected components of $G \setminus \{v\}$.
Let $n_i$ be the number of vertices of $G_i$ for $1\leq i\leq k$.  A total
ordering on~$V$ with minimal element~$v$ is equivalent to choosing a total
ordering on the vertices of~$G_i$ for all $1\leq i\leq k$ and a partition
of $[\![1,n-1]\!]$ into $k$ sets of size $(n_1,\ldots,n_k)$. Such an ordering
on~$V$ is compatible with the orientation of $G$ if and only if each ordering
on the vertices of $G_i$ is compatible with the orientation of $G_i$ for all
$1\leq i\leq k$. This implies that the number of total orderings on $V$
with minimal element $v$ is equal to 
$$
\left(\begin{matrix} n-1 \\ n_1 \ \cdots \ \ n_k \end{matrix}\right) \cdot
\prod_{i=1}^k (n_i! f^{\#}(G_i))= (n-1)! \cdot \prod_{i=1}^k f^{\#}(G_i)\,.
$$
Summing over all possible choices of a minimal element, the number of total
orderings on vertices of $G$ compatible with the orientation of $G$ is equal to
$$
(n-1)! \cdot \sum_{v \in B(G)} \left(\prod_{G'} f^{\#}(G')
\right),
$$
which completes the proof.
\end{proof}

\subsection{Decorations of oriented trees}

Let $I$ be a non-empty finite set of positive integers. An {\em $I$-decoration} of an oriented tree $\Gamma=(V,E)$ is the datum of
a function ${\rm dec}\colon I\to V$ such that for each vertex~$v$ the number
of outgoing edges plus the number of decorations is equal to two, i.e.\
$$
\#({\rm dec}^{-1}(v)) + \#(E\cap (\{v\} \times V)) \=2\,
$$
for all $v\in V$. 
If $I$ has cardinality greater than one, we denote by ${\rm OT}(I)$ the set
of $I$-decorated oriented trees.  One can easily check that the following two
properties hold:
\begin{itemize}
\item if $I$ has cardinality $n\geq 2$, then $\Gamma$ has $n-1$ vertices;
\item a vertex of a decorated tree is a bottom if and only if it has exactly
  two markings. 
\end{itemize} 
\par
We denote by ${\rm OT}(I)^{v}, {\rm OT}(I)^{b}$ and ${\rm OT}(I)^{t}$ the sets
of decorated trees with a choice of an arbitrary vertex, a choice of a
bottom and a choice of a top, respectively. 
If~$I = \{i\}$ has only one element, we define ${\rm OT}(\{i\})^{v}=\{i\}$ as a trivial graph decorated by~$i$. 
\par
\begin{Lemma}\label{le:phi-psi}
If $I$ has cardinality greater than one, then there is a bijection
\begin{eqnarray}\label{bijection1}
  \varphi^{t}\colon {\rm OT}(I)^{t} \to \left(\bigcup_{I'\subset I} {\rm OT}(I')^{v}
  \times {\rm OT}(I\setminus I')^{v}\right)\bigg/ (I' \sim I\setminus I')\,
\end{eqnarray}
given by cutting at a top vertex, were the union is over all non-empty
proper subsets. Similarly, there is a bijection
\begin{eqnarray}\label{bijection2}
\varphi^{b}\colon  {\rm OT}(I)^{b} \to \bigcup_{\{i_1,i_2\}\subset I, k > 0} \left(  \bigcup_{I=\{i_1,i_2\} \sqcup  I_1\sqcup \cdots \sqcup I_k} \prod_{j=1}^k {\rm OT}(I_j\cup \{e_j\} ) \right)\bigg/S_k\,
\end{eqnarray}
given by cutting at a bottom vertex, were the union is over all partitions of
$I$ into $k+1$ non-empty sets such that the first distinguished set has 
precisely two elements and where the element $e_j={\rm max}(I)+j$
for all $1\leq j\leq k$.
\par
We denote by $\psi^t$ and $\psi^b$
the inverses of $\varphi^t$ and $\varphi^b$, respectively.
\end{Lemma}
\par
\begin{proof}
Given a decorated tree with a chosen top vertex $v$, we define its image
under $\varphi^{t}$ as follows:
\begin{itemize}
\item If there are two markings on $v$, then $v$ has no outgoing edges, hence the graph has $v$ as a unique vertex and 
$I$ has only two elements. It follows that ${\rm OT}(I)^{v}$ has only one element (and so does the right-hand side
of~\eqref{bijection1}). 
\item If there is only one marking $i\in I$ on $v$ and one outgoing edge to a vertex $v'<v$, then $I'=\{i\}$ and the corresponding element in ${\rm OT}(I\setminus I')^{v}$ is the graph obtained by deleting $v$ and choosing $v'$ as the distinguished vertex.
\item If there are two outgoing edges to vertices $v'<v$ and $v''<v$, then $v$ has no $I$-markings and the graph obtained by deleting $v$ has two connected components. We define $I'$ to be the set of markings
on the component containing $v'$ and define the corresponding elements of ${\rm OT}(I')^{v}$
and ${\rm OT}(I\setminus I')^{v}$ to be the connected components containing~$v'$
and~$v''$ as chosen vertices, respectively. 
\end{itemize}
The inverse of $\varphi^t$ in the first two cases is clear, and in the last case is given by adding a top vertex adjacent to the two chosen vertices.
\par
Given a decorated tree with a chosen bottom vertex~$v$, in the same spirit we define the function~$\varphi^b$ as follows. Since $v$ is a bottom, it has no outgoing edges, hence it has exactly two $I$-markings~$i_1$ and $i_2$, and the corresponding $k$ graphs on the right-hand side of~\eqref{bijection2} are the $k$ connected components of the graph obtained by removing $v$. The inverse of $\varphi^{b}$ is given by gluing these $k$ graphs back to $v$ along the vertices marked by $e_1, \ldots, e_k$.  
\end{proof}
\par
Now we fix a ring $R'$ and a function $g\colon {\rm OT}(I)\to R'$. By slight abuse
of notation, we write $f^{\#}\colon {\rm OT}(I)\to \QQ$ for the composition of $f^{\#}$ defined in~\eqref{eq:f(G)} with the forgetful map of the decorations.  As a consequence of the two
preceding lemmas we see that the sum 
$$
S(g)=\sum_{\Gamma \in {\rm OT}(I)} f^{\#}(\Gamma) g(\Gamma)
$$
can be rewritten in two different ways, namely
\be \label{eq:rew1}
S(g)=\frac{1}{2(n-1)} \sum_{I'\subset I} \sum_{\begin{smallmatrix} (\Gamma', v') \in {\rm OT}(I')^v \\ (\Gamma'', v'') \in {\rm OT}(I\setminus I')^v  \end{smallmatrix}}
f^{\#}(\Gamma') f^{\#}(\Gamma'') \,g\bigl(\psi^{t}(\Gamma', v', \Gamma'', v'' )
\bigr) 
\ee
and  
\be \label{eq:rew2}
S(g)=\frac{1}{n-1} \sum_{\begin{smallmatrix} \{i_1,i_2\}\subset I, k > 0\\ I=\{i_1,i_2\} \sqcup  I_1\sqcup \cdots \sqcup I_k\end{smallmatrix}} \!\!\frac{1}{k!}
\sum_{(\Gamma_j)_{j=1}^k}   \Bigl(\prod_{j=1}^k f^{\#}(\Gamma_j)\Bigr)
\cdot g\Bigl(\psi^{b}\Bigl(\prod_{j=1}^k \Gamma_j\Bigr)\Bigr) \,,
\ee
where $n$ is the cardinality of $I$ (i.e.~$\Gamma \in {\rm OT}(I)$ has $n-1$ vertices as remarked before).  

\subsection{Explicit expansions over decorated trees}

In order to finish the proof of Theorem~\ref{BOintMV}, we will show that both $\cAA_I$ and $\cHH_I$ are equal to a generating series $\cSS_I$ that is directly defined as a sum over ${\rm OT}(I)$. 
\par
Let $\Gamma=(V,E, I\to V)$ be an oriented tree with decoration by $I$. A {\em twist assignment}
on $\Gamma$ is a function $\bfp\colon E \to \ZZ_{>0}$. We work over the ring
$R[[(z_{i})_{i\in I}, (z_{e})_{e\in E}]]$ of formal series in variables indexed by $I\cup E$.
Given a twisted decorated oriented tree, we define the contribution of a vertex as
 $$
\cHH_v= \cHH_{2}^{\bfp_v}(z_{v,1},z_{v,2}) \in R[[(z_{i})_{i\in I}, (z_{e})_{e\in E}]]\,,
$$
where $\bfp_v$ is the list of twists associated to all vertices~$v'$ with $v'>v$ and
where $(z_{v,1},z_{v,2})$ are the variables attached to either the markings of~$v$ or
the outgoing edges from $v$ to vertices~$v'$ with $v'<v$. Then we define the contribution of the
oriented tree $\Gamma$ as
$$
{\rm cont}(\Gamma) \= \sum_{\bfp: E\to \ZZ_{>0}} \bigl[\prod_{e\in E}
z_{e}^{\bfp(e)}\bigr] \prod_{v\in V} \cHH_{v}\,
$$
if $\Gamma$ is non-trivial, and define ${\rm cont}(\Gamma) = \cHH_v$ for the trivial graph $\Gamma$ with a unique vertex $v$ and no edges.  
Finally, we set $\cSS_{\{i\}} = \cHH_{\{i\}} = \cAA_{\{i\}}$ and for $|I| \geq 2$
$$
\cSS_{I} \= \sum_{\Gamma\in {\rm OT}(I)} f^{\#}(\Gamma) \,{\rm cont}(\Gamma)\,.
$$
\par
\begin{proof}[End of the proof of Theorem~\ref{BOintMV}]
We will show that $\cSS_{I}=\cHH_{I}$ and $\cSS_{I}=\cAA_{I}$ for all sets of positive
integers~$I$ with $n={\rm Card}(I)>2$. The equalities in the case $n=2$ are obvious 
from the definition.
We assume now that $n\geq 3$ and that $\cSS_{I'}=\cHH_{I'}=\cAA_{I'}$ for
all~$I'$ such that ${\rm Card}(I')<n$.
\par
We first prove that $\cSS_{I}=\cHH_{I}$. We begin by rewriting the defining
equation~\eqref{eq:Hindform} with two auxiliary ``edge'' variables
$z_{e'}$ and $z_{e''}$ for distinct indices $e', e'' \in \NN \setminus I$ as 
\bes
\cHH_I \= \frac{1}{2(n-1)} \sum_{I'\subset I}\sum_{p_{e'},p_{e''}>0} \left( [z_{e'}^{p_{e'}}z_{e''}^{p_{e''}}] \cHH_{\{e',e''\}}  \right)  \frac{\partial\cHH_{I'}}{\partial h_{p_{e'}}}
\frac{\partial \cHH_{I\setminus I'}}{\partial h_{p_{e''}}} \,.
\ees
To evaluate the derivative of $\cHH_{I'}$, there are two cases to consider, depending
on the cardinality of~$I'$. If $I'=\{i\}$, then
$\frac{\partial\cHH_{I'}}{\partial h_{p}}=z_{i}^{p}$ for all $p>0$. Otherwise, we use
the induction hypothesis to compute that 
\begin{eqnarray*}
\frac{\partial\cHH_{I'}}{\partial h_{p}}\=\frac{\partial\cSS_{I'}}{\partial h_{p}}&=& \sum_{\Gamma\in {\rm OT}(I')} f^{\#}(\Gamma) \frac{\partial\ {\rm cont}(\Gamma)}{\partial h_{p}} \\
&=& \sum_{\Gamma\in {\rm OT}(I')} f^{\#}(\Gamma) \sum_{\bfp: E\to \ZZ_{>0}} \left[\prod_{e\in E} z_{e}^{\bfp(e)}\right ] \sum_{v\in V} \frac{\partial\cHH_{v}}{\partial h_{p}} \left( \prod_{\hat{v}\neq v} \cHH_{\hat{v}}\right)\\
&=&\sum_{(\Gamma,v) \in {\rm OT}(I')^v} f^{\#}(\Gamma) \sum_{\bfp: E\to \ZZ_{>0}} \left[\prod_{e\in E} z_{e}^{\bfp(e)}\right ] \frac{\partial\cHH_{v}}{\partial h_{p}} \left( \prod_{\hat{v}\neq v} \cHH_{\hat{v}}\right)\,.
\end{eqnarray*}
Now we assume that both $I'$ and $I\setminus I'$ have at least two elements. Take two
oriented trees $(\Gamma'=(E',V',I'\to V'),v') \in {\rm OT}(I')^v$ and
$\Gamma''=(E'',V'',I\setminus I'\to V''),v'') \in {\rm OT}(I'')^v$. Let $(\Gamma=(V,E,I\to V) ,v)=\psi^t(\Gamma',\Gamma'')$ be the combined graph in ${\rm OT}(I)^{t}$ as described in Lemma~\ref{le:phi-psi}. 
\par
The datum of two twist assignments~$\bfp'$ and $\bfp''$ on $\Gamma'$ and $\Gamma''$ respectively 
together with a pair of positive integers $(p_{e'}, p_{e''})$ is equivalent to the datum of a twist assignment
$\bfp\colon V\to \ZZ_{>0}$ on the graph~$\Gamma$. Moreover, the contributions of the
vertices of~$\Gamma$ with the twist assignment~$\bfp$ are given by
\begin{itemize}
\item $ \cHH_{\{e',e''\}}$ for the top vertex $v$,
\item $\frac{\partial\cHH_{v'}}{\partial h_{p_{e'}}}$ and $\frac{\partial\cHH_{v''}}{\partial h_{p_{e''}}}$ for the two distinguished vertices of $\Gamma'$ and $\Gamma''$, and 
\item $\cHH_{\hat{v}}$ for all other vertices $\hat{v}$.
\end{itemize}
One checks that this is still true if one of the graphs $\Gamma'$ or $\Gamma''$
has only one marking. In summary we obtain that 
\begin{eqnarray*}
\cHH_I&=& \frac{1}{2(n-1)} \sum_{I'\subset I}  \sum_{\begin{smallmatrix} (\Gamma',v')\in {\rm OT}(I')^v \\ (\Gamma'',v'')\in {\rm OT}(I\setminus I')^v\end{smallmatrix}}f^{\#}(\Gamma')f^{\#}(\Gamma'') \sum_{\bfp\colon E\to \ZZ_{>0}} \left[z_{e}^{\bfp(e)}\right] \left(\prod_{\hat{v}\in V} \cHH_{\hat{v}} \right) \\
&=&\frac{1}{2(n-1)} \sum_{I'\subset I}  \sum_{\begin{smallmatrix} (\Gamma',v')\in {\rm OT}(I')^v \\ (\Gamma'',v'')\in {\rm OT}(I\setminus I')^v\end{smallmatrix}}f^{\#}(\Gamma')f^{\#}(\Gamma'')
{\rm cont}(\psi^t (\Gamma',v',\Gamma'',v''))\\
&=& \cSS_I\,,
\end{eqnarray*}
where we use~\eqref{eq:rew1} to pass from the above second line to the third.
\par
\medskip
Finally, we prove that $\cSS_{I}=\cAA_{I}$ by a similar argument. It suffices to prove for the case $I = [[1, n]]$.  
Using the induction hypothesis that $\cAA_{I_j\cup \{n+j\}} = \cSS_{I_j\cup \{n+j\}}$, we can rewrite the inductive definition~\eqref{eq:calArec} of $\cAA_I$ as   
\bas
 \cAA_{n}  &\= \frac{1}{(n-1)} \sum_{\begin{smallmatrix} k > 0,  [\![1,n]\!] =\{i_1,i_2\}\sqcup I_1\sqcup\cdots \sqcup I_k  \\ \Gamma_j \in {\rm OT}(I_j\cup\{n+j\})
\end{smallmatrix}}  \!\!\!\!  \!\!\!\! \frac{ \prod_{j=1}^k f^{\#}(\Gamma_j)}{k!}\,  
\\
&\phantom{\=} \quad \cdot
\sum_{\bfp = (p_1,\ldots,p_k) \atop
\bfp^{(j)}\colon E_j\to \ZZ_{>0}}
\Big([z_{n+1}^{p_1}\cdots z_{n+k}^{p_k}]\, \cHH_{i_1,i_2}^{\bfp}\Big) \left[\prod_{j=1}^k\prod_{e\in E_j} z_e^{\bfp^{(j)(e)}}\right]  \prod_{v\in V_1\cup \cdots \cup V_k} \cHH_v 
\,.
\eas
The datum of $\bfp$ together with the twist assignments $\bfp^{(j)}$ for the
split graphs~$\Gamma_j$ for $1\leq j\leq k$ is equivalent to a twist assignment for the
combined graph $\Gamma = \psi^b(\Gamma_1,\ldots, \Gamma_k)$ defined in Lemma~\ref{le:phi-psi}. Moreover, given such a twist assignment the contribution
of the vertex carrying~$i_1$ and~$i_2$ is $\cHH_{i_1,i_2}^{\bfp}$. Thus we obtain that 
\begin{eqnarray*}
\cAA_{n} \!\!\!\! &=&  \!\!\!\!   \frac{1}{(n-1)}  \sum_{\begin{smallmatrix} k > 0,\\
[\![1,n]\!] =\{i_1,i_2\}\sqcup I_1\sqcup \cdots \sqcup I_k  \\ \Gamma_j \in {\rm OT}(I_j\cup\{n+j\})
\end{smallmatrix}}  \!\!\!\!  \!\!\!\! \frac{ \prod_{j=1}^k f^{\#}(\Gamma_j)}{k!}  
\cdot  \sum_{\bfp\colon E(\Gamma) \to \ZZ_{>0}} \Big[\prod_{e\in E(\Gamma)} z_e^{\bfp(e)}\Big]  \prod_{v\in V(\Gamma)} \cHH_v \\
&=&  \!\!\!\!  \frac{1}{(n-1)}  \sum_{\begin{smallmatrix} k > 0, \\  
[\![1,n]\!] =\{i_1,i_2\}\sqcup I_1\sqcup \cdots \sqcup I_k  \\ \Gamma_j \in {\rm OT}(I_j\cup\{n+j\})
\end{smallmatrix}}  \!\!\!\!  \!\!\!\! \frac{ \prod_{j=1}^k f^{\#}(\Gamma_j)}{k!}  \cdot {\rm cont}(\psi^b(\Gamma_1,\ldots, \Gamma_k))\\
&=& \cSS(I)\,,
\end{eqnarray*}
where we use~\eqref{eq:rew2} to pass from the above second line to the third.
\end{proof}


\section{Spin and hyperelliptic components} \label{sec:spin}

In this section we prove a refinement of Theorem~\ref{intro:VolRec}
and Theorem~\ref{intro:IntFormula} with spin structures 
taken into account. We also prove the corresponding refinement
for hyperelliptic components in Section~\ref{sec:VRhyp}.
Along the way we revisit the counting problem for
torus covers with sign given by the spin parity and complete the
proof of Eskin, Okounkov and Pandharipande~(\cite{eop}) that the generating function
is a quasimodular form of the expected weight. We then show that the
$D_2$-recursion has a perfect analog when counting with spin parity
and use the techniques of Section~\ref{sec:D2VR} to convert this into the
recursion for intersection numbers.
\par
In this section we assume that all entries of $\mu=(m_1,\ldots,m_n)$ are even.
The {\em spin parity} of a flat surface $(X,x_1,\ldots, x_n,\omega) \in
\omoduli[{g,n}](\mu)$ is defined as  
$$ \phi(X,\omega) \= h^0\Bigl(X, \sum_{i=1}^n \frac{m_i}{2} x_i\Bigr) \mod 2\,.
$$
The parity is constant in a connected family of flat surfaces
by~\cite{Mumfordtheta}. We will denote by $\Om\M_{g,n}(\mu)^{\bullet}$
with $\bullet \in\{ {\rm odd},{\rm even}\}$ the moduli spaces of flat
surfaces with a fixed odd or even spin parity.  
Note that for $\mu = (2g-2)$ with $g\geq 4$ and $\mu = (g-1, g-1)$ with $g \geq 5$ odd, one of the two spin moduli spaces is disconnected, since it contains an extra hyperelliptic component (see \cite[Theorem~2]{kz03}).
Moreover, we will denote
by $\obarmoduli[{g,n}](\mu)^{\bullet}$ their incidence variety compactification
and will similarly use this symbol, e.g.\ in the form $v(\mu)^\bullet$,
$c_{1\leftrightarrow 2}(\mu,\mathcal{C})^\bullet$, and $c_{\area}(\mu)^{\bullet}$
for volumes and Siegel-Veech constants. 
\par
To state the refined version of the volume recursion we need a generalization
of the spin parity. Let $(\Gamma,\ell,\bfp)$ be a backbone graph.
A {\em spin assignment} is a function
$$
\phi\colon \{ v \in V(\Gamma), \ell(v)=0\} \to \{ 0, 1\}.
$$
The {\em parity of the spin assignment} is defined as 
\be \label{eq:spin-parity}
\phi(\Gamma) \,:=\, \!\!\! \sum_{v \in V(\Gamma), \ell(v)=0}
\!\!\! \phi(v)\,\,\,\,\,\, {\rm mod} \,\, 2.
\ee
Our goal is the following refinement of Theorems~\ref{intro:VolRec}
and~\ref{intro:IntFormula} under a
mild assumption. Recall that the tautological line bundle $\mathcal{O}(-1)$ over 
$\proj\omoduli[g,n](\mu)$ has a natural hermitian metric given by
$h(X,\omega)=\frac{i}{2} \int_X \omega\wedge \overline{\omega}$.
\par
\begin{asu} \label{asu}
  There exists a desingularization $f\colon Y\to \proj\obarmoduli[g,1](2g-2)$ such that $f^{*}h$
  extends to a good hermitian metric on $f^{*}\mathcal{O}(-1)$.
\end{asu}
\par
This assumption was already present in~\cite{SauvagetMinimal} and will
be proved in an appendix using the sequel to \cite{strata}.
Note that we do not need this assumption for Theorem~\ref{intro:VolRec}, 
as it is stated for the entire stratum whose cohomology class 
was computed recursively in~\cite{SauvagetClass}. However, currently we do 
not know the cohomology class of each individual spin component. 
\par
\begin{Thm} \label{thm:refinedVR}
If $n\geq 2$, then the rescaled volumes satisfy the recursion  
\bas \label{eq:volrefrec}
v(\mu)^{\rm odd}  \=
\sum_{k \geq 1}
\sum_{\bfg, \bfmu, \, \phi \, {\rm odd} }
  h_{\PP^1}((m_1,m_2),\bfp)
\cdot  \frac{\prod_{i=1}^k (2g_i-1+n(\mu_i))!v(\mu_i,  p_i-1)^{\phi(i)}}
      {{2^{k-1}k!(2g-3+n)!}}\,,
\eas
where the summation conventions for $\bfg$, $\bfmu$ and $\bfp$
are as in Theorem~\ref{intro:VolRec} and where 
the superscript $\phi(i)$ indicates the corresponding spin component.  
\end{Thm}
\par
We remark that the same formula holds when replacing ``odd'' by ``even'' in the theorem, which follows simply by subtracting the formula in Theorem~\ref{thm:refinedVR} from that in Theorem~\ref{intro:VolRec}. 
\par
\begin{Thm} \label{thm:refinedINT}
Let $\proj\Om\M_{g,n}(\mu)^\bullet$ with $\bullet \in\{ {\rm odd},{\rm even}\}$
be the connected component(s) of $\proj\Om\M_{g,n}(\mu)$ with a fixed spin
parity. Then the volume can be computed as an intersection number 
$$ {\rm vol}(\omoduli[g,n](\mu)^\bullet)
\= -\frac{2(2i\pi)^{2g}}{(2g-3+n)!}
\int_{\proj \obarmoduli[g,n](\mu)^\bullet} \xi^{2g-2}\cdot \prod_{i=1}^{n} \psi_i\,.
$$
\end{Thm}
\par
We first show in Section~\ref{sec:intspin} that the intersection numbers on the
right-hand side of Theorem~\ref{thm:refinedINT} satisfy a recursion as in
Theorem~\ref{thm:refinedVR}. This is parallel to Section~\ref{sec:RelMZInt}.
We then complete in Section~\ref{sec:SBHur} properties of the strict brackets
introduced by \cite{eop}. The volume recursion in Section~\ref{sec:spinD2}
is parallel to Section~\ref{sec:D2rec} and allows efficient
computations of volume differences of the spin components. We do not need to prove an analog
of Section~\ref{sec:D2VR} but can rather apply the results, since the structures of
the two recursions are exactly the same as before. Only in Section~\ref{subsec:spin-conclusion} we need Assumption~\ref{asu} to prove the beginning case of Theorem~\ref{thm:refinedVR}, i.e. the case of the minimal strata. 

\subsection{Intersection theory on connected components of the strata} \label{sec:intspin}

With a view toward Section~\ref{sec:VRhyp} for the hyperelliptic components, 
we allow here also the profile $\mu=(g-1,g-1)$ (with $g-1$ not necessarily even)
and $\bullet \in\{ {\rm odd},{\rm even}, {\rm hyp}\}$, and study the corresponding
union of connected components $\proj\omoduli[g,n](\mu)^\bullet$.
\par 
Let $(\Gamma,\ell,\bfp)$ be a twisted bi-colored graph and $D$ be an irreducible
component of the boundary $\proj\oOmM_{\Gamma,\ell}^{\bfp}$. We recall from
Section~\ref{sec:RelMZInt} that $\zeta_{\Gamma,\ell}^{\#}(D)$ is a divisor
of $\proj\obarmoduli[g,n](\mu)$ if and only if $\dim(\proj\oOmM_{\Gamma,\ell}^{\bfp})
=\dim(\proj\obarmoduli[g,n](\mu))-1$. Hence in this case we define
$\alpha(D)=\zeta_{\Gamma,\ell*}^{\#}(D)\in H^*(\proj\oOmM_{g,n})$, and define $\alpha(D)=0$ otherwise.  
\par
We will denote by $\proj\oOmM_{\Gamma,\ell}^{\bfp,\bullet}$ the union of the irreducible
components of $\proj\oOmM_{\Gamma,\ell}^{\bfp}$ that are mapped to
$\proj\obarmoduli[g,n](\mu)^\bullet$. 
\par
\begin{Prop} \label{eq:comprecursion}
For $1\leq i\leq n$ and each irreducible component $D$ of $\proj\oOmM_{\Gamma,\ell}^{\bfp,\bullet}$,
there exist constants $m_i^\bullet(D)\in \QQ$ such that
$$
(\xi+(m_i+1)\psi_i) [\proj\obarmoduli[g,n](\mu)^\bullet]
\= \sum_{\begin{smallmatrix} (\Gamma,\ell,\bfp)\\ i\mapsto v, \ell(v)=-1 \end{smallmatrix}}
\sum_{D\subset \proj\oOmM_{\Gamma,\ell}^{\bfp,\bullet}} m_i^\bullet(D) \, \alpha(D)\,,
$$
where the sum is over all twisted bi-colored graphs $(\Gamma,\ell,\bfp)$ with
the $i$-th marking in the lower level.  Moreover, if
$D\subset \proj\oOmM_{\Gamma,\ell}^{\bfp}$  and $(\Gamma,\ell,\bfp)$ is a backbone
graph, then $m_i^\bullet(D)=m(\bfp)$ is the multiplicity defined in~\eqref{eq:multtwist}.
\end{Prop}
\par
\begin{proof} We follow the same strategy as in~\cite[Theorem~5]{SauvagetClass}.
We consider the line bundle $\cOO(1)\otimes \cLL_i^{\otimes (m_i+1)}$ 
on $\proj \obarmoduli[g,n](\mu)^\bullet$. It has a global section~$s$ defined 
by mapping a differential to its $(i+1)$-st order at the marked point $x_i$. 
The vanishing locus of this section is exactly the union of the boundary 
components $\zeta_{\Gamma,\ell}^{\#}(\proj\oOmM_{\Gamma,\ell}^{\bfp})$,
thus proving the first part of the proposition.
\par
If $(\Gamma,\ell,\bfp)$ is a backbone graph, then each irreducible component  $D$ of $\proj\oOmM_{\Gamma,\ell}^{\bfp,\bullet}$
is contained in the boundary of exactly one connected component of the stratum (see e.g.~\cite[Corollary 4.4]{chenPB}). 
Thus the neighborhood of a generic point of $D$ in $\proj\omoduli[g,n](\mu)^\bullet$ is given by \cite[Lemma~5.6 and the subsequent formula]{SauvagetClass}.
In particular the multiplicity of $D$ in the vanishing locus of~$s$ is the 
same as that of the entire boundary stratum, which implies that $m_i^\bullet(D)=m(\bfp)$. 
\end{proof}
\par
By the same arguments as in Section~\ref{ssec:reduction} 
(see Proposition~\ref{pr:intbb}) one can show that
$D\cdot \xi^{2g-2}\neq 0$ only if $D$ is an irreducible component 
of $ \proj\oOmM_{\Gamma,\ell}^{\bfp}$ with $(\Gamma,\ell,\bfp)$ a backbone graph. 
If $(\Gamma,\ell,\bfp)$ is a backbone graph, then we let 
$\alpha_{\Gamma,\ell, \bfp}^{\bullet}=\zeta_{\Gamma,\ell*}^{\#}
[\proj\oOmM_{\Gamma,\ell}^{\bfp,\bullet}]$. Besides,  we let
$a_i(\mu)^{\bullet}= \int_{\proj\obarmoduli[g,n](\mu)^\bullet} \beta_i \cdot \xi$ 
for $1\leq i\leq n$. 
\par
\begin{Prop}\label{pr:refinedintrec}
 If $n \geq 2$, then the values of $a_i(\mu)^{\bullet}$ are the same for all $1\leq i \leq n$, denoted by $a(\mu)^{\bullet}$, and can be computed as
$$
 (m_1+1)(m_2+1) a(\mu)^\bullet \= \!\!\!\!\sum_{(\Gamma,\ell,\bfp)\in {\rm BB}_{1,2}}
 \! \frac{m(\bfp)}{|{\rm Aut}(\Gamma,\ell,\bfp)|}  \int_{\proj\oOmM_{g,n}} \!\!\! \!\!\!\alpha_{\Gamma,\ell, \bfp}^{\bullet}\cdot \xi^{2g-1}\cdot\prod_{i>2} \psi_i\,.
$$
\end{Prop}
\par
\begin{proof} This follows from the same argument as in the proof of Lemma~\ref{lem:indbis}. 
\end{proof}
\par
\begin{Prop} \label{pr:refinedintrecodd}
For $\mu$ of length bigger than one and with even entries, we have 
\bas
(m_1+1) (m_2+1) a(\mu)^\odd & \= \!\!\!
\sum_{\begin{smallmatrix}(\Gamma,\ell,\bfp,\phi),\\ \phi \, {\rm odd} \end{smallmatrix}} \frac{h_{\proj^1}((m_1,m_2),\bfp)}{|{\rm Aut}(\Gamma,\ell,\bfp,\phi)|}\\
& \, \cdot\,\prod_{v\in V(\Gamma),\,\ell(v)=0} p_v(2g_v-1+n(\mu_v)) a(\mu_v, p_v-1)^{\phi(v)}\,,
\eas
where the sum is over all choices of backbone graphs with only the first two
marked points in the lower level component. 
\end{Prop} 
\par
This proposition is a refined combination of Lemma~\ref{lem:inttree} and 
equation~\eqref{eq:ind3}.
Again we remark that the same formula holds when replacing ``odd'' by ``even'' 
in the proposition, which simply follows from subtracting the above from 
the corresponding formula for the entire stratum.  
\par
\begin{proof}
We apply Proposition~\ref{eq:comprecursion} to $\proj\oOmM_{g,n}(\mu)^\odd$.
The proposition then follows from the description of the boundary divisors of connected components of $\proj\oOmM_{g,n}(\mu)$. 
\par
Let $(L\to \mathcal X\to \Delta)$ be a one-parameter family of theta characteristics, i.e.\ 
$L$ is a line bundle such that $L|_X^{\otimes 2}\simeq \omega_{X}$ for every 
fiber curve $X$ parametrized by a complex disk (centered at the origin) 
such that
\begin{itemize}
\item  the restriction of $\mathcal X$ to $\Delta\setminus 0$ is a family 
of smooth curves;
\item  the central fiber $X_0$ is of compact type.
\end{itemize}
We assume that $L$ is odd, i.e. $L$ restricted to every smooth fiber is an odd theta characteristic. The restriction of $L$ to each irreducible component of $X_0$ (minus the nodes) is a theta characteristic of that component. Since $X_0$ is of compact type, the parity of $L|_{X_0}$ equals the sum 
of the parities over all irreducible components of $X_0$ (see e.g.~\cite[Proposition 4.1]{chenPB}), which implies that the number of components of $X_0$ with an odd theta characteristic is odd. 
\par 
Let $(\Gamma,\ell,\bfp)\in {\rm BB}(g,n)_{1,2}$. From the above description we deduce that $\proj\oOmM_{\Gamma,\ell}^{\bfp, \rm odd}$ can be written as 
$$
\bigcup_{\phi\,\odd} \oM_{-1}\times \prod_{v \in V(\Gamma)} \proj\oOmM_{g_v,n_v}(p_v-1,\mu_v)^{\phi(v)}\,
$$
where $\oM_{-1}$ and $\proj\oOmM_{g_v,n_v}(p_v-1,\mu_v)$ are defined as in
Section~\ref{ssec:boundary}. The arguments in the proof of 
Lemma~\ref{lem:indbis}  imply that
\bas
(m_1+1) (m_2+1) a(\mu)^\odd &\= \!\!\!\sum_{\begin{smallmatrix}(\Gamma,\ell,\bfp) \in  {\rm BB}_{1,2}\\ \phi \, {\rm odd} \end{smallmatrix}} \frac{h_{\proj^1}((m_1,m_2),\bfp)}{|{\rm Aut}(\Gamma,\ell,\bfp,\phi)|}\\
&\, \cdot\, \prod_{v\in V(\Gamma),\,\ell(v)=0} p_v^2\ a(\mu_v, p_v-1)^{\phi(v)}\,.
\eas
Then by the same line of arguments as in Section~\ref{ssec:treestrata} 
(expansions over rooted trees), we get the desired expression.
\end{proof}

\subsection{Strict brackets and Hurwitz numbers with spin parity} \label{sec:SBHur}

Let $f\colon \SP \to \QQ$ be any function on the set of strict partitions
(i.e.\ partitions with strictly decreasing part lengths). 
The replacement of the $q$-bracket in the context of spin-weighted
counting is the {\em strict bracket} defined by 
\bes
\bstr{f} \=  \frac{1}{(q)_\infty}
\sum_{\lambda \in \SP} (-1)^{\ell(\lambda)} f(\lambda) q^{|\lambda|}\,,
\qquad \bigl((q)_\infty \= \prod_{n \geq 1} (1-q^n)
= \sum_{\lambda \in \SP} (-1)^{\ell(\lambda)} q^{|\lambda|}\bigr) \,.
\ees
The analog of the algebra $\Lambda^*$  is the algebra $\bfLA^* =
\QQ[\bfp_1,\bfp_3,\bfp_5,\ldots]$ of {\em supersymmetric
functions}, where for odd~$\ell$ the functions~$\bfp_\ell$ are defined by
\bes \label{eq:defboldpf}
\bfp_\ell(\lambda) \= \sum_{i=1}^\infty \lambda_i^\ell - \frac{\zeta(-\ell)}{2}\,.
\ees
Note the modification of the constant term and the absence of the shift
in comparison to~\eqref{eq:defpk}. We provide $\bfLA^*$ with the
{\em weight grading} by declaring $\bfp_\ell$ to have weight~$\ell+1$.
On the other hand, \cite{eop} used characters of the modified Sergeev
group  $C(d) = S(d) \ltimes \Cliff(d)$ to produce elements in $\bfLA^*$.
Here $\Cliff(d)$ is generated by involutions $\xi_1,\ldots,\xi_d$ and
a central involution~$\ve$ with the relation $\xi_i \xi_j = \ve \xi_j \xi_i$. 
Irreducible representations of $C(d)$ are~$V^\lambda$ indexed by
$\lambda \in \SP$. We denote by $\bff_\mu(\lambda)$ the central character
of the action of a permutation $g_\mu \in S(d) \subset C(d)$ of cycle
type $\mu$ on~$V^\lambda$ by conjugation. The analog of the Burnside
formula is \cite[Theorem~2]{eop} stating that for a fixed profile
$\Pi = (\mu_1,\ldots,\mu_n)$ 
\be \label{eq:spinwtsum}
\sum_p \frac{(-1)^{\phi(p)} q^{\deg(p)}}{|\Aut(p)|}
\= 2^{\sum_{i=1}^n (\ell(\mu_i) - |\mu_i|)/2} \bstr{\bff_{\mu_1} \bff_{\mu_2} \cdots \bff_{\mu_n}}\,,
\ee
where the sum is over all covers $p\colon X \to E$ of a fixed base curve
and profile~$\Pi$. 
\par
\begin{Thm} \label{thm:bfformula}
If we define 
\be \label{eq:bfhldef} 
\bfh_\ell \= \frac{-1}{\ell}  [u^{\ell+1}] \bfP(u)^\ell 
\quad \text{where} \quad
\bfP(u) \= \exp \Bigl( -\sum_{s \geq 1, s\,\odd } u^{s+1} \bfp_s \Bigr)\,,
\ee
then the difference $\bff_\ell - \ell \bfh_\ell$ has weight strictly less
than~$\ell+1$. In particular $\bff_\ell$ belongs to the
subspace $\bfLA^*_{\leq \ell+1}$ of weight less than or equal to $\ell+1$.
More precisely,
\be \label{eq:bffprecise}
\bff_{(\ell)} \= \frac{-1}{2\ell} [t^{k+1}] \Biggl(\prod_{j=1}^{\ell-1} (1-jt)
\exp \Bigl(\sum_{j \,\, \odd} \frac{2p_jt^j}{j}(1-(1-\ell t)^{-j})
\Bigr) \Biggr)\,.
\ee
\end{Thm}
\par
This statement was missing in the proof of the following corollary,
one of the main theorems of \cite{eop}.
\par
\begin{Cor} The strict bracket $\bstr{\bff_{\ell_1}
\bff_{\ell_2} \cdots \bff_{\ell_n}}$ is a quasimodular form of mixed
weight less than or equal to $\sum_{i=1}^n (\ell_i + 1)$.
\end{Cor}
\par
We now prepare for the proof of Theorem~\ref{thm:bfformula} and prove the
corollary along with more precise statements on strict brackets in the next
subsection. From \cite[Definition~6.3 and Proposition~6.4]{IvaGauss} we know
that the central characters are given by 
\be \label{eq:ftoPd}
\bff_\rho \=  \sum_{\mu \in \SP} X^\rho_\mu  P_\mu^\downarrow\,
\ee
where the objects on the right-hand side are defined as follows. 
We define for any partition $\lambda$
the Hall-Littlewood symmetric polynomials
\bes
P_\lambda(x_1,\ldots,x_m; t) \= 
\sum_{\sigma \in S_n}
\sigma\Bigl(\prod_{i=1}^n x_i^{\lambda_i} \prod_{i<j,\, i< \ell(\lambda)}\frac{x_i -tx_j}{x_i -x_j}\Bigr)\,.
\ees
These polynomials have cousins where the
powers are replaced by falling factorials. That is,  writing
$n^{\downarrow k} = n (n-1)(n-2)\cdots (n-k+1)$, we define
\bes
P^\downarrow_\lambda(x_1,\ldots,x_m; t) = 
\sum_{\sigma \in S_n}
\sigma\Bigl(\prod_{i=1}^n x_i^{\downarrow\lambda_i}
\prod_{i<j,\, i< \ell(\lambda)}\frac{x_i -tx_j}{x_i -x_j}\Bigr)\,.
\ees
Next, we define $X^\bullet_\bullet(t)$ to be the base change matrix
from the basis of $\bfp_\rho$ to the basis $P_\lambda(x_1,\ldots,x_m; t)$,
that is, we define them by
\be \label{eq:ptoPwitht}
\bfp_\rho  \= \sum_{\lambda \vdash |\rho|}
X^\rho_\lambda(t) P_\lambda(\,\cdot\,; t)\,.
\ee
The existence and the fact that the $X^\bullet_\bullet(t)$ are polynomials
in~$t$ is shown in \cite[Section~III.7]{mac}. We abbreviate $X^\rho_\lambda=
X^\rho_\lambda(-1)$ and similarly $P_\lambda = P_\lambda(\,\cdot\,; -1)$ and
$P^\downarrow_\lambda = P^\downarrow_\lambda(\,\cdot\,; -1)$.
\par
\begin{proof}[Proof of Theorem~\ref{thm:bfformula}] We need to
prove~\eqref{eq:bffprecise}.
From there one can then derive~\eqref{eq:bfhldef} by expanding the exponential
function (just as in \cite{IvaOlsh} Proposition~3.5 is derived from
Proposition~3.3). We use that for $\rho=(\ell)$ a cycle, the coefficients
$X^\lambda_\rho$ in~\eqref{eq:ftoPd} are supported on~$\lambda$ with at most
two parts. More precisely, by \cite[Example~III.7.2]{mac} we know that
\ba \label{eq:fbasPdown}
\bff_{(\ell)} &\=  P^\downarrow_{(\ell)} + 2 \sum_{i=1}^{\lfloor \ell/2\rfloor}
(-i)^i P^\downarrow_{(\ell-i,i)} \\
&\= \sum_{1 \leq a,b \leq \ell(\lambda) \atop a\neq b} 
\sum_{i=0}^\ell (-1)^i  \lambda_a^{\downarrow \ell-i} \lambda_b^{\downarrow i}
 \frac{\lambda_a +\lambda_b}{\lambda_a -\lambda_b}
\prod_{i \neq a,b}  \frac{\lambda_a +\lambda_i}{\lambda_a -\lambda_i}
 \frac{\lambda_b +\lambda_i}{\lambda_b -\lambda_i} \,.
\ea
Using 
\bas
-\sum_{j \in \NN} \frac{1-(-1)^jp_j(\lambda)t^j}{j}(1-\ell t)^{-j}
&\=
\log \Bigl( \prod_{i=1}^{\ell(\lambda)} \frac{1-(\lambda_i+ \ell)t}
{1+(\lambda_i-\ell)t} \Bigr) \\ 
\eas
and the specialization of this formula for $\ell=0$, our goal is to
show that
\bas
\bff_{(\ell)}(\lambda) &\= \frac{-1}{2\ell} [t^{1}] \Biggl(\prod_{j=0}^{\ell-1} (t^{-1}-j)
\prod_{i=1}^{\ell(\lambda)} \frac{t^{-1} +\lambda_i}{t^{-1}-\lambda_i} 
 \frac{t^{-1} -(\lambda_i+\ell)} {t^{-1} +(\lambda_i-\ell)} 
\Biggr) \\
&\= \frac{-1}{2\ell} [z^{-1}] \Biggl(\prod_{j=0}^{\ell-1} (z-j)
\prod_{i=1}^{\ell(\lambda)} \frac{z +\lambda_i}{z -\lambda_i} 
 \frac{z -(\lambda_i+\ell)} {z +(\lambda_i-\ell)} 
\Biggr) \\
 &\= \sum_{a=1}^{\ell(\lambda)} \Bigl( 
\prod_{j=0}^{\ell-1} (\lambda_a-j)
\prod_{i \neq a}  \frac{\lambda_a +\lambda_i}{\lambda_a -\lambda_i}
\frac{\lambda_a -(\lambda_i+\ell)} {\lambda_a +(\lambda_i-\ell)}
\Bigr)\,.
\eas
Using that for $\ell$ odd 
\bes
\sum_{i=0}^\ell (-1)^i \bigl( \lambda_a^{\downarrow \ell-i} \lambda_b^{\downarrow i}
\,-\, \lambda_b^{\downarrow \ell-i} \lambda_a^{\downarrow i}\bigr)
\= \frac{ \lambda_a^{\downarrow \ell}(\lambda_a - \lambda_b -\ell) -   
\lambda_b^{\downarrow \ell} (\lambda_b - \lambda_a -\ell)}
{\lambda_a+\lambda_b-\ell}\,,
\ees
we see that our goal and the known~\eqref{eq:fbasPdown} agree.
\end{proof}

\subsection{Volume computations via cumulants for strict brackets}
\label{sec:spinD2}

We denote by an upper index~$\Delta$ the difference of the even and
odd spin related quantities, e.g.\ $v(\mu)^\Delta = v(\mu)^\even - v(\mu)^\odd$.
Cumulants for strict brackets are defined by the same formula~\eqref{slash}
as for $q$-brackets. We are interested in cumulants for the same reason
as we were for the case of the strata as in~\eqref{eq:cumutovol}.
\par
\begin{Prop} \label{prop:spinvol}
The difference of the volumes of the even and odd spin components of $\omoduli[g,n](\mu)$
can be computed in terms of cumulants by
\bes
\vol(\omoduli[g,n](\mu))^\Delta \= \frac{(2\pi i)^{2g}}{(2g-2+n)!}
\bsL{\bff_{(m_1+1)}| \cdots |\bff_{(m_n+1)}}\,,
\ees
and thus, in combination with Theorem~\ref{thm:bfformula} we have 
\bes
v(\mu)^\Delta \= \frac{(2\pi i)^{2g}}{(2g-2+n)!}
\bsL{\bfh_{(m_1+1)}| \cdots |\bfh_{(m_n+1)}}\,.
\ees
Here the subscript~$L$ refers to the leading term
\bas \label{eq:bfdeflead}
\la g_1|\cdots|g_n\ra_{\str,L} &\= [\hslash^{-k-1+n}] \,\la g_1|\cdots|g_n\ra_{\str,\hslash} 
\= \lim_{h\to 0} \hslash^{k+1-n}\,\evh[\langle g_1|\cdots|g_n\rangle_{\str,q}](\hslash)\,
\eas
for $g_i$ homogeneous of weight $k_i$ and $k = \sum_{i=1}^n k_i$.
\end{Prop}
\par
This proposition was certainly the motivation of \cite{eop}, which stops short
of this step. To derive the proposition from~\eqref{eq:spinwtsum}, we need one
more tool, the analog of the degree drop in Proposition~\ref{prop:degdrop}.
We use the fact that for strict brackets we have a closed formula (rather than only
a recursion as for $q$-brackets), proved in~\cite[Section~3.2.2]{eop} and
in more detail in~\cite[Section~13]{blochokounkov}. First, 
\be \label{eq:plstrict}
(-1) \cdot \bstr{\bfp_\ell} \= G_{\ell +1}  \,:=\, \frac{\zeta(-\ell)}{2}
+ \sum_{n \geq 1} \sigma_{\ell}(n) q^n
\ee
and for the more general statement we define the ``oddification'' of the
Eisenstein series to be
\bes
\cGG^\odd(z_1,\ldots,z_n) 
\= - \sum_{r=1}^\infty D_q^{(n-1)}G_{2r} \sum_{s_1+\cdots+s_n = r+n-1}
\frac{z_1^{2s_1-1}\cdots z_n^{2s_n-1}}{(2s_1-1)!\cdots(2s_n-1)!}\,,
\ees
where $D_q = q\partial/\partial q$. 
Then by Proposition~13.3 in loc.\ cit.\ the $n$-point function is given by
\be \label{eq:strnpoint}
\sum_{\ell_i \geq 1, \,\, \ell_i \,\odd} \bstr{\bfp_{\ell_1} \bfp_{\ell_2} \cdots \bfp_{\ell_n}}
\frac{z_1^{\ell_1}\cdots z_n^{\ell_n}}{\ell_1!\cdots \ell_n!} 
\= \sum_{\alpha \in \PPP(n)} \, \prod_{A\in\alpha}\,
\cGG^\odd\Bigl(\{z_a\}_{a \in A}\Bigr)\,.
\ee
Consequently, the cumulants are simply given by 
\be \label{eq:strcumulants}
\bstr{\bfp_{\ell_1}| \bfp_{\ell_2}| \cdots |\bfp_{\ell_n}}
\= \Bigl[\frac{z_1^{\ell_1}\cdots z_n^{\ell_n}}{\ell_1!\cdots \ell_n!} \Bigr]
\cGG^\odd(z_1,\ldots,z_n)\,.
\ee
\par
\begin{proof}[Proof of Proposition~\ref{prop:spinvol}]
Note that
$\deg\big(\evX \bstr{\bfp_{\ell_1} \bfp_{\ell_2} \cdots \bfp_{\ell_n}}\big) =
\frac12 \sum_{i=1}^n \ell_i$, the highest term being contributed by the partition
into singletons. From~\eqref{eq:strcumulants} we deduce that
$\deg\big(\evX \bstr{\bfp_{\ell_1}| \bfp_{\ell_2}| \cdots |\bfp_{\ell_n}}\big) =
\frac12 \sum_{i=1}^n \ell_i - (n-1)$, and thus obtain the expected degree drop. The claim
now follows from the usual approximation of Masur-Veech volumes by counting
torus covers (\cite{eo} and~\cite[Proposition~19.1]{cmz}).
\end{proof}
\par
\medskip
While~\eqref{eq:strcumulants} provides an easy and efficient way to
compute cumulants of strict brackets, we show that the more complicated
way via lifting of differential operators to~$\bfLA^*$ and the Key Lemma~\ref{KL:D2recursion}
also works here. The analog of Proposition~\ref{prop:liftviadelta} is the following result.  
\par
\begin{Prop} \label{prop:strhviaDiffOp}
With $\Delta(f) \= \sum_{\ell_1,\,\ell_2\,\ge\,1} (\ell_1+\ell_2) \,
\bfp_{\ell_1+\ell_2-1}\,\frac{\p^2}{\p \bfp_{\ell_1}\,\p \bfp_{\ell_2}}$ we have
\be
\la f \ra_{\str,\hslash} \= \frac{1}{\hslash^k}(e^{\hslash(\Delta -\p/\p \bfp_1)/2} f)\,
(\emptyset)\,,
\ee
where the  evaluation at the empty set is explicitly given by $\bfp_\ell \mapsto
-\frac{\zeta(-\ell)}{2}$.
\end{Prop}
\par
Note that we can regard the differential operator $(\Delta -\p/\p \bfp_1)/2$
appearing in the exponent the same as the operator~$D$ defined
in~\eqref{eq:defD} 
when viewing $\bfLA^*$ as a quotient algebra of $\Lambda^*$ with
all the even~$p_\ell$ set to zero, since the differential operator $\partial$
sending $p_\ell$ to a multiple of $p_{\ell-1}$ is zero on this quotient.
\par
\begin{proof} Using the description~\eqref{eq:defevh} of the $\hslash$-evaluation
we need to show that $\fd \la f \ra_{\str,\hslash} =  \la (\Delta -\p/\p \bfp_1) f
\ra_{\str,\hslash}$. Contrary to the case of $q$-brackets we will actually show the
stronger statement that $\fd \la f \ra_{\str} =  \la (\Delta -\p/\p \bfp_1) f
\ra_{\str}$. It suffices to check this for all the $n$-point functions.
For $n=1$ this can be checked directly from~\eqref{eq:plstrict}. For general $n$, we write
$W(z) = \sum_{s \geq 1} z^{2s-1}/(2s-1)!$. Using~\eqref{eq:strnpoint} and that
the commutator $[\fd,D_q]$ is multiplication by the weight, we compute that
\begin{flalign} 
\fd \la \prod_{i=1}^n W(z_i) \ra_{\str} &\=
\sum_{\alpha \in \PPP(n)} \sum_{A_1 \in \alpha, \atop |A_1| \geq 2}
\Bigl(\fd \cGG^\odd\bigl(\{z_a\}_{a \in A_1}\bigr) \cdot \!\!\!\!
\prod_{A \in \alpha \setminus \{A_1\}} \cGG^\odd\bigl(\{z_a\}_{a \in A}
\bigr)\Bigr)
\nonumber \\
&\phantom{\=} \,-\, \frac12 \sum_{i=1}^n z_i \cdot \Bigl(
\sum_{\alpha \in \PPP(\{1,\ldots, n\} \setminus \{i\})}
\prod_{A \in \alpha} \cGG^\odd\bigl(\{z_a\}_{a \in A}\bigr)\Bigr)
\label{eq:fdFn}
\end{flalign}
where for the factor in the summand with $|A_1| \geq 2$
\be \label{eq:fdGodd}
\fd \cGG^\odd\bigl(\{z_a\}_{a\in A_1}\bigr) \= - \frac12 
\sum_{r \geq 1} 2r D_q^{|A_1|-2} G_{2r} \cdot \sum_{s_a \geq 1, \atop
  \sum s_a = r + |A_1|-1} \prod_{a \in A_1}
\frac{z_a^{2s_a-1}}{(2s_a-1)!}\,,
\ee
and where the summation is over all tuples~$(s_a)_{a \in A_1}$.
Since $\tfrac12 \p/\p \bfp_1 (\prod_{i=1}^n W(z_i)) = \tfrac12
\sum_{i=1}^n z_i \prod_{j\neq i} W(z_j)$, the strict bracket of this expression
is precisely the second line on the right-hand side of~\eqref{eq:fdFn}.
Since
\bes
\frac12 \Delta\Big(\prod_{i=1}^n W(z_i)\Big) = \sum_{1 \leq i\neq j \leq n}
(z_i+z_j) W(z_i+z_j) \prod_{k \in \{1,\ldots,n\} \setminus \{i,j\}} W(z_k)\,,
\ees
its strict bracket matches the first line on the right-hand side of~\eqref{eq:fdFn},
and the part containing the variable for $W(z_i+z_j)$ produces of course the
special factor~\eqref{eq:fdGodd}.
\end{proof}
\par
This proposition provides an efficient algorithm to compute the
differences of volumes of the spin components. The definitions below are completely analogous 
to the beginning of Section~\ref{sec:D2rec}, except that objects with even indices
have disappeared and they are written in boldface letters for distinction.
For the substitution, we define
\bes
\bfP_Z(u) \= \exp\Bigl(\sum_{ j \geq 1} \Bigl(\frac{1}{2}\Bigr)^{\frac{j+1}{2}}
\zeta(-j)  \,u^{j+1}\Bigr)
\quad \text{and} \quad
\ual_\ell \= \bigl[u^{\ell}\bigr] \frac{1}{(u/\bfP_Z(u))^{-1}}\,.
\ees
We let $\bfR = \QQ[\bfh_1,\bfh_3,\ldots]$ and define for a finite set
$I=\{i_1,\ldots,i_n\}$ of positive  integers the formal series $\bfcH_I\in
\bfR[[z_{i_1}, \ldots, z_{i_n}]]$ by
\begin{flalign}
&\bfcH_{\{i\}} \= \frac{1}{z_i} + \sum_{\ell\geq 1} \bfh_\ell z_i^\ell\,, 
\quad  
\bfcH_{\{i,j\}} \= \frac{z_i \bfcH'(z_i) - z_j\bfcH'(z_j)}
{\bfcH(z_j) - \bfcH(z_i)} -1\,,   \nonumber  \\
&\bfcH_{I} = \frac{1}{2(n-1)}\sum_{I=I'\sqcup I''}
D_2(\bfcH_{I'},\bfcH_{I''})\,, \label{eq:bfHindform}
\end{flalign}
with
\bes
D_2(f,g) \= \sum_{\ell_1,\ell_2 \geq 1,\,\, \odd} [z_1^{\ell_1} z_2^{\ell_2}] \bfcH_{\{1,2\}}\,
\frac{\partial f}{\partial \bfh_{\ell_1}}\,\frac{\partial g}{\partial \bfh_{\ell_2}} \,.
\ees
We still set $\bfcH_{n}=\bfcH_{[\![1,n]\!]}$ and 
$\bfh_{\ell_1,\ldots,\ell_n} = [z_1^{\ell_1}\cdots z_n^{\ell_n}]\bfcH_{n}$.
\par
\begin{Cor} \label{cor:eodiff}
The even-odd volume differences of the stratum with signature $\mu=(m_1,\ldots,m_n)$ can
be computed as 
\bes
v(\mu)^\Delta \=  \frac{(2\pi i)^{2g}}{(2g-2+n)! } \,
\bigl.\bfh_{m_1+1,\ldots,m_n+1}\bigr|_{\bfh_\ell \mapsto \ual_\ell} 
\ees
using the recursion~\eqref{eq:bfHindform}.
\end{Cor}
\par
\begin{proof} Thanks to Proposition~\ref{prop:strhviaDiffOp} and the subsequent
remark, the proof of Theorem~\ref{thm:D2recursion} can be copied verbatim here.
The extra factor $2^{-\ell/2}$ in the definition of $\bfP_Z$ in comparison to the
constant term $-\zeta(-\ell)/2$ of the evaluation of $\bfp_\ell$ compensates for
the fact that the strict bracket of the $\bff_\ell$ gives the counting function
in~\eqref{eq:spinwtsum} up to a power of two.
\end{proof}

\subsection{Conclusion of the proofs for spin components} \label{subsec:spin-conclusion}

\begin{proof}[Proof of Theorem~\ref{thm:refinedINT} and Theorem~\ref{thm:refinedVR}]
Theorem~\ref{thm:refinedVR} is a consequence of Corrolary~\ref{cor:eodiff}. Indeed the arguments of Section~\ref{sec:D2VR} adapted to the series $\bfcH_n$ show that the recursion in Theorem~\ref{thm:refinedVR} is a consequence of the recursion in~\eqref{eq:bfHindform}. 
\par
To prove Theorem~\ref{thm:refinedINT}, we consider first the case $n=1$, i.e. $\mu = (2g-2)$. Let $\ol{\nu}_{\mu}$ be the 
push-forward of the Masur-Veech volume $\nu_\mu$ form to $\proj\OmM_{g,n}(\mu)$.
Assumption~\ref{asu} implies by the same argument as in
\cite[Lemma~2.1]{SauvagetMinimal}
that $\frac{2(2i\pi)^{2g}}{(2g-1)!}\xi^{2g-1}$  can be represented by a
meromorphic differential form (of Poincar\'e growth at the boundary), 
whose restriction to $\proj \Om\M_{g,1}(2g-2)$ is equal to $\ol{\nu}_{\mu}$. 
This implies that 
$$
{\rm vol}(2g-2)^\bullet \= \frac{2(2i\pi)^{2g}}{(2g-1)!}
\int_{\proj \obarmoduli[g,1](2g-2)^\bullet} \xi^{2g-1}\,.
$$
Now for the case $n \geq 2$ Theorem~\ref{thm:refinedVR} and 
Proposition~\ref{pr:refinedintrecodd} determine $a(\mu)^{\bullet}$ 
and $v(\mu)^{\bullet}$ by the equivalent recursive formulas, and hence 
they coincide up to the obvious normalizing factors.  
\end{proof}
\par
\subsection{Volume recursion for hyperelliptic components}
\label{sec:VRhyp}

In this subsection we prove the volume recursion for hyperelliptic
components, which is  analogous to but not quite the same as the recursion
in Theorem~\ref{intro:VolRec}.
It is a consequence of the work of Athreya, Eskin and Zorich (\cite{aez})
on volumes of the strata of quadratic differentials in genus zero. 
\par
Recall that only the strata $\omoduli[g](g-1,g-1)$ and $\omoduli[g](2g-2)$
have hyperelliptic components. For the hyperelliptic components we still
have an interpretation of their volumes as intersection numbers
as well as a volume recursion as follows.  
\par
\begin{Thm} \label{thm:hypint}
For the hyperelliptic components we have 
$$
{\rm vol}(\omoduli[g,1](2g-2)^\hyp) \= \frac{2(2i\pi)^{2g}}{(2g-1)!}
\int_{\proj \obarmoduli[g,1](2g-2)^\hyp} \xi^{2g-1}\,
$$
and
$$
{\rm vol}(\omoduli[g,2](g-1,g-1)^\hyp) \= \frac{2(2i\pi)^{2g}}{g(2g-1)!}
\int_{\proj \obarmoduli[g,2](g-1,g-1)^\hyp} \xi^{2g-1}\psi_2\,,
$$
provided that Assumption~\ref{asu} holds.
\end{Thm}
\par
As before we set
\bas \ 
v(2g-2)^\hyp &\= (2g-1){\rm vol}\,(\omoduli[{g,1}](2g-2)^\hyp)\,, \\
v(g-1,g-1)^\hyp &\= g^2\,{\rm vol}\,(\omoduli[{g,2}](g-1,g-1)^\hyp) \,.
\eas
\par
\begin{Prop} \label{prop:HRvolrec}
The volumes of the hyperelliptic components $\omoduli[g,2](g-1,g-1)$
satisfy the recursion
\bas
v(g\ms 1,g\ms 1)^\hyp = v(2g\ms 2)^\hyp   
+ \sum_{\ell = 1}^{g-1}  \frac{(2\ell\ms1)!v(2\ell\ms2)^\hyp\, (2(g\ms\ell)\ms1)!
v(2g\ms2\ell\ms2)^\hyp}
{4 (2g-1)!}\,.
\eas
\end{Prop}
\par
Note in comparison to Theorem~\ref{intro:VolRec} that only the terms $k=1$
and $k=2$ appear and that the Hurwitz number $h_\PP$ is identically one here.
As a preparation for the proof recall that the canonical double cover construction
provides isomorphisms
\bas
\cQQ_g(2g-3,(-1)^{2g-3}) &\,\cong\, \omoduli[{g,1}](2g-2)^\hyp\,, \\
\cQQ_g(2g-2,(-1)^{2g-2}) &\,\cong\, \omoduli[{g,2}](g-1,g-1)^\hyp
\eas
that preserve the Masur-Veech volume and the $\SL_2(\RR)$-action. Taking into account 
the factorials for labeling zeros and poles the main result of \cite{aez} can
be translated as 
\bas
{\rm vol}\,(\omoduli[{g,1}](2g-2)^\hyp) \= \frac{2}{(2g+1)!} \frac{(2g-3)!!}{(2g-2)!!} \pi^{2g}\,, \\
{\rm vol}\,(\omoduli[{g,2}](g-1,g-1)^\hyp) \= \frac{8}{(2g+2)!} \frac{(2g-2)!!}{(2g-1)!!} \pi^{2g}
\eas
where the double factorial notation means $(2k)!! = 2^k k!$ and $(2k-1)!! = (2k)!/2^kk!$.
\par
\begin{proof} Expanding the definition of the double factorials and including the
summand $v(2g-2)^\hyp$ as the two boundary terms of the sum 
(i.e. $\ell = 0$ and $\ell=g$), we need to show that
\be
\sum_{\ell=0}^{g} \frac{1}{2\ell+1}\binom{2\ell}{\ell}
\frac{1}{2g-2\ell+1} \binom{2(g-\ell)}{g-\ell}
\= 2 \frac{16^g}{(g+1)^2} {\binom{2g+2}{g+1}}^{-1} \,.
\ee
For this purpose it suffices to prove the following two identities of generating series 
\be \label{eq:arctan1}
\sum_{\ell \geq 0}  \frac{1}{2\ell+1}\binom{2\ell}{\ell} x^{2\ell}
\= \frac{1}{2x} \arctan\Biggl(\frac{2x}{\sqrt{1-4x^2}}\Biggr)
\ee
and 
\be \label{eq:arctan2}
2\, \sum_{g \geq 0} \frac{16^g}{g^2} {\binom{2g}{g}}^{-1} x^{2g} 
\= \frac{1}{4x^2} \arctan\Biggl(\frac{2x}{\sqrt{1-4x^2}}\Biggr)^2\,,
\ee
so that we can take the square of the first series and compare the $x^{2g}$-terms. To
prove~\eqref{eq:arctan1} we multiply it by~$x$, differentiate, and are then left with
showing that $\sum_{\ell \geq 0}  \binom{2\ell}{\ell} x^{2\ell} = 1/\sqrt{1-4x^2}$,
which follows from the binomial theorem. To prove~\eqref{eq:arctan2} we
differentiate and are then left with the identity which is already proved in \cite[p.~452, Equation~(9)]{lehmer}.
\end{proof}
\par
The last ingredient is the following straightforward consequence of Proposition~\ref{pr:refinedintrec} (analogous to the case of spin components in Proposition~\ref{pr:refinedintrecodd}). 
\par
\begin{Prop}
\label{prop:hyp-int-recursion}
For $\mu=(g-1,g-1)$, we have  
\ba
g^2 a(g-1, g-1)^\hyp & \=  (2g-1)^2 a(2g-2)^\hyp \\
&\+  \frac{1}{2} \sum_{g_1=1}^{g-1} \bigg(
h_{\proj^1}\big((g-1,g-1),(2g_1-1, 2g-2g_1-1)\big)  \\
& \quad 
\cdot (2g_1-1)^2a(2g_1-2)^\hyp(2g-2g_1-1)^2 a(2g-2g_1-2)^\hyp\bigg)\,.
\ea
\end{Prop} 
\par
\begin{proof}[Proof of Theorem~\ref{thm:hypint}]
Since the proof of Theorem~\ref{thm:hypint} for the case $\mu=(2g-2)$ was already given
along with the proof of Theorem~\ref{thm:refinedINT}, it remains to show
that $v(g-1,g-1)^\hyp$ and $a(g-1,g-1)^\hyp = \int_{\proj\obarmoduli[g,2](g-1,g-1)^\hyp} \beta_i \cdot \xi$
satisfy the same recursion. It is elementary to check 
that $h_{\proj^1}\big((g-1,g-1),(2g_1-1, 2g-2g_1-1)\big) = 1$.
Then the desired conclusion thus follows from Propositions~\ref{prop:HRvolrec} and~\ref{prop:hyp-int-recursion}. 
\end{proof}



\section{An overview of Siegel-Veech constants}
\label{sec:SV}

Let $(X,\omega)$ be a flat surface, consisting of a Riemann surface~$X$ 
and an Abelian differential~$\omega$ on~$X$. Siegel-Veech constants
measure the asymptotic growth rate of the number of saddle connections
(abbreviated s.c.) or cylinders with bounded length (of the waist curve)
in~$(X,\omega)$. There are many variants that we now introduce and compare.

\subsection{Saddle connection and area Siegel-Veech constants}

For each pair of zeros~$(z_1,z_2)$ of $\omega$ we let
\be \label{eq:Aphy}
A^\phy_{1 \lra 2}(T) \= |\{\gamma \subset X \,\, \text{a saddle connection
joining $z_1$ and $z_2$},\, \Bigl|\int_\gamma \omega \Bigr| \leq T\}|
\ee
be the counting function. The upper index emphasizes that we count
all physically distinct saddle connections. It should be distinguished
from the version
\be \label{eq:Aho}
A^\ho_{1 \lra 2}(T) \= |\{\gamma \subset X \,\, \text{a homology
  class\ of s.c.\ joining $z_1$ and $z_2$},\, \Bigl|\int_\gamma \omega \Bigr|
\leq T\}|\,,
\ee
where a collection of homologous saddle connections just counts for one.
Quadratic upper and lower bounds for such counting
functions were established by Masur (\cite{masur90}). Fundamental works of
Veech (\cite{veech98}) and Eskin-Masur (\cite{eskinmasur}) showed that for
almost every flat surface $(X,\omega)$ in the sense of the Masur-Veech measure
(see \cite{masur82} and \cite{veech82}) there is a quadratic asymptotic, i.e. that
\be \label{eq:defscSV}
A^\phy_{1 \lra 2}(T) \, \sim \, c^\phy_{1 \lra 2}(X,\omega) \,\pi T^2\,,
\qquad A^\ho_{1 \lra 2}(T) \, \sim \, c^\ho_{1 \lra 2}(X,\omega) \,\pi T^2\,.
\ee
The constants $c^\phy_{1 \lra 2}(X,\omega)$ and $c^\ho_{1 \lra 2}(X,\omega)$
are the first type of Siegel-Veech constants we study here, called the
{\em saddle connection Siegel-Veech constants}. The difference between
these two Siegel-Veech constants becomes negligible as the genus of~$X$
tends to infinity, which follows from the results of Aggarwal and
Zorich (see~\cite[Remark 1.1]{Agg2}).
\par
\medskip
The second type of Siegel-Veech constants counts homotopy classes of closed
geodesics, or equivalently flat cylinders. Again, there are two variants, the
naive count and the count where each cylinder is weighted by
its relative area. As above, the most important counting function
with good properties (see e.g.\ \cite{cmz}) and connection to Lyapunov
exponents (\cite{ekz}) is the second variant. For the precise definition
we consider
\be \label{eq:Narea}
A_{\cyl}(T) \= \sum_{Z \subset X \text{cylinder} \atop w(Z) \leq T} \, 1\,,
\qquad
A_{\area}(T) \= \sum_{Z \subset X \text{cylinder} \atop w(Z) \leq T}
\frac{\area(Z)}{\area(X)}\,,
\ee
where $w(Z)$ denotes the width of $Z$, i.e. the length of its core curve. We then 
define the {\em cylinder Siegel-Veech constant} and the
{\em area Siegel-Veech constant} by the asymptotic equalities
\be \label{eq:SVarea}
A_{\cyl}(T) \, \sim \, c_{\cyl}(X,\omega){\pi T^2} \,, \qquad 
A_{\area}(T) \, \sim \, c_{\area}(X,\omega){\pi T^2} \,.
\ee
\par
There is a natural action of $\GL_2(\RR)$ on the moduli space of flat
surfaces $\omoduli[g]$ and the orbit closures are nice submanifolds,
in fact linear in period coordinates by the fundamental work of Eskin, Mirzakhani and Mohammadi (\cite{esmi} and \cite{esmimo}). We
refer to them as {\em affine invariant manifolds}, using typically the
letter~$\cMM$. The intersection with the hypersurface of area one flat surfaces
(denoted by the same symbol~$\cMM$) comes with a finite $\SL_2(\RR)$-invariant
ergodic measure $\nu_\cMM$ with support~$\cMM$. This measure is unique up 
to scale and for affine invariant manifolds defined over~$\QQ$ there 
are natural choices of the scaling.
\par
It follows from the Siegel-Veech axioms 
(see \cite{eskinmasur}) that Siegel-Veech constants for almost all
flat surfaces $(X,\omega)$ in an $\SL_2({\RR})$-orbit closure~$\cMM$ agree.
We call these surfaces {\em generic} (for $\cMM$) and write e.g.\
$ c_{1 \lra 2}^\star(\cMM) = c_{1 \lra 2}^\star(X,\omega)$ for $(X,\omega)$ generic.
\par
The relevant orbit closures in this paper are the connected components of
the strata of Abelian differentials and certain Hurwitz spaces inside the strata.
We usually abbreviate by $c_{1 \lra 2}^\star(\mu) =
c_{1 \lra 2}^\star(\omoduli[g,n](\mu))$ the Siegel-Veech constants for the strata with signature $\mu$.

\subsection{Configurations and the principal boundary}
\label{ssec:configurations}

One of the main insights of \cite{emz} is that Siegel-Veech constants
can be computed separately according to topological types, called
configurations. We formalize their notion of configurations briefly so that
it also applies to Hurwitz spaces, and in fact to all $\SL_2(\RR)$-orbit
closures~$\cMM$ provided with the generalization of the Masur-Veech
measure~$\nu_\cMM$. The concept of configurations will be used for showing
the equivalence between Theorem~\ref{intro:SVproduct} and
Theorem~\ref{intro:VolRec} in Section~\ref{sec:VRandSC}.
\par
Let $(X,z_1,\ldots,z_n)$ be a pointed topological surface. A 
{\em configuration $\cCC$ of saddle connections joining~$z_1$ and~$z_2$}
for~$\cMM$ is a set of simple non-intersecting arcs from~$z_1$ to~$z_2$
up to homotopy preserving the cyclic ordering of the arcs both at~$z_1$
and~$z_2$. The last condition implies that the tubular neighborhood of
the configuration is a well-defined subsurface of~$X$, in fact a
{\em ribbon graph~$R(\cCC)$} associated with the configuration.
The number of arcs in the configuration is called the {\em multiplicity}
of the configuration.
\par
We say that the saddle connections of length~$\leq T$ joining~$z_1$
and~$z_2$ on a flat surface $(X,\omega)$ {\em belong} to the
configuration $\cCC$, if the set of these saddle connections 
is homotopic to~$\cCC$.
Each configuration gives rise to a counting function $A_{1 \lra 2}^\star
(T,\cCC)$ for saddle connections belonging to the configuration and to the
corresponding Siegel-Veech constant $c_{1 \lra 2}^\star(\cMM,\cCC)$, where
$\star \in \{\phy,\ho\}$ respectively. A configuration~$\cCC$ is {\em relevant}
if $c^\star_{1 \lra 2}(\cMM,\cCC) >0$.
\par
A {\em full set of saddle connection  configurations} for an affine invariant
manifold $\cMM$ is a finite set of saddle connection configurations $\cCC_i$,
with $i\in I$ such that the contributions of the configurations $\cCC_i$ sum up to
the full Siegel-Veech constant, i.e.\ such that
\be \label{eq:SV12sum}
\sum_{i \in I} c^\star_{1 \lra 2}(\cMM,\cCC_i) \= c^\star_{1 \lra 2}(\cMM)
\ee
for $\star \in \{\phy,\ho\}$ respectively.
\par
Note that \cite[Section~3.2]{emz} in their definition of configurations
made a further subdivision of the notion by adding metric data, i.e.\ specifying angles
between saddle connections. In that context, Eskin, Masur and Zorich determined a full
set of saddle connection configurations for the strata and used
the Siegel-Veech transform to connect the computation to volume
computations. The following statement summarizes 
Proposition~3.3, Corollary~7.2 and Lemma~8.1 of \cite{emz}.
\par
\begin{Prop} \label{prop:emzsummary}
For any stratum~$\omoduli[g,n](\mu)$ a full set of saddle connection
configurations is the set of collections of pairwise homologous simple
disjoint arcs joining~$z_1$ and $z_2$ (up to homotopy).
\end{Prop}
\par
In this proposition, several configurations are irrelevant, for example
those with a connected component of genus zero after removing the saddle connections in the configuration. 
\par
The general strategy to compute Siegel-Veech constants is the following
relation to volumes, where the submanifold $\cMM$ is in
a stratum with labeled zeros.
\par
\begin{Prop} \label{prop:SVviaboundary}
  The saddle connection Siegel-Veech constants of an
affine invariant manifold~$\cMM$ can be computed as 
\be \label{eq:SVviaboundary}
c_{1 \lra 2}^\star (\cMM) \= \lim_{\ve \to 0} \frac1{\pi\ve^2} \sum_\cCC
m^\star(\cCC)\,
\frac{\nu_{\cMM}(\cMM^\ve(\cCC))}
{\nu_{\cMM}(\cMM)}\,,
\ee
where the sum runs over the full set of saddle connection configurations
and where $m^\ho(\cCC) = 1$ for all $\cCC$ while $m^\phy(\cCC)$ is equal
to the number of arcs in~$\cCC$.
\end{Prop}
\par
\begin{proof}
This is a direct consequence of the Siegel-Veech transform applied
to the characteristic function of a disc of radius~$\ve$, see
\cite[Lemma~7.3]{emz} together with the Eskin-Masur bound on the
number of short saddle connections (Theorem on p.~84 of \cite{emz}).
\end{proof}
\par
We conclude with remarks on Siegel-Veech constants for general affine 
invariant manifolds to put the digression on Hurwitz spaces 
(Section~\ref{sec:NewHurwitz}) in context. There is another variant, 
besides~\eqref{eq:Aphy} and~\eqref{eq:Aho}, of counting saddle connections. 
Given an affine invariant manifold~$\cMM$ we say that two saddle connections on
$(X,\omega) \in \cMM$ are {\em $\cMM$-parallel} if they are parallel
and stay parallel in a neighborhood of $(X,\omega)$ in~$\cMM$. 
(The terminology is completely analogous to the notion of {\em $\cMM$-parallel
cylinders} introduced in \cite{WrightCyl}.) We thus define 
the counting function $A^\Mp_{1 \lra 2}(T)$ and the Siegel-Veech constant
$c_{1 \lra 2}^\Mp (\cMM)$
in analogy to~\eqref{eq:Aho} and~\eqref{eq:defscSV}, counting once
every $\cMM$-parallel class of cylinders. \cite[Proposition~3.1]{emz}
can now be restated as $c_{1 \lra 2}^\Mp (\cMM) = c_{1 \lra 2}^\ho (\cMM)$
if $\cMM$ is a connected component of a stratum.
For Hurwitz spaces the two values can be different, but we will see
(Proposition~\ref{prop:admissible}) that their
difference becomes negligible as the degree of the covers tends to infinity.
\par
\medskip
In the first part of \cite{emz} on recursive computations of
Siegel-Veech constants, Eskin, Masur and Zorich called the locus of degenerate
surfaces that contribute to the Siegel-Veech counting the
{\em principal boundary}. At that time the notion of principal boundary 
was used only as a partial topological compactification. Presently, we 
dispose of a complete and geometric  compactification for the strata 
(\cite{strata}) and for Hurwitz spaces (by admissible covers), 
and we can then identify the principal boundary as part of the compactification 
(see \cite{chenPB} for the case of the strata and Section~\ref{sec:NewHurwitz} 
for the case of Hurwitz spaces).
The reader should keep in mind that the locus ``principal boundary''
depends on the type of saddle connections under consideration.
\par
\medskip
Finally we remark that there is a zoo of possibilities
of associating weights with saddle connections and cylinders and to define 
Siegel-Veech constants accordingly. This started with \cite{vorobets}, 
and see also \cite{baugou} for computations and conversions of Siegel-Veech
constants. 


\section{Saddle connection Siegel-Veech constants}
\label{sec:VRandSC}

In this section we deduce from the volume recursion and its refinement for
spin and hyperelliptic components a proof of Theorem~\ref{intro:SVproduct}.
Almost all we need here has been proven already in \cite{emz}. We
start with two more auxiliary statements.
\par
\begin{Prop} \label{prop:fullsetSC}
The full set of saddle connection configurations for the strata
given in Proposition~\ref{prop:emzsummary} is in bijection with
(possibly unstable) 
backbone graphs and a cyclic ordering of its vertices at level zero. 
The subset of relevant configurations is in bijection
with stable backbone graphs. 
\end{Prop}
\par
Although not needed in the sequel, we relate for the convenience of the
reader our notion of twists and the angle information that \cite{emz} recorded. Let $\{\gamma_i\}_{i=1}^k$ be the arcs of
a configuration realized by~$(X,\omega)$ labeled cyclically
and let $a_i'$ and $a_i''$ be the angles between $\gamma_i$ and $\gamma_{i+1}$
at $z_1$ and $z_2$ respectively. If an edge~$e$ is separated by the loop
formed by $\gamma_i$ and $\gamma_{i+1}$, then the twist is
\be\label{eq:angle}
\bfp(e) \= \tfrac{1}{2\pi}(a_i' + a_i'') - 1\,. 
\ee
\par
The above proposition can be seen as follows. Given a collection of~$k$ homologous short saddle
connections there is a sphere (with $z_1$ as its south pole and $z_2$
as its north pole, see~\cite[Figure~5]{emz}) supporting the~$k$
saddle connections. This sphere is the source of the backbone of the graph. The components
at level zero are bounded by the arcs $\gamma_i$ and $\gamma_{i+1}$.
The converse is obvious, given that all the edges of a (stable) backbone graph are separating by definition.
Finally formula~\eqref{eq:angle} is just a restatement of the Gauss-Bonnet theorem.
\par
Recall that a backbone graph (being of compact type) is compatible
with a unique twist $\bfp(\cdot)$. If the vertices at level zero are
labeled as $1,\ldots, k$ as usual and if $(h_j, i(h_j))$ is the edge connecting
the $j$-th vertex to level~$-1$, we write  $p_j = |\bfp(h_j)|$ as we
did in Section~\ref{sec:RelMZInt}.
\par
\begin{Prop} \label{prop:boundaryvol}
For each configuration~$\cCC$ corresponding to a backbone graph~$\Gamma$,
the volume of the subspace $\omoduli[g,n]^\ve(\mu, \cCC)$ satisfies
\bas
& \phantom{\=}\,\, \sum_{k \geq 1} \sum_{\bfg, \bfmu} \pi \ve^2
\frac{h_{\PP^1}((m_1,m_2),\bfp)}
{k! \, (2g-3+n)!}
\cdot \prod_{i=1}^k
{(2g_i-1+n_i)!}\,{p_i}\,\vol(\omoduli[g_i, n_i+1](\mu_i,p_i-1)) \\
&\= 2^{1-k} \vol(\omoduli[g,n]^\ve(\mu, \cCC)) \,+\, o(\ve^2)\,
\eas
as $\ve \to 0$, where the summation conventions and $\bfp$ are
as in Theorem~\ref{intro:VolRec} and where $n_i = n(\mu_i)$. 
\end{Prop}
\par
\begin{proof} This is mainly contained in~\cite[Corollary~7.2, Formulas~8.1
and~8.2]{emz}, stating that the volume of the locus with an $\ve$-short
configuration is $\pi\ve^2$ times the volume of the corresponding boundary.
We now explain the combinatorial factors that appear. First, the factor of
two and the factorials result from the passage of the boundary volume
element in the ambient stratum to the product of the volume elements
of the components at the boundary, as explained in detail
in \cite[Section~6]{emz}. The $1/k!$ stems from labelling the level zero
vertices. Second, we need to count the ways to obtain a surface in
$\omoduli[g,n]^\ve(\mu, \cCC)$ by gluing a collection of surfaces
$(X_i,\omega_i)$ in $\omoduli[g_i,n_i+1](\mu_i,p_i-1)$.
\par
Suppose we are given a branched cover $b\colon \PP^1 \to \PP^1$ that has
ramification profile $\Hmu= ((m_1+1), (m_2+1),(p_1,\ldots,p_k))$ over the
points $0,1$ and $\infty$. We provide the domain with the differential
$\omega_{-1} = b^* dz$. Since this differential has no
residues we can glue $t \cdot \omega_{-1}$ with the surfaces $(X_i,\omega_i)$
by cutting the pole of order $p_i+1$ and gluing it to an
annular neighborhood of the zero of order $p_i-1$ of $\omega_i$.
For $t \leq \ve$ this provides a surface in $\omoduli[g,n]^\ve(\mu, \cCC)$,
see e.g.\ \cite[Section~4]{strata} for details of the construction.
The plumbing construction also depends on the choice of a $p_i$-th root of
unity at each pole (from the choice of a horizontal slit at a zero of order
$p_i-1$).  In total there are $h_{\PP^1}((m_1,m_2),\bfp)\cdot \prod_{i=1}^k p_i$
possibilities involved in the construction, thus justifying the remaining
combinatorial factors in the formula.
\par
We claim that this construction provides a collection of maps to
$\omoduli[g,n]^\ve(\mu, \cCC)$ that are almost everywhere injections
if none of the
surfaces  $(X_i,\omega_i)$ has a period of length smaller than~$\ve$.
In fact, if two such plumbed surfaces are isomorphic, this isomorphism
restricts to an automorphism of $(\PP^1, \omega_{-1})$ (see
\cite[Lemma~8.1]{emz} for more details) and this happens only on
a measure zero set. The locus where one of the $(X_i,\omega_i)$ has a short
period is subsumed in the $o(\ve^2)$ (\cite[Lemma~7.1]{emz}). Conversely,
for each surface~$(X,\omega)$
in $\omoduli[g,n]^\ve(\mu, \cCC)$ we can cut a ribbon graph around
the configuration $\cCC$. The restriction of $\omega$ has no periods
since the boundary curves are homologous by definition of~$\cCC$.
It can thus be integrated and completed to a map $b\colon \PP^1 \to \PP^1$
with ramification profile as above. 
\end{proof}
\par
\begin{proof}[Proof of Theorem~\ref{intro:SVproduct}]
We first focus on the case that $\omoduli[{g,n}](\mu)$ is connected.
A decomposition $g = \sum_{i=1}^k g_i$ and 
$(m_3,\ldots, m_n)= \mu_1\sqcup \cdots \sqcup  \mu_k$ (as in equation~\eqref{eq:volintro} we proved)
determines uniquely a configuration and the converse is true up to
the labeling of the~$k$ vertices at level zero by Proposition~\ref{prop:fullsetSC}.
The configuration is relevant if and only if the volumes on the right-hand
side of~\eqref{eq:volintro} are non-zero. Since the saddle connection Siegel-Veech constant
is the sum of ratios of the boundary volumes over the total volume
(by Proposition~\eqref{prop:SVviaboundary}), comparing the formula in Proposition~\ref{prop:boundaryvol} with equation~\eqref{eq:volintro} thus implies
Theorem~\ref{intro:SVproduct}. More precisely, note that the rescaled volume $v(\mu)$ defined
in~\eqref{eq:vol-rescale} involves a product of all $(m_i+1)$, while on the
right-hand side $(m_1+1)(m_2+1)$ is missing and this factor gives the right
hand side of the desired formula in Theorem~\ref{intro:SVproduct}. The factor
of~$\pi$ in Proposition~\ref{prop:boundaryvol} cancels with the one in
Proposition~\ref{prop:fullsetSC}.
\par
For disconnected strata with components parameterized by
$S \subseteq \{\odd,\even,\hyp\}$ the same proof gives the averaged version that 
\bes
\frac{1}{v(\mu)} \, \sum_{\bullet \in S} v(\mu)^\bullet \,c^\ho_{1\lra 2}(\mu)^\bullet
\= (m_1 +1)(m_2 + 1)
\ees
by rewriting equation~\eqref{eq:SVviaboundary} as a sum over the
connected components. Moreover, Theorem~\ref{thm:refinedVR} and Proposition~\ref{prop:HRvolrec} give analogous volume recursions for the spin components and the hyperelliptic components. It remains to show that the summations
in these recursions correspond to configurations for the spin and hyperelliptic components respectively.  For the spin components it follows from the
additivity of the Arf-invariant on stable curves of compact type
(see \cite[Proposition~4.6]{chenPB} and \cite[Lemma~10.1]{emz}).
For the hyperelliptic components the relevant degenerations are explained
in \cite[Proposition~4.3]{chenPB} and \cite[Lemma~10.3]{emz}.
\end{proof}
\par


\section{The viewpoint of Hurwitz spaces}
\label{sec:NewHurwitz}

This section is a digression on how to interpret the volume
recursion and the saddle connection Siegel-Veech constant
from the viewpoint of Hurwitz spaces. The results in this
section are not needed for proving any of the theorems stated in the
introduction. We will rather explain and motivate
\begin{itemize}
\item why the homologous count of saddle connections is more natural
  than the physical count from the viewpoint of intersection theory,
\item how to heuristically deduce the value of the saddle connection
Siegel-Veech constant in Theorem~\ref{intro:SVproduct} from an equidistribution
of cycles in Hurwitz tuples, and
\item why backbone graphs correspond to configurations.
\end{itemize}
\par
As usual $\mu = (m_1,\ldots,m_n)$ is a partition of $2g-2$ and the
ramification profile~$\Hmu$ consists of~$n$ cycles $\mu^{(i)}$ of
length $m_i +1$ unless stated otherwise. Let  $r(\mu) = 2g + n(\mu) - 1$
be the dimension of the (unprojectivized) stratum $\omoduli[g](\mu)$. 
\par
\begin{Thm} \label{thm:H2}
There exists a constant $M(\mu)$ such that the Hurwitz numbers~$N_d^\circ(\Hmu)$
for connected torus covers of profile~$\Hmu$ can be approximated as
\begin{flalign} 
&\phantom{\=} M(\mu)\,N_d^\circ(\Hmu) \label{eq:H2} \\
& \=  \sum_{\Gamma \in \BB^\star_{1,2}} \frac{d}{k!}\,h_{\PP^1}((m_1,m_2),\bfp)
\sum_{d_1+\cdots +d_k = d}
\prod_{i=1}^k p_i N^\circ_{d_i}(\Hmu_i,(p_i)) +  o(d^{r(\mu)-1})\,, \nonumber
\end{flalign}
where $\Hmu \setminus \{\mu^{(1)}, \mu^{(2)}\}
\= \Hmu_1 \sqcup \cdots \sqcup \Hmu_k$ is the
decomposition of the profile according to the leg assignment in~$\Gamma$
and where $\bfp = (p_1,\ldots p_k)$ is the unique 
twist compatible with~$\Gamma$.
\end{Thm}
\par
At the end of the section we will show by combining Theorem~\ref{intro:VolRec}
together with Theorem~\ref{thm:H2} that $M(\mu)= (m_1+1)(m_2+1)$. Indeed an
independent proof of this equality would provide an alternative proof of
Theorem~\ref{intro:VolRec} (and hence Theorem~\ref{intro:IntFormula}) that would 
bypass the complicated combinatorics in Sections~\ref{sec:D2rec} and~\ref{sec:D2VR}.
\par
The strategy to prove Theorem~\ref{thm:H2} consists of comparing the Hurwitz
number~$N_d^\circ(\Hmu)$, that is the fiber cardinality of the forgetful
map $f_T\colon H_d(\Hmu) \to \moduli[1,n]$ to the target curve with the
fiber cardinality of the extension of $f_T$ to the space of admissible
covers $\BH_d(\Hmu) = \BH_{d,g,1}(\Hmu)$ over degenerate targets of the
following type. 
\par
\subsection{Admissible torus covers}
\label{sec:NotHurwitz}
Let $\ES$ be the stable curve of genus one consisting of a
$\PP^1$-component carrying precisely the first two marked points and of
an elliptic curve~$E$ carrying the remaining marked points, joint
at a node~$q_E$. If $p\colon X \to \ES$ is an admissible cover, we denote by~$X_0$
and $X_{-1}$ the (possibly reducible) curves mapping to~$E$ and to $\PP^1$
respectively, both deprived of their unramified~$\PP^1$-components.
See Figure~\ref{fig:exadmcov} for examples of such admissible covers.
\par
The admissible covers of $\ES$ come in two types. One possibility is
that the first two branch points are in the same (hence the unique)
component of $X_{-1}$. The stable dual graph of the cover is thus
a graph $\Gamma \in \ABB^\star_{1,2}$ and we denote by $N^{\circ}_d(\Pi,\ES,\Gamma)$
the number of such covers. The second possibility is that each of the two
branch points is on its own component~$X_{-1}^{(i)}$ for $i=1,2$. Consequently,
$p|_{X_{-1}^{(i)}}$ is a cyclic cover of degree~$(m_i+1)$. By 
contracting the components over $\PP^1$ we see that such covers
are (up to the automorphism group of size
$|\Aut(X_{-1}/\PP^1)|  = (m_1+1)(m_2+1)$) in bijective correspondence with
covers of~$E$ with the profile $\Hmu_{(12)} = ((\mu^{(1)}, \mu^{(2)}),\mu^{(3)},
\ldots,\mu^{(n)})$, where the first two ramification points are piled over the same branch point. 
\par
\begin{Prop}\label{prop:admissible}
For $\Gamma \in \ABB^\star_{1,2}$ the bound $N^{\circ}_d(\Pi,\ES,\Gamma)  
\= O(d^{r(\mu)-2})$ holds as $d\to \infty$. The upper bound
is attained, i.e.\ there exists $C>0$ such that 
$N^{\circ}_d(\Pi,\ES,\Gamma)  \,\geq \,C d^{r(\mu)-2}$, 
if and only if $\Gamma \in \BB^\star_{1,2}$. Moreover in this case 
\be \label{eq:hur-adm}
N^{\circ}_d(\Pi,\ES,\Gamma) \=  \frac{1}{k!} h_{\PP^1}((m_1,m_2),\bfp)
\sum_{d=d_1+\cdots +d_k} \prod_{i=1}^k N^\circ_{d_i}(\Hmu_i,(p_i)) \,,
\ee
where $\Hmu_i$ and $\bfp$ are associated with~$\Gamma$ as in 
Theorem~\ref{thm:H2}. 
\end{Prop}
\par
\begin{proof} Recall from \cite{eo} or the proof of~\cite[Proposition~9.4]{cmz} that if $\Hmu$ is the profile for a cover $\pi$ with $\pi^* \omega \in
\omoduli[g,n](\mu)$, then there exist $C_1,C_2 \neq 0$ such that
$C_1\cdot d^{r(\mu)-1} \leq N^{\circ}_d(\Hmu) \leq C_2\cdot d^{r(\mu)-1}$
as $d \to \infty$.
 Suppose that $\Gamma$ has $k$ components
on level $0$, each of genus $g_i$ and with $n_i$ marked points or nodes.
Let $g_0$ be the genus of the component on level $-1$, and let
$b = h^1(\Gamma)$. Then 
$$b + \sum_{i=0}^k g_i = g \quad \mbox{and}\quad \sum_{i=1}^k n_i
\= n-2 + k + b\,.$$
The cover of the~$\PP^1$-component has finitely many choices independent of $d$. Over the
elliptic component of~$\ES$, the number of choices of covers has asymptotic
growth given by $B \sum_{d_1+\cdots+d_k = d} d_i^{2g_i -2 +n_i}$ 
for some constant~$B$ independent of $d$. This quantity is a polynomial
of degree 
$$ \Bigl( \sum_{i=1}^k (2g_i - 2+ n_i) \Bigr) + (k-1)
\= 2g - 2g_0 - b + n - 3\,. $$
We thus conclude that the total number of admissible covers
$N^0_d(\Pi,\ES,\Gamma)$ has asymptotic growth  given by a polynomial of
degree $r(\mu)-2-b - 2g_0 \leq r(\mu)-2$, with equality attained if and only if 
$b = g_0 = 0$, i.e. if and only if $\Gamma\in \BB^\star_{1,2}$.
\par
To justify equation~\eqref{eq:hur-adm} we refer to the computation of the
Hurwitz numbers in Proposition~\ref{prop:intpsiPP1}
and divide by $k!$
to account for our auxiliary labeling of the $k$ components of~$X_0$.
\end{proof}
\par
Proposition~\ref{prop:admissible} reveals the geometric reason for homologous count of
saddle connections behind the recursions in Sections~\ref{sec:RelMZInt}
to~\ref{sec:D2VR}. The factor $h_{\PP^1}((m_1, m_2),\bfp)$ (possibly
with $1/k!$ if all branches are labeled) appears in the direct count
of admissible covers and in the count of configurations. There is no
extra factor~$k$ in~\eqref{eq:hur-adm}, which corresponds to our setting of 
the coefficient $m^{\hom}(\mathcal C) = 1$ (instead of $k$) in~\eqref{eq:SVviaboundary} 
for homologous count of saddle connections (instead of physical count).  
\par
\begin{proof}[Proof of Theorem~\ref{thm:H2}] We first show that
\ba \label{eq:H2app}
&\phantom{\=} N_d^\circ(\Hmu) -  N_d^\circ(\Hmu_{(12)}) \\
& \=  \sum_{\Gamma \in \BB^\star_{1,2}}  \frac{1}{k!} h_{\PP^1}((m_1, m_2),\bfp)
\sum_{d_1+\cdots +d_k = d}
\prod_{i=1}^k p_i N^\circ_{d_i}(\Hmu_i,(p_i)) +  O(d^{r(\mu)-2})\,. 
\ea
To see this, note that the ramification order of $f_T$ over $\ES$
at the branch through a cover~$\pi\colon X \to E $ is equal to the product of
ramification orders at the nodes of ~$X$.
This results in the factors $p_i$ inside the product of the right-hand
side and cancels the factor $1/|\Aut(X_{-1}/\PP^1)|$ when counting $\ES$-covers
instead of counting $N_d^\circ(\Hmu_{(12)})$.
\par
On the other hand, since the volume of the stratum can be approximated 
by counting covers of profile $\Hmu_{(12)}$ and since the generating function
of counting these covers is a quasimodular form, arguments as in
\cite[Proposition~9.4]{cmz} imply the existence of a constant
$M(\mu)$ such that
\ba\label{eq:hurwitz12}
 \frac{N_d^\circ(\Hmu_{(12)})}{N_d^{\circ}(\Hmu)} \=
\frac{d -  M(\mu)}{d} \,+ \,o (1/d )\,. 
\ea
The combination of equations~\eqref{eq:H2app} and~\eqref{eq:hurwitz12} thus implies the desired formula~\eqref{eq:H2}.
\end{proof}
\par
We now address the equidistribution heuristics for saddle connection
Siegel-Veech constants. Recall that $N^\circ(\Hmu)$ is the number
(weighted by $|\Aut(p)|$) of transitive Hurwitz tuples
$(\alpha,\beta,(\gamma_i)_{i=1}^n) \in S_d^{n+2}$
with $[\alpha,\beta] = \prod_{i=1}^n \gamma_i$ and $\gamma_i$ of type $\mu^{(i)}$.
\par
\begin{Prop} If the pairs $(\gamma_1,\gamma_2)$ appearing in the Hurwitz tuples
of profile $\Hmu$ equidistribute among pairs of $(m_i+1)$-cycles in $S_d^2$
as $d \to \infty$, then $M(\mu) = (m_1+1)(m_2+1)$.
\end{Prop}
\par
\begin{proof} If the non-trivial cycles in $\gamma_1$ and $\gamma_2$
have no letter in common, then taking $(\alpha,\beta,\gamma_1 \circ \gamma_2,
\gamma_3,\ldots,\gamma_n)$ is a Hurwitz tuple of profile $\Hmu_{(12)}$.
Assuming equidistribution and comparing to the total number of Hurwitz tuples, 
the number of Hurwitz tuples with
$\gamma_1$ and $\gamma_2$ having two letters in common is negligible,
while the ratio of those having one letter in common is $(m_1+1)(m_2+1)/d + o(1/d)$.
\end{proof}
\par
\begin{Ex} 
For the reader's convenience we illustrate the contributions to the
right-hand side of~\eqref{eq:H2app} for the stratum $\omoduli[2](1,1)$
explicitly in Figure~\ref{fig:exadmcov}.
\begin{figure}
\begin{tikzpicture}[scale=0.68]
\tikzstyle{every node}=[font=\scriptsize]

\begin{scope}
\begin{scope}[rotate=90]
%

\path[draw,rotate=0]{(0,0) -- +(65:2.75) arc (160:-20:.15cm) -- +(-115:2.41)
	arc (160:340:.15cm) -- +(65:2.63)
	arc (160:-20:.15cm) -- +(-115:2.4)	
	arc (160:340:.15cm) -- +(65:2.63)	
	arc (160:-20:.15cm) -- +(-115:2.4)
	arc (160:340:.15cm) -- +(65:2.63)	
	arc (160:-20:.15cm) -- +(-115:2.63)
};
\draw (2.10,0) -- +(65:2.85); 

\draw[] (0,5.2) --+ (-65:3.15)
(.65,5.2) --+ (-65:3.15)
(1.3,5.2) --+ (-65:3.15); 

\coordinate (P01) at (2.3,5.2);
\coordinate (P02) at (2.095,5.2);
\coordinate (P04) at (3.31,2.59);
\coordinate (P05) at (1.98,5.2); 

\path[draw] (P02)  --+(-65:.9) node(P03)[inner sep=0]{}  arc (200:380:.095cm) -- (P01);
\path[draw] (P04) arc (280:380:.15cm) --+(115:1.64) arc (20:200:.09cm) --+(-65:.615)
arc (390:188:.055cm) -- (P05);
\path[draw] (P04)  --+(115:1.04) arc (20:200:.05cm) node(P06)[inner sep=0]{} --+(-65:.898) 
arc (210:270:.17cm);

\draw[black,-latex] (1,2.58) -- +(180:.5); 
\draw (-1,5.2) --+ (-65:3.2); 
\draw (-.98,0) -- +(65:3.2); 
\draw (-.65,4.3) -- (-.5,4.3);
\begin{scope}[xshift=.42cm,yshift=-.91cm]
\draw (-.65,4.3) -- (-.5,4.3);
\end{scope}
\end{scope} 

\draw (0,-1.15) node(E)[]{$E$}; 
\draw (-5.1,-1.15) node(P)[]{$\PP^1$}; 
\end{scope}

\begin{scope}[xshift=6cm]
\begin{scope}[rotate=90]

\path[draw]{(0,0) -- +(65:2.75) 
	arc (160:-20:.15cm) -- +(-115:2.41)
	arc (160:340:.15cm) -- +(65:2.63) 
	arc (160:-20:.15cm) -- +(-115:2.4) 	
	arc (160:340:.15cm) -- +(65:2.63) 	
	arc (140:130:.85cm)
	(2.788,2.575) -- +(-115:2.63)
	arc (160:340:.15cm) -- +(65:2.63)	
	arc (160:-20:.15cm) -- +(-115:2.4)
	arc (160:340:.15cm) -- +(65:2.62)	
	arc (160:-20:.15cm) -- +(-115:2.62)
};

\draw[] (0,5.2) --+ (-65:3.15)
(.65,5.2) node(P3)[]{} --+ (-65:3.15) node(P1)[]{}
(1.95,5.2) --+ (-65:3.15)  
(2.6,5.2) --+ (-65:3.15);

\path[] (1.2,5.2) coordinate(P3)[inner sep=0]{} --+ (-65:3.15) coordinate(P1)[inner sep=0]{}
(1.46,5.2) coordinate(P4)[inner sep=0]{} --+ (-65:3.15) coordinate(P2)[inner sep=0]{};    
\draw (P1) -- +(115:.7)  arc (210:40:.12cm) -- (P2);
\draw[] (P3) -- +(295:1.153)  arc (200:365:.12cm) -- (P4);
\path[draw] (1.8,3.91) node(P5)[inner sep=0]{}  arc (210:40:.12cm) node(P6)[inner sep=0]{} --+(294.5:.83) node(P7)[]{} arc (380:200:.12cm) -- cycle;

\draw[-latex] (1,2.58) -- +(180:.5); 
\draw (-1,5.2) --+ (-65:3.2) (-.98,0) -- +(65:3.2); 
\draw (-.65,4.3) -- (-.5,4.3);
\begin{scope}[xshift=.42cm,yshift=-.91cm]
\draw (-.65,4.3) -- (-.5,4.3);
\end{scope}
\end{scope} 
\draw (0,-1.15) node(E)[]{$E$}; 
\draw (-5.1,-1.15) node(P)[]{$\PP^1$}; 
\end{scope}
%
\begin{scope}[xshift=12cm]
\begin{scope}[rotate=90]

\path[draw,rotate=0]{(0,0) -- +(65:2.75) arc (160:-20:.15cm) -- +(-115:2.41)
	arc (160:340:.15cm) -- +(65:2.63)
	arc (160:-20:.15cm) -- +(-115:2.4)	
	arc (160:340:.15cm) -- +(65:2.63)	
	arc (160:-20:.15cm) -- +(-115:2.4)
	arc (160:340:.15cm) -- +(65:2.63)	
	arc (160:-20:.15cm) -- +(-115:2.4)
	arc (160:340:.15cm) -- +(65:2.62)	
	arc (160:-20:.15cm) -- +(-115:2.62)};

\draw[] (0,5.2) --+ (-65:3.15)
(1.95,5.2) --+ (-65:3.15)  
(2.6,5.2) --+ (-65:3.15);

\path[] (.65,5.2) node(P3)[inner sep=0]{}  --+ (-65:3.15) node(P1)[inner sep=0]{}
(1.3,5.2) node(P4)[inner sep=0]{} --+ (-65:3.15) node(P2)[inner sep=0]{};    

\draw[] (P1) -- +(115:.65)  arc (210:40:.3cm) -- (P2);
\draw (P3) -- +(295:.4)  arc (200:365:.3cm) -- (P4);
\path[draw] (1.1,4.24) node(P5)[inner sep=0]{}  arc (210:40:.3cm) -- (1.63,4.48) node(P6)[inner sep=0]{} --+(295:.9) node(P7)[]{} arc (365:200:.3cm) -- cycle;
\draw (-1,5.2) --+ (-65:3.2) (-.98,0) -- +(65:3.2); 
\draw (-.65,4.3) -- (-.5,4.3);
\begin{scope}[xshift=.42cm,yshift=-.91cm]
\draw (-.65,4.3) -- (-.5,4.3);
\end{scope}
\end{scope} 
\draw (0,-1.15) node(E)[]{$E$}; 
\draw (-5.1,-1.15) node(P)[]{$\PP^1$}; 
\end{scope}
  
\end{tikzpicture}
\caption{Configurations for Hurwitz spaces in $\omoduli[2](1,1)$}
\label{fig:exadmcov}
\end{figure}

The picture on the left gives stable graphs in $\BB^\star_{1,2}$, while
the pictures in the middle and on the right give graphs in
$\ABB^\star_{1,2} \setminus \BB^\star_{1,2}$. The preimages
of~$E$ in the middle and on the right are unramified and thus again are elliptic
curves, while on the left the preimage of $E$ is a curve of genus two.
\par
The saddle connection Siegel-Veech counting in this case was carried out
in \cite{ems} in a similar way as summarized in Theorem~\ref{thm:H2},
despite that  only primitive torus covers were considered.
\par
\end{Ex}

\subsection{The principal boundary of Hurwitz spaces}
\label{sec:Pb}

We focus on saddle connections joining the first
two marked zeros and determine a full set of configurations and
the corresponding principal boundary of the Hurwitz spaces. 
We say that $\Gamma \in {\rm ABB}^\star_{1,2}$ is {\em realizable} in $\BH_d(\Hmu)$
if there is an admissible cover  $p\colon X \to \ES$ whose stable graph
is~$\Gamma$ and such that the vertices with $\ell(v)=0$ correspond
bijectively to the components of $X_0$. Recall also the definition of ribbon 
graphs associated to configurations in Section~\ref{ssec:configurations}.  
\par
\begin{Prop} \label{prop:ConfandGamma}
Associating with a configuration~$\cCC$ the boundary
curves of the ribbon graph~$R(\cCC)$ induces a map $\varphi\colon \cCC
\to \Gamma(\cCC)$ from a full set of saddle connection configurations
onto the subset of $\ABB^\star_{1,2}$  that is realizable in $\BH_d(\Hmu)$.
The image of~$\varphi$ is independent of~$d$ for~$d$ large enough. 
The fibers of~$\varphi$ are finite with cardinality bounded
independently of~$d$.
\par
Moreover if a graph~$\Gamma \in \BB^\star_{1,2}$ is realizable, then 
the configurations in $\varphi^{-1}(\Gamma)$ are in bijection with cyclic 
orderings of the components at level $0$.
\end{Prop}
\par
\begin{proof}  To define $\varphi$, we pinch the boundary curves 
of~$R(\cCC)$ to obtain a pointed nodal curve. The configuration $\cCC$ of 
saddle connections remains in one component of the curve that contains $z_1$ 
and $z_2$. We provide the dual graph of the curve with the level structure 
such that the component containing~$z_1$ and~$z_2$ is the unique one at 
level~$-1$ and all the other components are on level $0$.  
This way we thus obtain a graph $\varphi(\cCC) \in \ABB^\star_{1,2}$. 
We leave the straightforward verification of the other statements to 
the reader.
\end{proof}
\par
An application of the Riemann-Hurwitz formula shows that any configuration in
$\varphi^{-1}(\Gamma)$ has multiplicity $|E(\Gamma)|+2g(X_{-1})$.
In the special case $\Gamma \in \BB^\star_{1,2}$ (i.e.\ if $g(X_{-1}) = 0$
and $\Gamma$ is of compact type), the configuration consists of 
$k=|E(\Gamma)| $ pairwise homologous arcs. However, the cover on the 
right-hand side of Figure~\ref{fig:exadmcov} shows that graphs
in $\Gamma \in \ABB^\star_{1,2} \setminus \BB^\star_{1,2}$ also contribute.
It is not hard to give an example that the fiber cardinality of~$\varphi$
over a target graph with $g(X_{-1})>0$ can indeed be larger than one,
and we leave it to the reader since it is irrelevant to our applications.  
\par
\medskip
Finally we address that Theorem~\ref{thm:H2} and an a priori knowledge
that $M(\mu) = (m_1+1)(m_2+1)$ would give an alternative proof of Theorem~\ref{intro:VolRec}.
In terms of our volume normalization, \cite[Proposition~9.4]{cmz} says that  
\ba \label{eq:vol-hur}
\sum_{d=1}^D N_d^\circ(\Hmu) 
&\= \frac{v(\mu)}{2r\prod_{i=1}^n (m_i+1)} \,D^{r} + O(D^{r-1} \log D)
\ea
as $D\to \infty$, where $r = 2g + n -1$. To sum the right-hand side of~\eqref{eq:H2} we let
$S_D(\Pi_i) = \sum_{d=1}^D N_{d}^\circ(\Pi_i)$. The Euler integral
$\int_{0}^1 t^{a-1}(1-t)^{b-1}dt = \frac{(a-1)!(b-1)!}{(a+b-1)!}$ used recursively
implies the following result.  
\par
\begin{Lemma} \label{lm:eulerintegral}
Suppose that $S_D(\Pi_i) = v_i D^{r_i} + O(D^{r_i-1} \log D)$
as $D \to \infty$ and that there exists a constant $C$ depending
on $\Pi_i$ only such that $N_{d}^\circ(\Pi_i) < C d^{r_i-1}$
for $i=1,\ldots,k$. Then 
\be
\lim_{D \to \infty}  \sum_{d_1+\cdots + d_k\leq D}
\frac{(d_1+\cdots + d_k)\,N_{d_1}^\circ(\Pi_1)\cdots N_{d_k}^\circ(\Pi_k)}
{D^{r_1 + \cdots +r_k+1}}
\= \frac{\prod_{i=1}^k (r_i! v_k)\sum_{i=1}^k r_i}{(r_{1} +\cdots + r_k+1)!}\,.
\ee
\end{Lemma}
\par
\begin{proof}[Alternative proof of Theorem~\ref{thm:H2} (assuming 
$M(\mu) = (m_1+1)(m_2+1)$)]
With the ab\-bre\-viation
$r_i = 2g_i + n(\mu_i)$ we obtain from~\eqref{eq:vol-hur} that  
$$ \sum_{d_i=1}^{D} N_{d_i}^{\circ} (\Hmu_i, (p_i)) \=
\frac{ v(\mu_i, p_i-1)}{2r_ip_i \prod_{m_i\in \mu_i} (m_i+1)} D^{r_i}
+ O(D^{r_i-1} \log D)\,. $$
Since $r_1+\cdots + r_k = 2g+n-2$, Lemma~\ref{lm:eulerintegral} implies that 
\ba \label{eq:5.2RHS}
& \phantom{\=}\, \sum_{d_1+\cdots +d_n\leq D} (d_1+\cdots +d_n) \prod_{i=1}^k
p_i\,N^\circ_{d_i}(\Hmu_i,(p_i)) \\ 
& \=  D^{2g+n-1} \frac{ \prod_{i=1}^k  (2g_i+n(\mu_i)-1)! v(\mu_i, p_i-1)}
{2^k(2g+n-1)(2g+n-3)!\prod_{i=3}^n (m_i+1)}\,.
\ea
Summing over all backbone graphs and taking the limit after dividing
by $D^{2g+n-1}$ thus implies the desired formula~\eqref{eq:volintro}.
\end{proof}
\par
Conversely, the above argument shows that the mere knowledge of Theorem~\ref{thm:H2}
gives the recursion in Theorem~\ref{intro:VolRec} with $M(\mu)$ on the
left-hand side that replaces $(m_1+1)(m_2+1)$, and hence the two theorems taken
together thus determine the value 
$M(\mu) = (m_1+1)(m_2+1)$ as claimed in the beginning of the section.


\section{Area Siegel-Veech constants}\label{sec:areaSV}

The goal of this section is to show that area Siegel-Veech constants
are ratios of intersection numbers, i.e.\ to prove
Theorem~\ref{intro:IntFormula2}. For this purpose we introduce
\be \label{eq:defd}
  d_i(\mu) \= \int_{\proj\obarmoduli[g,n](\mu)} \!\!\!\!\!\!\! \beta_i\cdot \delta_0
\=\frac{1}{m_i+1} \int_{\proj\obarmoduli[g,n](\mu)} \!\!\!\!\!\!\!  \xi^{2g-2} \cdot \delta_0\cdot \prod_{j\neq i} \psi_j\,,
\ee
and then Theorem~\ref{intro:IntFormula2} can be reformulated as
\be \label{eq:IntF2reform}
c_{\rm area}(\mu)\= \frac{-1}{4\pi^2} \frac{d_i(\mu)}{a_i(\mu)}\,.
\ee
The proof proceeds similarly to the proof of Theorem~\ref{intro:IntFormula}
by showing a recursive formula for both the intersection numbers
and the area Siegel-Veech constants. The difference in the formulas
is that one vertex at level zero of the backbone graphs is
distinguished by carrying the Siegel-Veech weight.
We remark that in this section area Siegel-Veech constants for disconnected
strata are volume-weighted averages of the constants for the individual
components. 
\par
The intersection number recursion leads to the remarkable formula 
\ba\label{eq:SVareasc}
(m_1+1)(m_2+1)c_{\rm area}(\mu) \= \sum_{\mathcal{C} \,\,\text{saddle connection}
\atop \text{configuration}}
c^{\rm hom}_{1 \leftrightarrow 2}(\mathcal{C}) c_{\rm area}(\mathcal{C})\,,
\ea
where $c_{\rm area}(\mathcal{C})$ is defined to be the sum of the
area Siegel-Veech constants of the splitting pieces induced by the
saddle connection configuration. The other recursion leads to a very
efficient way to compute area Siegel-Veech constants, given in
Theorem~\ref{thm:areaSVcomp}.

\subsection{A recursion for the $d_i(\mu)$ via intersection theory}

We have seen that the values of $a_i(\mu)$ do not depend on the index $i$. Similarly for $d_i(\mu)$ it suffices to focus on the case $i = 1$.  
To state the recursion, we introduce the generating series
\bes
\Delta(t) \=\sum_{g\geq 1} (2g-1)^2 d_1(2g-2)t^{2g}
\= \frac12t^2 - \frac1{16}t^4 + \frac{91}{2304}t^6
- \frac{4173}{829440}t^8 + \cdots,
\ees
whose coefficients are determined using the following proposition.
\par
\begin{Prop}\label{prop:indintSV}
The  generating function  of the intersection numbers 
$d_1(2g-2)$ is determined by the coefficient extraction identity
\be \label{eq:indSV1}
b_{j-1} \=  \frac{2}{j!} [t^{1}]
\left(\Delta(t)\cAA(t)^{j}\right)
\ee
while the intersection numbers $d(\mu) = d_i(\mu)$ with $n(\mu) \geq 2$ are 
given recursively by
\begin{flalign} 
& \phantom{\=} (m_1+1)(m_2+1) d_1(\mu)  \label{eq:indSV2} \\
&\= \sum_{k \geq 1} 
\sum_{\bfg, \bfmu}
\frac{h_{\proj^1}((m_1,m_2),\bfp)}{(k-1)!}\,
\frac{d_1(p_1-1,\mu_1)}{a_1(p_1-1,\mu_1)}
\prod_{i=1}^k (2g_i-1+n(\mu_i)) p_i a_1(p_i-1,\mu_i)\,
\nonumber
\end{flalign}
for $n=n(\mu) \geq 2$, with the usual summation conventions
as in Theorem~\ref{intro:VolRec} and Theorem~\ref{thm:indintall}.
\end{Prop}
\par
The first identity~\eqref{eq:indSV1} was proved in \cite{SauvagetMinimal}.
The proof of the second identity~\eqref{eq:indSV2} will be completed by the
end of this subsection. This identity together with the conversions
in Section~\ref{sec:VRandSC} implies~\eqref{eq:SVareasc} immediately.
We start the proof with the following analog of Proposition~\ref{pr:intbb1}. 
\par
\begin{Prop}\label{pr:intbbSV}
If $(\Gamma,\ell,\bfp)$ is a backbone graph in  ${\rm BB}(g,n)_{1,2}$, then 
\begin{eqnarray*}
\int_{\proj\oOmM_{g,n}} \!\!\!\!\! \alpha_{\Gamma,\ell,\bfp} \cdot \xi^{2g-2}\cdot
\delta_0 \cdot \prod_{i=3}^n \psi_i \!\!\!\!\!
&=& {m(\bfp)}\cdot {h_{\proj_1}(\mu_{-1}, (p_v)_{v\in V(\Gamma),\ell(v)=0})}\\
\!\!\!\!\!&\!\!\!\!\! \cdot \!\!\!\!\!& \!\!\!\!\! \!\!\!\!\! \sum_{v\in V(\Gamma),\ell(v)=0} d_1({p_v-1,\mu_v}) \!\!\! \!\!\!\!\! \prod_{v'\in V(\Gamma)\setminus\{v\},\ell(v')=0} \!\!\!\!\!\!\!\!\!\!\!\!\! a_1({p_{v'}-1,\mu_{v'}})\,
\end{eqnarray*}
with the conventions for $p_v$ as in Proposition~\ref{pr:intbb1}.
\end{Prop}
\par
\begin{proof} 
We have the equality that 
$$
\zeta_{\Gamma}^*(\delta_0) \= \sum_{v\in V(\Gamma),\ell(v)=0}  \delta^v_0\,,
$$
where $\delta^v_0=\delta_{0} \otimes 1\in H^*(\oM_{g_v,n_v}, \mathbb Q) \otimes \left(\bigotimes_{v'\neq v} H^*(\oM_{g_{v'},n_{v'}}, \mathbb Q)\right)\simeq H^*(\oM_{\Gamma}, \mathbb Q)$. Combining with the fact that $\delta_0\lambda_g = 0$, it implies that 
$$
\zeta_{\Gamma}^*(\delta_0\lambda_{g-1})= \!\!\!\!\!  \sum_{v\in V(\Gamma),\ell(v)=0} \left(\delta^v_0 \lambda_{v,g_v-1}\!\!\!\!\!  \bigotimes_{v'\neq v,\ell(v')=0} \!\!\! \lambda_{v',g_{v'}} \right).
$$
Therefore, we obtain that  
\begin{eqnarray*}
\lambda_{g-1}\cdot\delta_0\cdot \alpha_{\Gamma,\ell, \bfp}^0= && \!\!\!\!\!\!\!  \!\!\!\!\!\sum_{v\in V(\Gamma), \ell(V)=0} \zeta_{\Gamma*}\Bigg( [\oM_{-1}] \otimes  \left(\lambda_{v,g_{v}-1} \cdot \delta_0 \cdot [\proj \oOmM_{g_{v},n_{v}}(p_{v}-1,\mu_{v})]^0\right) \\
&&\bigotimes_{v'\neq v, \ell(v')=0}  \lambda_{v',g_{v'}} [\proj \oOmM_{g_{v'},n_{v'}}(p_{v'}-1,\mu_{v'})]^0  \Bigg).
\end{eqnarray*}
Using the last formula in Lemma~\ref{lem:toplambda}, the rest of the proof then follows from the same argument as in Proposition~\ref{pr:intbb1}. 
\end{proof}
\par
We also need the following analog of Lemma~\ref{lem:indbis}.
\par
\begin{Lemma}\label{lem:indbisSV} 
The values of $d_1(\mu)$ satisfy the recursion 
\ba \label{eq:d1rec}
& \phantom{\=} (m_1+1)(m_2+1) d_1(\mu)  \\ 
& \= \sum_{ k \geq 1}
\sum_{\bfg, \bfmu}
h_{\PP^1}((m_1,m_2,\mu_0),\bfp) \cdot
\frac{d_1(p_1-1,\mu_1)}{(k-1)!}
\cdot \left( \prod_{i=2}^k p_i^2 a(p_i-1,\mu_i) \right)\,
\ea
with summation conventions as in Lemma~\ref{lem:indbis}.
\end{Lemma} 
\par
\begin{proof} We use the formula in Proposition~\ref{prop:indclasses} for $i=2$,
multiply by $\xi^{2g-2} \delta_0 \prod_{i=3}^n \psi_i$ and
apply~$p_*$. The left-hand side then evaluates (by 
Lemma~\ref{lem:toplambda} and the fact that $\delta_0\lambda_g =0$) 
to the left-hand side of~\eqref{eq:d1rec}. The right-hand side
evaluates (by Proposition~\ref{pr:intbb}) to the weighted sum over
all $(\Gamma,\ell,\bfp) \in {\rm BB}(g,n)_{1,2}$ of the expression in
Proposition~\ref{pr:intbbSV}. To prove the lemma we interpret as usual
a backbone graph as a decomposition of~$g$ and the marked points.
The factor $(k-1)!$ (instead of $k!$ in Lemma~\ref{lem:indbis}) comes
from the fact that one of the top level vertices of the backbone graph is distinguished.
\end{proof}
\par
With the same notation as in Section~\ref{ssec:treestrata} we now define
the $d$-contribution of a rooted tree to be 
\begin{eqnarray*}
d(\Gamma,\ell,\bfp)&=& m_0(\bfp)^2 h(\Gamma_0,\ell_0,\bfp_0)   \!\!
\sum_{v\in V(\Gamma),\atop \ell(v)=0} d_1({p_v-1,\mu_v}) \!\!\! \!
\prod_{v'\in V(\Gamma)\setminus\{v\},
\atop \ell(v')=0} \!\!\!\!\!\!\! a({p_{v'}-1,\mu_{v'}})\,.\\
\end{eqnarray*}
Then we can rewrite Lemma~\ref{lem:indbisSV} as 
\bes
(m_1+1)(m_2+1) d_1(\mu) \=\sum_{(\Gamma,\ell,\bfp)\in {\rm RT}(g,\mu)_{1,2}}
\frac{d(\Gamma,\ell,\bfp)}{|{\rm Aut}(\Gamma,\ell,\bfp)| }\,,
\ees
which is the analog of Lemma~\ref{lem:inttree}. 
\par
\begin{proof}[End of the proof
of Proposition~\ref{prop:indintSV}]
Now the proof of the proposition can be completed similarly to the end of the proof of 
 Theorem~\ref{thm:indintall} at the end of Section~\ref{sec:RelMZInt}. 
 \end{proof}
\par
As a consequence, $d_i(\mu)$ does not depend on~$i$ and we simply write
$d(\mu)$ from now on. 

\subsection{A recursion for area Siegel-Veech numerators via weighted counting of covers}

We recall the main steps of~\cite{cmz} that reduce the computation
of area Siegel-Veech constants to a statement about cumulants.
An application of the Siegel-Veech formula gives (\cite[Proposition~17.1]{cmz})
the quantity we aim for as 
\be \label{eq:HurStr1}
c_{{\rm area}}\,(\omoduli(m_1,\ldots,m_n))
\= \lim_{D \to \infty} \frac{3}{\pi^2}
\frac{\sum_{d=1}^D c_{-1}^\circ(d,\Pi)}{\sum_{d=1}^D N^\circ_d(\Pi)}\,,
\ee
where $N^\circ_d(\Pi)$ is the number of connected torus covers of degree~$d$ with 
ramification profile $\Pi = (m_1+1,\ldots,m_n+1)$ and where $c_{-1}^\circ(d,\Pi)$
is the
sum over those covers with $-1$st Siegel-Veech weight (see~\cite[Section~3]{cmz}). 
The relation of the sum of Fourier coefficients to the growth
polynomial (\cite[Proposition~9.4]{cmz}) and a rewriting of the
Siegel-Veech weighted counting (\cite[Corollary~13.2]{cmz})
translate this into
\ba \label{eq:HurStr2}
c_{{\rm area}}\,(\omoduli(m_1,\ldots,m_n))
&\=
\frac{3}{\pi^2} \frac{\bL{T_{-1}|f_{m_1+1}|\cdots|f_{m_n+1}}}
{\bL{f_{m_1+1}|\cdots|f_{m_n+1}}} \\
 &\=
     \frac{3}{\pi^2} \frac{\bL{T_{-1}|h_{m_1+1}|\cdots|h_{m_n+1}}}
          {\bL{h_{m_1+1}|\cdots|h_{m_n+1}}}\,,
\ea
where $T_{-1}$ is a hook-length moment function on partitions 
(but not an element of the ring~$\Lambda^*$, see \cite[Section~13]{cmz}) and 
where we use Proposition~\ref{prop:EOconv} and~\eqref{eq:ggtop}
for the second equality. We have seen in Section~\ref{sec:D2rec} how to compute
the denominator and related it in Section~\ref{sec:D2VR} to~$a(\mu)$. Now we
take care of the numerator.
\par
Recall that we defined $\Phi^H(\bfu)_q$ in~\eqref{eq:defPhi} in Section~\ref{sec:D2rec}. Define
now 
\bes
C'_{-1}(\bfu)_q \= \sum _{\bfn >0} \bq{T_{-1} \cdot \prod_{\ell \geq 1}
  h_\ell^{n_\ell}} \,\frac{\bfu^\bfn}{\bfn!}\,.
\ees
By definition of cumulants (or by \cite[Proposition~6.2]{cmz}) we are
interested in the leading term of the quotient 
\bes
C^\circ_{-1}(\bfu)_q \,:=\, \frac{C'_{-1}(\bfu)_q}{\Phi^H(\bfu)_q}
\= \sum_{\bfn \geq 0}\la\underbrace{T_{-1}|h_1|\cdots|h_1}_{n_1}|\underbrace{h_2|\cdots|h_2}_{n_2}|\cdots\ra_q\,\frac{\bfu^\bfn}{\bfn!}\,.
\ees
To evaluate the numerator of this fraction, recall from
\cite[Section 16]{cmz} the definition of the modified $q$-bracket
\be \label{eq:modsbqs}
\sbqs f \= \sbq{T_{-1}\,f} \m \sbq{T_{-1}}\,\sbq f 
\m \frac1{24}\,\sbq{\partial_2(f)}\,,
\ee
where $\p_2$ is the differential operator
$$ \p_2\colon \frac{\p}{\p p_1} \+ \sum_{\ell \geq 2}
\ell(\ell-1)p_{\ell-2} \frac{\p}{\p p_\ell}\,.
$$
This bracket is useful, since its effect can be computed by differential
operators acting (contrary to~$T_{-1}$) within the Bloch-Okounkov algebra.
In fact, \cite[Theorem~16.1]{cmz} states that
\bes  
\sbqs f \= \sum_{j\ge1} G_2^{(j-1)}\,\la\rho_{0,j}(f)\ra_q\, \+
\sum_{i\ge2,\,j\ge0} G_i^{(j)}\,\la\rho_{i,j}(f)\ra_q\,,
\ees
where~$\rho_{i,j}$ are differential operators of degree~$j$ that shift
the weight by~$-i-2j$, whose definition we recall in~\eqref{eq:rho0j}
below. Motivated by the action of these operators we define
\be \label{eq:defcCC}
\cCC'_{-1}(\bfu) \= 
\sum _{\bfn >0}  \sum_{j \geq 1} j!\,\rho_{0,j}\Bigl(\prod_{\ell \geq 1}
  h_\ell^{n_\ell}\Bigr) \,\frac{\bfu^\bfn}{\bfn!}\,,
\ee
and we let $\Phi^H(\bfu) = \exp(\sum_{\ell \geq 1} h_\ell u_\ell)$
such that $\Phi^H(\bfu)_q = \sbq{\Phi^H(\bfu)}$.
\par
\begin{Lemma} \label{le:leadagree}
The leading terms of  
\bes
C^\circ_{-1}(\bfu)_q \quad \text{and}
\quad
\cCC^\circ_{-1}(\bfu)_q := \frac{-1}{24}
\frac{\sbq{\cCC'_{-1}(\bfu)}}{\Phi^H(\bfu)_q}\,
\ees
agree.
\end{Lemma}
\par
\begin{proof}
First note that $\evX(G_2^{(j-1)})(X) = \tfrac{-1}{24}(j!X + (j-1)!)$
by the defining formulas in \cite[Section~9]{cmz}. This is the reason for
the factor $j!$ in~\eqref{eq:defcCC}. From the non-vanishing of the
area Siegel-Veech constant, we know that the leading degree contribution
is as in~\eqref{eq:deflead}. Lower weight terms before passing to the
cumulant quotient will contribute to lower order in the growth polynomial.
Since $\p_2$ is of degree~$-2$, its contribution in~\eqref{eq:modsbqs} is negligible and we can work with the star-brackets. For the same reason,
the terms with~$i>0$ in the definition $\sbqs f$ are dominated by the
corresponding term with~$i=0$ and can be neglected.
\end{proof}
\par
Our goal is to compute the $h$-evalutaion of $\cCC_{-1}^\circ(\bfu)$  and
its leading term using Proposition~\ref{prop:liftviadelta}. 
\par
\begin{Lemma} \label{le:p2exp}
The commutation relation
  \be \label{eq:commdiffop}
\partial_2 \circ e^D(f) \= e^D  \Bigl(\sum_{j \geq 1} j!\, \rho_{0,j}(f)\Bigr)
  \ee
holds for every $f \in \Lambda^{*}$.
\end{Lemma}
\par
\begin{proof} We will check the relation on the $n$-point function for
every~$n$. Since we will recall formulas from \cite{cmz} we use the
rescalings $Q_k = p_{k-1}/(k-1)!$ of the generators of $\Lambda^*$,
where $Q_0 = 1$ and $Q_1=0$. The following identities even hold on
the polynomial ring $R = \QQ[Q_0,Q_1,Q_2,\ldots]$ mapping to $\Lambda^*$.
We set $W(z) = \sum_{k \geq 0} Q_k z^{k-1}$. We recall from
\cite[Theorem~14.2]{cmz} the action of the operators~$\rho_{0,j}$, namely
\be \label{eq:rho0j}
\rho_{0,j}\bigl(W(z_1)\cdots W(z_n)\bigr) \= 
\sum_{J\subset N\atop|J|=j} W(z_J)\,z_J \Bigl(\prod_{j\in J} z_j\Bigr)
\,\prod_{\nu\in N\ssm J}W(z_\nu)
\ee
where $z_J = \sum_{j \in J} z_j$ and $N = \{1,\ldots,n\}$. On the other
hand, in terms of the $Q_i$, the operator~$D$ defined in Section~\ref{sec:D2rec}
is just $D= \tfrac12 (\Delta - \p^2)$, where~$\p$ is the differential
operator sending~$Q_i$ to~$Q_{i-1}$. From \cite[Proposition~10.5]{cmz}
we know that 
\bes
e^{D}\, \bigl(W\bigl(z_1)\cdots W\bigl(z_n\bigr)\bigr)
\= e^{-z_N^2/2} \, \sum_{\alpha \in \PPP(n)}
\Bigl({\prod_{A\in\alpha} \,
z_A^{|A|-1} W(z_A)}\Bigr)\,.
\ees
Using these identities we can evaluate both sides of~\eqref{eq:commdiffop} to
be of the form
\bes
e^{-z_N^2/2} \, \sum_{\alpha \in \PPP(n)}
\Bigl({\prod_{A\in\alpha} \,
W(z_A) R(\{z_a\}_{a \in A})}\Bigr)
\ees
where $R(\{z_a\})$ are polynomials that are visibly different
on the two sides, but in fact agree by using the identity 
\bes
z_A^{n+1} \=
\sum_{\emptyset \neq J \subset A} |J|!\, z_J\, \Bigl(\prod_{\nu \in J}
z_\nu \Bigr)\, z_A^{n-|J|}\,.
\ees
To verify this expression, let
$e_i = (-1)^{i} [x^{n-i}] \prod_{a \in A} (x-z_a)$ be the elementary
symmetric functions in the $z_a$. Then the contribution with $|J| = j$
to the right-hand side is $e_1^{n-j}(e_1e_j - (j+1)e_{j+1})$. This means
that the right-hand side is a telescoping sum where only the first term
remains after summation.
\end{proof}
\par
The preceding Lemma~\ref{le:p2exp}, Proposition~\ref{prop:liftviadelta} for
the computation of the $h$-brackets, Lemma~\ref{le:leadagree}
and~\eqref{eq:HurStr2} now imply immediately our goal:
\par
\begin{Thm} \label{thm:areaSVcomp}
The area Siegel-Veech constants can be computed as
\bes
c_{\rm area}(m_1,\ldots,m_n) \= \frac{-1}{8\pi^2}
\frac {[z_1^{m_1+1}\cdots z_n^{m_n+1}] \, \partial_2(\cHH_n)}
{[z_1^{m_1+1}\cdots z_n^{m_n+1}] \,\, \cHH_n} \Bigr|_{h_\ell \mapsto \alpha_\ell}\,,
\ees
where $\cHH_n(z_1,\ldots,z_n)$ is recursively defined as 
in Section~\ref{sec:D2rec}.
\end{Thm}
\par
\subsection{Proof of Theorem~\ref{intro:IntFormula2}}
\label{sec:proofareaSV}

We start with an explicit formula for the $\p_2$-derivative used
in Theorem~\ref{thm:areaSVcomp} in the case of the minimal strata.
\par
\begin{Prop} \label{prop:SVmin}
The area Siegel-Veech constants for the minimal strata are 
\be \label{eq:blupp}
c_{\rm area}(\omoduli[g](2g-2)) \= \frac{-1}{8\pi^2}\frac{[u^{2g-1}]\,
\cDD(u)}
{[u^{2g-1}]\, \cAA(u)}\,,
\ee
where
\bes
\cDD(u) \= (\cAA'(u) + u\cAA''(u))/u\cAA'(u)^2 \= 
t - \frac1{18}t^3 + \frac{91}{2304}t^5 - \cdots\,.
\ees
\end{Prop}
\par
\begin{proof} Differentiating~\eqref{eq:auxsum} gives
  \bes
\sum_{n \geq 2} n^2 p_{n-1} \cHH^{-n}(z) \=  \frac{\cHH(z)}{\cHH'(z)}
\Bigl(\frac1z - \frac{\cHH(z)}{z^2\cHH'(z)}
- \frac{\cHH(z)\cHH''(z)}{z\cHH'(z)^2} \Bigr)\,.
  \ees
Combining these two equalities gives $\p_2(\cHH_1(u)) =
(\cHH'(u) + u\cHH''(u))/u\cHH'(u)^2$ and the claim by
substituting $h_\ell \mapsto \alpha_\ell$.
\end{proof}
\par
\begin{proof}[Proof of Theorem~\ref{intro:IntFormula2}]
We start with the case of a single zero. Comparing~\eqref{eq:blupp}
and~\eqref{eq:indSV1} we need to show that $\cDD(u) = 2\Delta(u)/u$, i.e.\
in view of~\eqref{eq:ind1} we need to show that 
\bes
[u^0] (\cDD(u) \cAA(u)^{2g-1}) \= {(2g-1)!} [u^{2g-2}]B(u)
\= (2g-1) [u^0] \cAA(u)^{2g-2}\,.
\ees
This equality can be implied by showing that
\bes
[u^{-1}] \cAA(u)^{2g-1}\frac{\cAA'(u) + u\cAA''(u)}{ (2g-1) u^2\cAA'(u)^2} \=
[u^{-1}] \frac{\cAA(u)^{2g-2}}u \,,
\ees
which in turn follows since the derivative
\bes  \Bigl(\frac{\cAA(u)^{2g-1} + u \cAA'(u)}{(2g-1) u\cAA'(u)} \Bigr)' \=
\cAA^{2g-1}\frac{\cAA'(u) + u\cAA''(u)}{ (2g-1) u^2\cAA'(u)^2}
- \frac{\cAA(u)^{2g-2}}u
\ees
has no $(-1)$-term. Finally, to deal with the case of multiple zeros, we recall
from~\eqref{for:liftint} that $a_i(\mu) = [z_1^{m_1+1}\cdots z_n^{m_n+1}]
/(2g-2+n) \prod_{j=1}^n (m_j+1)\, \cHH_n$ 
and hence we need to show that
\bes
d(\mu) \= \frac{[z_1^{m_1+1}\ldots z_n^{m_n+1}] \partial_2(\cHH_n)}
{(2g-2+n)\prod_{i=1}^n (m_i+1)}\Bigr|_{h_\ell \mapsto \alpha_\ell}
\ees
after knowing that this is true for the case of $n(\mu)=1$. This follows
immediately from differentiating~\eqref{VRBO}, since after 
substituting $h_\ell \mapsto \alpha_\ell$ this is exactly the sum
of the recursions~\eqref{eq:indSV2} (known to hold for the $d(\mu)$),
averaging over all pairs $(m_r,m_s)$ of the entries of~$\mu$, as
in~\eqref{eq:average}.
\end{proof}
\par
\medskip
Given Theorem~\ref{intro:IntFormula2} for the area Siegel-Veech computation of the strata on one hand, and the 
refined Theorem~\ref{thm:refinedINT} for the volume computation of the spin components 
on the other hand, it is natural to suspect that area Siegel-Veech
constants for the spin components can also be computed as ratios of intersection numbers  
\be \label{for:intSVconnected} 
c_{\rm area}(\mu)^\bullet
\= \frac{-1}{4\pi^2}\frac{\int_{\proj\obarmoduli[g,n](\mu)^{\bullet}} \beta_i \cdot \delta_0}{\int_{\proj\obarmoduli[g,n](\mu)^{\bullet}} \beta_i \cdot \xi}
\ee
for all $1\leq i\leq n$, where $\bullet\in\{ {\rm even},{\rm odd}\}$.
Using Assumption~\ref{asu} to deal with the case of the minimal strata, the
validity of~\eqref{for:intSVconnected} is equivalent to the validity of 
\bes
c_{\rm area}(\mu)^{\odd} \= \frac{-1}{16\pi^2}
[z_1^{m_1+1}\cdots z_n^{m_n+1}] \, \left( \frac{(2\pi i)^{2g}(\partial_2(\cHH_n) - \partial^{\Delta}_{2}(\bfcH_n))}{(2g-2+n)! v(\mu)^{\odd}}\right)\Bigg|_{\begin{smallmatrix} h_{\ell} \mapsto \alpha_\ell \\ \bfh_\ell \mapsto \ual_\ell \end{smallmatrix}}
\ees
where $\cHH_n$ and $\bfcH_n$  are recursively defined as in
Sections~\ref{sec:D2rec} and~\ref{sec:spin}, and where
$$
\partial_2^{\Delta} \=\frac{\partial}{\partial \bfp_1}+\sum_{\ell\geq 1}^{\infty}
2\ell (2\ell+1) \bfp_{2\ell-1}\frac{\partial}{\partial \bfp_{2\ell+1}}
$$
is the analog of the differential operator $\partial_2$ on
the algebra of super-symmetric functions. There is a clear strategy towards
this goal:
\begin{itemize}
\item Show that there are operators like the $T_p$ for $p \geq -1$ odd as
  in \cite[Section 12]{cmz}, whose strict brackets compute Siegel-Veech
  weighted and spin-weighted Hurwitz numbers.
\item Show that the action of $T_p$ inside strict brackets can be
encoded by differential operators like the $\rho_{ij}$ in \cite[Section 14]{cmz}.
\item Show that these operators satisfy a commutation relation as
in Lemma~\ref{le:p2exp}, with $\partial$ replaced by $\partial_2^{\Delta}$.
\item Conclude by comparing the recursions as in the preceding proofs.
\end{itemize}
Given the length of this paper, we do not attempt to provide details here.



\section{Large genus asymptotics}
\label{sec:asy}

In this section we study the large genus asymptotics of
Masur-Veech volumes and area Siegel-Veech constants and prove the
conjectures of Eskin and Zorich in \cite{EZVol} by using our previous results.

\subsection{Volume asymptotics}
\label{sec:asymin}
We recall from \cite[Theorems~12.1 and~19.1]{cmz} and
\cite[Theorem~1.9]{SauvagetMinimal} that the asymptotic expansions of $v(2g-2)$
and $v(1,\ldots,1)$ can be computed using the mechanisms of very rapidly
divergent series (\cite[Appendix]{cmz} as 
\begin{eqnarray*} v(1^{2g-2})  \;&\sim&\; {4} \,
\Bigl(1\,-\, \frac{\pi^2}{24g} \,-\, 
\frac{60\pi^2 - \pi^4}{1152g^2}\,-\, \cdots\Bigr)\,,  \\
v(2g-2)  \;&\sim&\; {4} \,
\Bigl(1\,-\, \frac{\pi^2}{12g} \,-\, 
\frac{24\pi^2 - \pi^4}{288g^2} \,-\, \cdots\Bigr)\,.
\end{eqnarray*} 
\par
Let $\mu=(m_1,\ldots,m_n)$ be a partition of $2g-2$ into $n$ positive integers
with $n\geq 2$. We write $\mu'=(m_1+m_2, m_3,\ldots,m_n)$ and
$\mu''=(m_1+m_2-2, m_3,\ldots,m_n)$. We use Theorem~\ref{intro:VolRec} and the
two obvious backbone graphs, the one with a single top level component of
genus~$g$ (i.e.\ $k=1$) and the one with two top level components (i.e.\ $k=2$)
of genus~$1$ and~$g-1$ respectively, to deduce the inequality 
\be \label{eq:vmu2terms}
v(\mu)\geq v(\mu') + \frac{\pi^2 (2g-5+n)!}{6(2g-3+n)!} v(\mu'')\,,
\ee
where we use $h_{\proj^1}((m_1,m_2), (m_1+m_2+2))=h_{\proj^1}((m_1,m_2), (m_1+m_2,1))=1$
for $m_1,m_2 >0$ and $v(0)=\pi^2/6$.
In particular this inequality implies that $v(\mu)\geq v(\mu')$ and thus 
$$
v(2g-2) \,\leq\, v(\mu) \,\leq\, v(1,\ldots,1)
$$
for all $\mu$.  Consequently, there exists a constant $C>0$ such that for all $\mu$ we have the inequality $|v(\mu)-4|<C/g$. Now we introduce the notation
$$\widehat{v}(\mu):= v(\mu)-4+\frac{2\pi^2}{3(2g-3+n)}\,.$$
Then the inequality~\eqref{eq:vmu2terms} implies that 
\begin{eqnarray*}
\widehat{v}(\mu)- \widehat{v}(\mu')&\geq& - \frac{2\pi^2}{3(2g-3+n)(2g-4+n)}+ \frac{\pi^2 (2g-5+n)!}{6(2g-3+n)!} v(\mu'')\\ 
&\geq &  - \frac{C \pi^2 }{6g(2g-3+n)(2g-4+n)}\,.
\end{eqnarray*}
In particular for all $\mu$ we have 
$$
\widehat{v}(2g-2)- \frac{C \pi^2 (n-1)}{6g(2g-1)^2}  \leq \widehat{v}(\mu)\leq \widehat{v}(1,\ldots,1) +  \frac{C \pi^2 (2g-2-n) }{6g  (2g-1)^2}\,.
$$
Since $n$ is bounded by $2g-2$, there exists a constant $C'$ such that $|\widehat{v}(\mu)|\leq C'/g^2$ for all $\mu$. Thus the first part of Theorem~\ref{cor:EZconj} holds. 

\subsection{Asymptotics of Siegel-Veech constants}
We apply the same strategy to control the asymptotic behavior of area
Siegel-Veech constants. We denote by 
$$
 \widetilde{d}(\mu)\,:=\, c_{\rm area}(\mu)v(\mu)
 \ = \frac{1}{4\pi^2} \frac{2(2\pi){^2g}}{(2g-3+n)!}
 \left( \prod_{i=1}^n (m_i+1) \right) \cdot  |d(\mu)| 
$$
where $d(\mu)$ is defined in~\eqref{eq:defd} and where the second equality
stems from~\eqref{eq:IntF2reform}. For $\mu = (m_1,\ldots, m_n)$ with $n \geq 2$, we write 
$\mu'=(m_1+m_2, m_3,\ldots,m_n)$ and $\mu''=(m_1+m_2-2, m_3,\ldots,m_n)$ as before. Then from
Proposition~\ref{prop:indintSV} we have $\widetilde{d}(0)=1$ and we obtain the inequality
$$
\widetilde{d}(\mu) \,\geq\, \widetilde{d}(\mu') \+  \frac{\pi^2 (2g-5+n)!}{6(2g-3+n)!}
\widetilde{d}(\mu'') \+ \frac{(2g-5+n)!}{2(2g-3+n)!} v(\mu'')\,.
$$
In particular this inequality implies that $\widetilde{d}(\mu) \geq
\widetilde{d}(\mu')$. Moreover, we know the asymptotic expansions
\bes
\widetilde{d}(2g-2) \,\sim\, 2 - \frac{3+\pi^2}{6g} +\cdots \quad
\text{and} \quad 
\widetilde{d}(1,\ldots,1)\,\sim\, 2 - \frac{3+\pi^2}{12g}+\cdots \, 
\ees
from\cite[Theorem~19.4]{cmz} and~\cite[Theorem~1.9]{SauvagetMinimal}.
Consequently, there exists a constant $C$ such that $|\widetilde{d}(\mu)-2|< C/g$ for all $\mu$. Then by the same argument as above we can show that there exists a constant $C'$ such that
$$
\Big|\widetilde{d}(\mu)-2+ \frac{3+\pi^2}{3(2g-3+n)} \Big| \,< \, C'/g^2\,.
$$
Therefore, using the fact that $c_{\rm area}(\mu)=\widetilde{d}(\mu)/v(\mu)$ we thus deduce the second part of Theorem~\ref{cor:EZconj}.

\subsection{Spin asymptotics}

The goal here is to prove that the volumes of the odd and even spin
components are asymptotically equal. This is the content of
Theorem~\ref{thm:EZconj3} in the introduction that refines the conjecture
of Eskin and Zorich (\cite[Conjecture~2]{EZVol}).
Recall that we defined in Section~\ref{sec:spinD2}
\be \label{eq:PZredef}
\bfP_Z(u) \= \exp\Bigl(\sum_{ j \geq 1} \Bigl(\frac12\Bigr)^{j/2}
\frac{\zeta(-j)}{2} u^{j+1}\Bigr)\,.
\ee
\par
\begin{Prop} \label{prop:spinminasy}
The difference $v(2g-2)^{\Delta} = v(2g-2)^{\rm even}
  -v(2g-2)^{\rm odd}$ can be computed as the coefficient extraction  
\be \label{eq:deltamin}
v(2g-2)^{\Delta}  \=  \frac{2 (2\pi i)^{2g}}{(2g-1)!}
\bigl[u^{2g-1}\bigr] \frac{1}{(u/\bfP_Z)^{-1}}\,.
\ee
Moreover, it has the asymptotics
\be \label{eq:deltaAsyMin}
v(2g-2)^{\Delta} \;\sim\; \Bigl(\frac{-1}{2}\Bigr)^{g-2} \,
\Bigl(1\,+\, \frac{2\pi^2}{3g} \,+\, 
\frac{12\pi^2 + \pi^4}{18g^2} \,+\, \cdots\Bigr)\,
\ee
as $g \to \infty$.
\end{Prop}
\par
For the reader's convenience we give a table for low genus values: 
\begin{figure}[H]
$$ \begin{array}{|c|c|c|c|c|c|}
\hline
g & 1 & 2 & 3 & 4 & 5 \\
\hline &&&&& \\ [-\halfbls]
v(2g-2)^{\Delta}  & \,\,\frac{-1}{3} \,\,&  \frac{1}{40}
& \frac{-143}{108864}  & \frac{15697}{279936000} & \frac{-2561}{1103872000}  \\
      [-\halfbls] &&&&&\\
\hline
\end{array}
$$
\end{figure}
\par

\par
\begin{proof}[Proof of Proposition~\ref{prop:spinminasy}]
The first statement is a reformulation of a special case of
Corollary~\ref{cor:eodiff}.
\par
The power series $\bfP_Z(u)$ is a very rapidly divergent series, just as $P_B(u)$ is,
since the coefficients $\ell!b_\ell$ and 
$\zeta(-\ell)/2 =  \ell!b_\ell \cdot  \sqrt{2}(2^{\ell}-2^{-\ell})$
differ by a factor that grows only geometrically. The asymptotic statement
thus follows from the method of very rapidly divergent series. 
\end{proof}
\par
\begin{proof}[Proof of Theorem~\ref{thm:EZconj3}]
Proposition~\ref{prop:spinminasy} together with Theorem~\ref{cor:EZconj}
implies that there exists a constant $C'$ such that for all $g\geq 1$
$$
|v(2g-2)^{\rm odd}-v(2g-2)^{\rm even}|\, \leq \, C'/g\,.
$$
Repeated application of Theorem~\ref{cor:EZconj} implies that 
$v(\mu)^\bullet\geq v(2g-2)^{\bullet}$ for all $\mu$ with $|\mu| = 2g-2$.
Theorem~\ref{cor:EZconj} moreover implies that there exists a
constant $C''$ such that for all $\mu$ with $|\mu| = 2g-2$ the  inequality 
$$
|v(\mu)-4| \, \leq\,  C''/g
$$
holds. Thus for all $\mu$ with $|\mu| = 2g-2$ we have 
\begin{eqnarray*}
v(2g-2)^{\rm odd}\leq v(\mu)^{\rm odd} &=& v(\mu)-v(\mu)^{\rm even}
\,\,\, \leq \,\,\, v(\mu)- v(2g-2)^{\rm even}\\ 
& \leq & 2 + (C'+3C'')/g\,.
\end{eqnarray*} 
It follows that $|v(\mu)^{\rm odd}-2|\leq (C'+3C'')/g$ and the same holds for
$v(\mu)^{\rm even}$. This implies the desired claim.
\end{proof}
\par


\bibliography{VolRec}

\newcommand{\etalchar}[1]{$^{#1}$}
\def\cprime{$'$}
\begin{thebibliography}{BCG{\etalchar{+}}18}

\bibitem[ACG11]{acgh2}
E.~Arbarello, M.~Cornalba, and P.~Griffiths.
\newblock {\em Geometry of algebraic curves. {V}olume {II}}, volume 268 of {\em
  Grundlehren der Mathematischen Wissenschaften}.
\newblock Springer, Heidelberg, 2011.

\bibitem[AEZ16]{aez}
J.~Athreya, A.~Eskin, and A.~Zorich.
\newblock Right-angled billiards and volumes of moduli spaces of quadratic
  differentials on {$\Bbb C\rm P^1$}.
\newblock {\em Ann. Sci. \'Ec. Norm. Sup\'er. (4)}, 49(6):1311--1386, 2016.
\newblock With an appendix by Jon Chaika.

\bibitem[Agg18a]{Agg2}
A.~Aggarwal.
\newblock {Large Genus Asymptotics for Siegel-Veech Constants}, 2018,
  arXiv:1810.05227.
\newblock Pre\-print.

\bibitem[Agg18b]{aggarwal}
A.~Aggarwal.
\newblock {Large Genus Asymptotics for Volumes of Strata of Abelian
  Differentials}, 2018, arXiv:1804.05431.
\newblock Pre\-print.

\bibitem[BCG{\etalchar{+}}18]{strata}
M.~Bainbridge, D.~Chen, Q.~Gendron, S.~Grushevsky, and M.~M{\"o}ller.
\newblock Compactification of strata of {A}belian differentials.
\newblock {\em Duke Math. J.}, 167(12):2347--2416, 2018.

\bibitem[BG15]{baugou}
M.~Bauer and E.~Goujard.
\newblock Geometry of periodic regions on flat surfaces and associated
  {S}iegel-{V}eech constants.
\newblock {\em Geom. Dedicata}, 174:203--233, 2015.

\bibitem[BO00]{blochokounkov}
S.~Bloch and A.~Okounkov.
\newblock The character of the infinite wedge representation.
\newblock {\em Adv. Math.}, 149(1):1--60, 2000.

\bibitem[CC16]{chenPB}
D.~Chen and Q.~Chen.
\newblock Principal boundary of moduli spaces of abelian and quadratic
  differentials, 2016, arXiv:1611.01591.
\newblock Preprint.

\bibitem[CMZ18]{cmz}
D.~Chen, M.~M\"oller, and D.~Zagier.
\newblock Quasimodularity and large genus limits of {S}iegel-{V}eech constants.
\newblock {\em J. Amer. Math. Soc.}, 31(4):1059--1163, 2018.

\bibitem[DGZZ18]{DGZZ}
V.~Delecroix, E.~Goujard, P.~Zograf, and A.~Zorich.
\newblock {Masur-Veech volumes, Siegel-Veech constants and intersection numbers
  on moduli spaces}, 2018, in preparation.

\bibitem[EKZ14]{ekz}
A.~Eskin, M.~Kontsevich, and A.~Zorich.
\newblock Sum of {L}yapunov exponents of the {H}odge bundle with respect to the
  {T}eichm\"uller geodesic flow.
\newblock {\em Publ. Math. Inst. Hautes \'Etudes Sci.}, 120:207--333, 2014.

\bibitem[EM01]{eskinmasur}
A.~Eskin and H.~Masur.
\newblock Asymptotic formulas on flat surfaces.
\newblock {\em Ergodic Theory Dynam. Systems}, 21(2):443--478, 2001.

\bibitem[EM18]{esmi}
A.~Eskin and M.~Mirzakhani.
\newblock Invariant and stationary measures for the {${\rm SL}(2,\Bbb R)$}
  action on moduli space.
\newblock {\em Publ. Math. Inst. Hautes \'Etudes Sci.}, 127:95--324, 2018.

\bibitem[EMM15]{esmimo}
A.~Eskin, M.~Mirzakhani, and A.~Mohammadi.
\newblock {Isolation, equidistribution, and orbit closures for the
  SL($2,\mathbb{R}$) action on moduli space. }.
\newblock {\em Ann. of Math. (2)}, 182(2):673--721, 2015.

\bibitem[EMS03]{ems}
A.~Eskin, H.~Masur, and M.~Schmoll.
\newblock Billiards in rectangles with barriers.
\newblock {\em Duke Math. J.}, 118(3):427--463, 2003.

\bibitem[EMZ03]{emz}
A.~Eskin, H.~Masur, and A.~Zorich.
\newblock Moduli spaces of {A}belian differentials: the principal boundary,
  counting problems, and the {S}iegel-{V}eech constants.
\newblock {\em Publ. Math. Inst. Hautes \'Etudes Sci.}, 97:61--179, 2003.

\bibitem[EO01]{eo}
A.~Eskin and A.~Okounkov.
\newblock Asymptotics of numbers of branched coverings of a torus and volumes
  of moduli spaces of holomorphic differentials.
\newblock {\em Invent. Math.}, 145(1):59--103, 2001.

\bibitem[EOP08]{eop}
A.~Eskin, A.~Okounkov, and R.~Pandharipande.
\newblock The theta characteristic of a branched covering.
\newblock {\em Adv. Math.}, 217(3):873--888, 2008.

\bibitem[EZ15]{EZVol}
A.~Eskin and A.~Zorich.
\newblock Volumes of strata of {A}belian differentials and {S}iegel-{V}eech
  constants in large genera.
\newblock {\em Arnold Math. J.}, 1(4):481--488, 2015.

\bibitem[FP18]{fapa}
G.~Farkas and R.~Pandharipande.
\newblock The moduli space of twisted canonical divisors.
\newblock {\em J. Inst. Math. Jussieu}, 17(3):615--672, 2018.

\bibitem[Ges16]{Gessel}
I.~Gessel.
\newblock Lagrange inversion.
\newblock {\em J. Combin. Theory Ser. A}, 144:212--249, 2016.

\bibitem[GM16]{GoujM}
E.~Goujard and M.~M{\"o}ller.
\newblock {Counting Feynman-like graphs: Quasimodularity and Siegel-Veech
  weight }, 2016, arXiv:math/1609.01658.

\bibitem[HM82]{harrismumford}
J.~Harris and D.~Mumford.
\newblock On the {K}odaira dimension of the moduli space of curves.
\newblock {\em Invent. Math.}, 67(1):23--88, 1982.
\newblock With an appendix by William Fulton.

\bibitem[HM98]{harrismorrison}
J.~Harris and I.~Morrison.
\newblock {\em Moduli of {C}urves}, volume 187 of {\em Graduate Texts in
  Mathematics}.
\newblock Springer-Verlag, New York, 1998.

\bibitem[IO02]{IvaOlsh}
V.~Ivanov and G.~Olshanski.
\newblock Kerov's central limit theorem for the {P}lancherel measure on {Y}oung
  diagrams.
\newblock In {\em Symmetric functions 2001: surveys of developments and
  perspectives}, volume~74 of {\em NATO Sci. Ser. II Math. Phys. Chem.}, pages
  93--151. Kluwer Acad. Publ., Dordrecht, 2002.

\bibitem[Iva01]{IvaGauss}
V.~N. Ivanov.
\newblock The {G}aussian limit for projective characters of large symmetric
  groups.
\newblock {\em Zap. Nauchn. Sem. S.-Peterburg. Otdel. Mat. Inst. Steklov.
  (POMI)}, 283(Teor. Predst. Din. Sist. Komb. i Algoritm. Metody. 6):73--97,
  259, 2001.

\bibitem[Kon92]{KontWitten}
M.~Kontsevich.
\newblock Intersection theory on the moduli space of curves and the matrix
  {A}iry function.
\newblock {\em Comm. Math. Phys.}, 147(1):1--23, 1992.

\bibitem[KZ03]{kz03}
M.~Kontsevich and A.~Zorich.
\newblock Connected components of the moduli spaces of {A}belian differentials
  with prescribed singularities.
\newblock {\em Invent. Math.}, 153(3):631--678, 2003.

\bibitem[Leh85]{lehmer}
D.~H. Lehmer.
\newblock Interesting series involving the central binomial coefficient.
\newblock {\em Amer. Math. Monthly}, 92(7):449--457, 1985.

\bibitem[Mac95]{mac}
I.~G. Macdonald.
\newblock {\em Symmetric functions and {H}all polynomials}.
\newblock Oxford Mathematical Monographs. The Clarendon Press Oxford University
  Press, New York, second edition, 1995.
\newblock With contributions by A. Zelevinsky, Oxford Science Publications.

\bibitem[Mas82]{masur82}
H.~Masur.
\newblock Interval exchange transformations and measured foliations.
\newblock {\em Ann. of Math. (2)}, 115(1):169--200, 1982.

\bibitem[Mas90]{masur90}
H.~Masur.
\newblock The growth rate of trajectories of a quadratic differential.
\newblock {\em Ergodic Theory Dynam. Systems}, 10(1):151--176, 1990.

\bibitem[Mir07]{MirzWP}
M.~Mirzakhani.
\newblock Weil-{P}etersson volumes and intersection theory on the moduli space
  of curves.
\newblock {\em J. Amer. Math. Soc.}, 20(1):1--23, 2007.

\bibitem[Mum71]{Mumfordtheta}
D.~Mumford.
\newblock Theta characteristics of an algebraic curve.
\newblock {\em Ann. Sci. \'Ecole Norm. Sup. (4)}, 4:181--192, 1971.

\bibitem[Mum83]{mumford83}
D.~Mumford.
\newblock Towards an enumerative geometry of the moduli space of curves.
\newblock In {\em Arithmetic and geometry, Vol. II}, volume~36 of {\em Progr.
  Math.}, pages 271--328. Birkh\"auser Boston, Boston, MA, 1983.

\bibitem[OP09]{OPWitten}
A.~Okounkov and R.~Pandharipande.
\newblock Gromov-{W}itten theory, {H}urwitz numbers, and matrix models.
\newblock In {\em Algebraic geometry---{S}eattle 2005. {P}art 1}, volume~80 of
  {\em Proc. Sympos. Pure Math.}, pages 325--414. Amer. Math. Soc., Providence,
  RI, 2009.

\bibitem[Sau17]{SauvagetClass}
A.~Sauvaget.
\newblock {Cohomology classes of strata of differentials}, 2017, arXiv:1701.07867.
\newblock Preprint.

\bibitem[Sau18]{SauvagetMinimal}
A.~Sauvaget.
\newblock Volumes and {S}iegel--{V}eech constants of {$\mathcal{H}$} (2{G} --
  2) and {H}odge integrals.
\newblock {\em Geom. Funct. Anal.}, 28(6):1756--1779, 2018.

\bibitem[Vee82]{veech82}
W.~Veech.
\newblock Gauss measures for transformations on the space of interval exchange
  maps.
\newblock {\em Ann. of Math. (2)}, 115(1):201--242, 1982.

\bibitem[Vee98]{veech98}
W.~Veech.
\newblock Siegel measures.
\newblock {\em Ann. of Math. (2)}, 148(3):895--944, 1998.

\bibitem[Vor97]{vorobets}
Ya.~B. Vorobets.
\newblock Ergodicity of billiards in polygons.
\newblock {\em Mat. Sb.}, 188(3):65--112, 1997.

\bibitem[Wol85]{wolpert85}
S.~Wolpert.
\newblock On obtaining a positive line bundle from the {W}eil-{P}etersson
  class.
\newblock {\em Amer. J. Math.}, 107(6):1485--1507 (1986), 1985.

\bibitem[Wri15]{WrightCyl}
A.~Wright.
\newblock Cylinder deformations in orbit closures of translation surfaces.
\newblock {\em Geom. Topol.}, 19(1):413--438, 2015.

\bibitem[Zag16]{zagBO}
D.~Zagier.
\newblock Partitions, quasimodular forms, and the {B}loch-{O}kounkov theorem.
\newblock {\em Ramanujan J.}, 41(1-3):345--368, 2016.

\end{thebibliography}

\end{document}